\documentclass[11pt,times,letter]{article}

\usepackage{amsmath, amsfonts,amsthm,amssymb,bm}
\usepackage{dsfont}
\usepackage{tabularx}
\usepackage[english]{babel}
\usepackage{float}

\usepackage[normalem]{ulem}

\usepackage{subfigure}
\usepackage{latexsym,amssymb,amsfonts,graphicx}
\usepackage{amsmath,dsfont}
\usepackage{verbatim}
\usepackage{mathrsfs}
\usepackage{bm}
\usepackage{color}
\usepackage{epsfig}
\usepackage{epstopdf}
\usepackage[title]{appendix}
\usepackage{url}
\usepackage{bm}
\usepackage{natbib}

\usepackage[colorlinks,linkcolor=blue,citecolor=blue,anchorcolor=blue]{hyperref}

\newcommand{\Px}{ \mathbb{P} }

\newcommand{\Ex}{ \mathbb{E} }

\def\esssup_#1{\underset{#1}{\mathrm{ess\,sup\, }}}
\def\essinf_#1{\underset{#1}{\mathrm{ess\,inf\, }}}
\def\argmax_#1{\underset{#1}{\mathrm{arg\,max\, }}}
\def\argmin_#1{\underset{#1}{\mathrm{arg\,min\, }}}

\newcommand{\Fx}{\mathbb{F} }

\newcommand{\R}{\mathds{R}}

\newtheorem{theorem}{Theorem}[section]

\numberwithin{equation}{section}

\newtheorem{proposition}[theorem]{Proposition}
\newtheorem{remark}[theorem]{Remark}
\newtheorem{lemma}[theorem]{Lemma}
\newtheorem{corollary}[theorem]{Corollary}

\allowdisplaybreaks[4]

\definecolor{Red}{rgb}{1.00, 0.00, 0.00}

\definecolor{DRed}{rgb}{0.5, 0.00, 0.00}

\definecolor{Blue}{rgb}{0.00, 0.00, 1.00}

\definecolor{Green}{rgb}{0.0, 0.4, 0.0}

\title{An Extended Merton Problem with Relaxed Benchmark Tracking}

\author{Lijun Bo \thanks{Email: lijunbo@ustc.edu.cn, School of Mathematics and Statistics, Xidian University, Xi'an, 710126, China.}
\and
Yijie Huang \thanks{Email: yijie.huang@polyu.edu.hk, Department of Applied Mathematics, The Hong Kong Polytechnic University, Kowloon, Hong Kong, China.}
\and
Xiang Yu \thanks{Email: xiang.yu@polyu.edu.hk, Department of Applied Mathematics, The Hong Kong Polytechnic University, Kowloon, Hong Kong, China.}
}
\date{\vspace{-0.8cm}}

\begin{document}
\maketitle

\begin{abstract}
This paper studies Merton's problem in an extended formulation by incorporating the benchmark tracking on the wealth process. We consider a tracking formulation where the fund manager aims to maximize the trade-off between the expected utility of consumption and the expected largest shortfall of the wealth with reference to the benchmark level. Equivalently, the problem can be interpreted as a mixed stochastic control problem if a fictitious capital injection singular control is allowed, subjecting to the dynamic constraint that the wealth process compensated by the costly capital injection outperforms the benchmark at all times. By considering an auxiliary state process, we formulate an equivalent stochastic control problem with state reflections at zero. For general utility functions and It\^o's diffusion benchmark process, we develop a convex duality theorem, new to the literature, to the auxiliary stochastic control problem with state reflections in which the dual process also exhibits reflections from above. For CRRA utility and geometric Brownian motion benchmark process, we further derive the optimal portfolio and consumption in feedback form using the new duality theorem, allowing us to discuss some interesting financial implications induced by the additional risk-taking from the capital injection and the goal of tracking.

\vspace{0.1in}

\noindent\textbf{Keywords}: Benchmark tracking, expected largest shortfall, duality theorem, reflected diffusion processes, consumption and portfolio choice, Neumann boundary condition.
\end{abstract}

\section{Introduction}

Since the pioneer studies of \cite{Merton69,Merton1971}, the optimal portfolio and consumption problem via utility maximization has attracted tremendous generalizations to address various new demands and challenges coming from more realistic market models, performance measures, trading constraints, and among others. In the present paper, we consider a new type of  Merton problem in the context of fund management when the fund manager is also concerned about the relative performance with respect to a stochastic benchmark. It has been well documented that, in order to enhance client's confidence or to attract more new clients, tracking or outperforming a benchmark has become common in practice in fund management. Some typical benchmark processes are S$\&$P500 Index, Goldman Sachs Commodity Index, Hang Seng Index. Other popular examples of benchmark processes in fund management may refer to inflation rates, exchange rates, liability or education cost indices, etc.

Portfolio management with benchmark tracking has been an important research topic in quantitative finance, which is often used to assess the performance of fund management and may directly affect the fund manager's performance-based incentives. Various tracking formulations have been considered in the literature. For example, in \cite{Browne9}, \cite{Browne99}, and \cite{Browne00}, several active portfolio management objectives are studied including: maximizing the probability that the agent's wealth achieves a performance goal relative to the benchmark before falling below it to a predetermined shortfall; minimizing the expected time to reach the performance goal; the mixture of these two objectives and some further extensions. Another conventional way to optimize the tracking error is to minimize the variance or downside variance relative to the index value or return, see \cite{Gaivoronski05}, \cite{YaoZZ06} and \cite{NLFC}, which leads to some standard linear-quadratic stochastic control problems. In \cite{Strub18}, another objective function is introduced to measure the similarity between the normalized historical trajectories of the tracking portfolio and the index where the rebalancing transaction costs can also be taken into account. Notably, all aforementioned studies only focus on the task of portfolio tracking, and it is not straightforward to accommodate the optimal consumption within their problem formulations as the methodology developed in previous studies may fail if one also needs to maximize the expected utility over consumption in the objective function, for instance, the stochastic control problem is no longer of linear-quadratic type.

Recently, to address the optimal tracking of the non-decreasing and absolutely continuous benchmark process, a new tracking formulation using the fictitious capital injection is studied in \cite{BoLiaoYu21}, in which the fund manager can fictitiously inject some capital into the fund account such that the total capital outperforms the targeted benchmark process at all times.  The capital injection, also called the bail-out strategy, has been extensively studied in the setting of the De Finetti's optimal dividend problems, see e.g. \cite{Lokka08}, \cite{Eisenberg09}, \cite{Eisenberg11} and \cite{FS19}, in which the beneficiary of the dividends may strategically inject capital in order to avoid bankruptcy of firms. In the present paper, we borrow the idea of capital injection and employ it to measure and control the size of difference between the wealth process and the benchmark level. That is, the goal of benchmark tracking is to enforce the wealth process compensated by the capital injection to satisfy the dynamic floor constraint above the benchmark process. Later, it is revealed in \cite{BHY24} that this tracking formulation is very suitable to incorporate the consumption planning, in which the fund manager seeks a high expected utility on consumption and also steers the wealth process as close as possible to the targeted benchmark. Some treatment of the HJB equation with two Neumann boundary conditions are carefully developed when the benchmark process involves a running maximum term. However, as a price of the monotone benchmark process, the quantitative properties of the optimal portfolio and consumption strategies cannot be concluded in \cite{BoLiaoYu21} and \cite{BHY24}, leaving the financial implications therein inadequate.

In the present paper, our goal is to refine the quantitative analysis of the trade-off between the optimal consumption via utility maximization and the minimization of the tracking error via the cost minimization of total capital injection (or equivalently the minimization of the expected largest shortfall risk). We adopt the same tracking formulation in \cite{BoLiaoYu21,BHY24} to examine the consumption planning. Thanks to the explicit characterization of the optimal singular control of capital injection, the total capital injection in fact records the largest shortfall when the wealth falls below the benchmark process. That is, even when the capital injection may not be feasible in the practical fund management, the quantity of total capital injection in \eqref{eq:w} can still be interpreted as the risk measure of the expected largest shortfall of the wealth process with respect to the benchmark level. Therefore, even without capital injection, our equivalent unconstrained problem \eqref{eq_orig_pb} is still well-defined when the wealth process is allowed to fall below the benchmark, which can be viewed as an extended Merton problem amended by the minimization of the benchmark-moderated shortfall risk, see Remark \ref{remA} for more details. By introducing the auxiliary reflected state process $X=(X_t)_{t\geq0}$ in \eqref{state-X}, we end up to solve the equivalent stochastic control problem \eqref{eq:u0} over the portfolio and consumption strategies. We stress that the methodology in the present paper differs substantially from the HJB equation approach in \cite{BoLiaoYu21,BHY24}. Indeed, the technical study on the existence of the classical solution to the HJB equation with Neumann boundary conditions and the proof of the verification theorem using some estimations of the optimal feedback controls play the core roles in \cite{BoLiaoYu21,BHY24}, which heavily rely on the specific choice of CRRA utility. In sharp contrast, we develop the convex duality theorem and the general verification arguments on optimal control in the present paper, which can powerfully cope with general utility functions with some modest growth conditions. Moreover, the new duality approach also allows us to work with a class of general Ito diffusion processes to model the stochastic benchmark. 

On the technical front, to handle the reflection of the primal state process in the model, we modify the standard dual process in \cite{Karatzas1987} and \cite{Cox1989}, namely the state price density process in the Black-Scholes model, by introducing the reflection of the dual process from above at a constant barrier; see its definition in \eqref{eq:Yt}. To further define the appropriate dual problem and close the duality gap, we first provide a new characterization of admissible portfolio and consumption controls using the dual process in a fashion of the budget constraint, however, involving the reflection local times from both primal and dual processes; see Lemma \ref{lem:budget-constraint}. Interestingly, we can verify that, under a  constructed candidate optimal strategy, (i) the budget constraint becomes an equality constraint, and (ii) the reflected primal state process attains its lower reflection boundary if and only if the reflected dual process hits its upper reflection boundary. Accordingly, we can define the dual value function in the form of \eqref{eq:dual-func}, which can be shown to satisfy a linear PDE \eqref{eq:HJB-dual-hatv} with the Neumann boundary condition. We then rigorously prove the convex duality theorem (see Theorem \ref{thm:duality}) between the value function $v(x,z)$ for the auxiliary control problem \eqref{eq:u0} and the dual value function \eqref{eq:dual-func} in the context with reflections by verifying the transversality condition and the optimality of the constructed optimal controls in terms of the dual process satisfying the equality budget constraint. Using the inverse transform, we can readily express the optimal portfolio-consumption strategy in terms of the primal value function (see Corollary~\ref{coro:optimal-control}).

Thanks to the newly established duality theorem, in the special case of CRRA utility and the GBM benchmark process, we can further characterize the optimal portfolio and consumption control in semi-analytical feedback form in terms of the primal state variables, see Corollary \ref{coro:optimal}. Consequently, we are capable to discuss and numerically illustrate some new and notable financial implications of the optimal portfolio and consumption strategies induced by the trade-off between the utility maximization and the goal of tracking. Firstly, we highlight that the feedback optimal portfolio and optimal consumption strategies exhibit convex (concave) property with respect to the wealth level when the fund manager is more (less) risk averse (see Proposition \ref{lem:prop}), which differs significantly from the classical Merton's solution. Some numerical examples and discussions on their financial implications are given in Section \ref{sec:num}. Secondly, due to the extra risk-taking from the capital injection, the optimal portfolio and consumption behavior will become more aggressive than their counterparts in the Merton problem, and it is interesting to see from our main result that the optimal portfolio and consumption are both positive at the instant when the capital is injected (see Remark \ref{rem0}). Even when there is no benchmark, the allowance of fictitious capital injection already enlarges the space of admissible controls as the no-bankruptcy constraint is relaxed and our credit line of wealth is captured by the total capital injection or the largest shortfall risk.  A detailed comparison result with the Merton solution are summarized in Remark \ref{compareM}, and a brief description of financial implications on the adjustment impact by the capital injection are provided in Section \ref{sec:num} through numerical examples. Thirdly, it is revealed that our extended formulation of Merton's problem by minimizing the proposed tracking error (the cost of expected total capital injection or the expected largest shortfall) naturally induces the lowest subsistence level on the optimal consumption as shown in Figure \ref{fig:simulation}. Notably, the subsistent consumption constraint has been extensively studied in the literature to partially explain various empirical observations, which often requires the minimum level of initial wealth to guarantee the well-posedness of the problem. Interestingly, as the minimum consumption constraint is not imposed in our admissible control space, our problem formulation is flexible with all wealth levels and the subsistent consumption level results from the reflection of the auxiliary state process. Therefore, our solution naturally indicates that the optimal consumption will stay above a positive constant level, which echos with some previous financial insights by imposing the subsistent consumption constraint in the literature. Fourthly, we observe that the optimal control in our formulation is not necessarily monotone in the risk aversion parameter, which is consistent with some existing empirical studies, see Figure \ref{fig:optimal-p} and simulations in Figure \ref{fig:optimal-x-1}. We thereby provide a different perspective to explain some observed diverse dependence on the risk aversion parameter, see some detailed discussions of financial insights therein. Finally, to verify that our problem formulation is well defined and financially sound, it is shown in Lemma \ref{lem:inject-captial} and Remark \ref{remA} that the expected discounted total capital injection is both bounded below and above. The lower bound indicates the necessity of capital injection to meet the dynamic benchmark floor constraint, and the upper bound implies the finite risk measured by the expected largest shortfall when wealth process falls below the benchmark.

The remainder of this paper is organized as follows. In Section \ref{sec:model}, we introduce the market model and the relaxed benchmark tracking formulation using the fictitious capital injection. In Section \ref{sec:aux}, by introducing an auxiliary state process with reflection, we formulate an equivalent stochastic control problem and establish the convex duality theorem for general utility functions and Ito benchmark processes by introducing the dual reflected process and the dual value function. Under the CRRA utility and GBM benchmark process,  the verification theorem on optimal feedback control in semi-analytical form is presented in Section \ref{sec:CRRA} together with some quantitative properties of the optimal controls. Some numerical examples and financial implications are presented in Section \ref{sec:num}. Section~\ref{sec:proof} contains proofs of all results in previous sections.

\section{Market Model and Problem Formulation}\label{sec:model}

Consider a financial market model consisting of $d$ risky assets under a filtered probability space $(\Omega, \mathcal{F}, \Fx,\mathbb{P})$ with the filtration $\mathbb{F}=(\mathcal{F}_t)_{t\geq 0}$ satisfying the usual conditions. We introduce the following Black-Scholes model where the price process vector $S=(S^1,\ldots,S^d)^{\top}=(S_t^1,\ldots,S_t^d)_{t\geq0}^{\top}$ of $d$ risky assets is described by
\begin{align}\label{stockSDE}
dS_t = {\rm diag}(S_t)(\mu dt+\sigma dW_t),\quad S_0\in(0,\infty)^d,\ \ t\geq 0.
\end{align}
Here, ${\rm diag}(S_t)$ denotes the diagonal matrix with the diagonal elements given by $S_t$, and $W=(W^1,\ldots,W^d)^{\top}=(W_t^1,\ldots,W_t^d)_{t\geq 0}^{\top}$ is a $d$-dimensional $\Fx$-adapted Brownian motion, and $\mu=(\mu_1,\ldots,\mu_d)^{\top}\in\R^d$ denotes the vector of return rate and $\sigma=(\sigma_{ij})_{d\times d}$ is the volatility matrix which is invertible. It is assumed that the riskless interest rate $r=0$, which amounts to the change of num\'{e}raire. From this point onwards, all processes including the wealth process and the benchmark process are defined after the change of num\'{e}raire.

At time $t\geq 0$, let $\theta_t^i$ be the amount of wealth that the fund manager allocates in asset $S^i=(S_t^i)_{t\geq 0}$, and let $c_t$ be the non-negative consumption rate. The self-financing wealth process under the control $\theta=(\theta_t^1,\ldots,\theta_t^d)_{t\geq 0}^{\top}$ and the control $c=(c_t)_{t\geq 0}$ is given by
\begin{align}\label{eq:wealth2}
dV^{\theta,c}_t &=\theta_t^{\top}\mu dt+ \theta_t^{\top}\sigma dW_t-c_tdt,\quad t\geq 0,
\end{align}
where $V_0^{\theta,c}=\textrm{v}\geq 0$ denotes the initial wealth of the fund manager.

In the present paper, it is considered that the fund manager has concerns on the relative performance with respect to an external benchmark process.  The benchmark process $Z=\left(Z_t\right)_{t \geq 0}$ is described by the It\^o's diffusion process:
\begin{align}\label{eq:Zt}
d Z_t=\mu_Z (Z_t) d t+\sigma_Z (Z_t) d W_t^{\gamma}, \quad Z_0=z\in\R,
\end{align}
where the Brownian motion $W_t^{\gamma}:=\gamma^{\top}W_t$ for $t\geq0$ and $\gamma=(\gamma_1,\ldots,\gamma_d)^{\top}\in\R^d$ satisfying $|\gamma|=1$, i.e., the Brownian motion $W^{\gamma}=(W_t^{\gamma})_{t\geq0}$ is a linear combination of $W$ with weights $\gamma$. We impose the following conditions on coefficients $\mu_Z(\cdot)$ and $\sigma_Z(\cdot)$ that
\begin{itemize}
\item[$(\bm{A_{Z}})$] the coefficients $\mu_Z:\R\to \R$ and $\sigma_Z:\R\to\R$ belong to $C^2(\R)$ with bounded derivatives of various orders and satisfy  $\mu_Z(z)-\sigma_Z(z)\gamma^{\top}\sigma^{-1}\mu\geq 0$ for all $z\in\R$.
\end{itemize}
Note that both the Ornstein–Uhlenbeck process and GBM  satisfy the above assumption.

We consider the relaxed benchmark tracking formulation based on the expected largest shortfall by allowing the wealth process $V_t^{\theta,c}$ to fall below the benchmark $Z_t$. Precisely, given the benchmark process $Z=(Z_t)_{t\geq0}$, denote by $A^{\theta,c}_t$ the largest shortfall (up to time $t$) when the wealth falls below the benchmark, that is  $A^{\theta,c}_t =\sup_{s\leq t}(V_{s}^{\theta,c}-Z_s)^{-}$ for $t\geq 0$. Here for $x\in\R$, $x^-:=\max\{-x,0\}$. Then,  $\Ex[\int_0^{\infty} e^{-\rho t}d\sup_{s\leq t}(V_{s}^{\theta,c}-Z_s)^{-}]$ defines a risk measure on the expected largest shortfall with reference to the benchmark (see also the conventional definition of expected shortfall with respect to a random variable at the terminal time in \cite{Pham02} and the intra-horizon expected shortfall considered in \cite{Farkas21} and references therein). The fund manager needs to strategically choose the portfolio-consumption control to optimize the tradeoff between the expected utility of consumption and the expected largest shortfall.  As an extended Merton's  problem, the agent solves the following stochastic control problem, for all $(\mathrm{v},z)\in\R_+\times \R$ with $\R_+:=[0,\infty)$,
\begin{align}\label{eq_orig_pb}
{\rm w}(\mathrm{v}, z) &=\sup_{(\theta,c)\in\mathbb{U}}\ \Ex\left[ \int_0^{\infty} e^{-\rho t} U(c_t)dt-\beta\int_0^{\infty} e^{-\rho t} dA_t^{\theta,c}\right],
\end{align}
where the admissible control set $\mathbb{U}$ is defined as:
\begin{align*}
&\mathbb{U}:=\big\{\text{$(\theta,c)=(\theta_t,c_t)_{t\geq0}$; $(\theta,c)$ is an $\Fx$-adapted process taking values on $\R^d\times\R_+$}\big\}.
\end{align*}
The constant $\rho>0$ is the subjective discount rate, and the parameter $\beta>0$ describes the relative importance between the consumption performance and the expected largest shortfall.  Here, the general utility function $U(\cdot):\R_+\to \R$ is strictly increasing, strictly concave and second-order continuously differentiable. Moreover, $U'(\cdot)$ is positive with $U'(\infty):=\lim_{x\to\infty}U'(x)=0$.

Note that we can also equivalently interpret the problem \eqref{eq_orig_pb} as a mixed stochastic control problem with a fictitious strategic capital injection if it is allowed. Similar to \cite{BoLiaoYu21}, we may assume that the fund manager can strategically choose the dynamic portfolio and consumption as well as the fictitious capital injection such that the total capital outperforms the benchmark process at all times. That is, the fund manager optimally chooses the regular control $\theta=(\theta_t)_{t\geq0}$ as the dynamic portfolio in risky assets, the regular control $c=(c_t)_{t\geq0}$ as the consumption rate and the singular control $A=(A_t)_{t\geq0}$ as the cumulative capital injection such that $A_t+V_t^{\theta,c}\geq Z_t$ at any time $t\geq0$. In this case, the agent solves the following problem, for all $(\mathrm{v},z)\in\R_+\times \R$,
\begin{align}\label{eq:w}
\begin{cases}
\displaystyle \tilde{{\rm w}}(\mathrm{v},z):=\sup_{(\theta,c,A)\in\mathbb{U}} \Ex\left[\int_0^{\infty} e^{-\rho t} U(c_t)dt- \beta\left(A_0+\int_0^{\infty} e^{-\rho t}dA_t\right) \right],\\[1.4em]
\displaystyle \text{subject to}~Z_t \le A_t + V^{\theta,c}_t~\text{for all}~ t\geq 0,
\end{cases}
\end{align}
where the admissible control set $\tilde{\mathbb{U}}$ is defined as:
\begin{align*}
&\tilde{\mathbb{U}}:=\big\{\text{$(\theta,c,A)=(\theta_t,c_t,A_t)_{t\geq0}$; $(\theta,c)$ is an $\Fx$-adapted process taking values on $\R^d\times\R_+$, $A$ is a}\\
&\qquad\quad\text{nonnegative, non-decreasing process with r.c.l.l. paths and initial  value $A_0=a\geq 0$}\big\}.
\end{align*}
It follows from Lemma 2.4 in \cite{BoLiaoYu21} that,  for fixed regular control $(\theta,c)$, the optimal singular control $A^{(\theta,c),*}=(A^{(\theta,c),*}_t)_{t\geq0}$ satisfies $A^{\theta,c}_t =0\vee \sup_{s\leq t}(Z_{s}-V_{s}^{\theta,c})=\sup_{s\leq t}(V_{s}^{\theta,c}-Z_s)^{-},~\forall t\geq 0$.
Thus, problem \eqref{eq:w} admits the equivalent formulation as an unconstrained control problem with a running maximum cost that
\begin{align}\label{eq_orig_pb-w}
\tilde{{\rm w}}(\mathrm{v}, z) &=-\beta(z-\mathrm{v})^+\nonumber\\
&\quad+\sup_{(\theta,c)\in\mathbb{U}}\ \Ex\left[ \int_0^{\infty} e^{-\rho t} U(c_t)dt-\beta\int_0^{\infty} e^{-\rho t} d\left(0\vee \sup_{s\leq t}(Z_{s}-V_{s}^{\theta,c})\right)\right],\nonumber\\
&=-\beta(z-\mathrm{v})^++{\rm w}(\mathrm{v}, z),
\end{align}
which shows the equivalence between problems \eqref{eq_orig_pb} and \eqref{eq:w}. Here for $x\in\R$, $x^+:=\max\{x,0\}$.

We stress that when $\mu_Z(\cdot)=\sigma_Z(\cdot)\equiv 0$ and $z=0$ (constant zero benchmark) and the relative importance parameter $\beta\to+\infty$, the problem \eqref{eq:w} degenerates to the classical Merton's problem in \cite{Merton1971} as the capital injection is prohibited. However, when the benchmark process keeps constant zero that $Z_t\equiv 0$, but the capital injection is allowed with finite parameter $\beta\in(0,\infty)$, our problem formulation in \eqref{eq:w} actually motivates the fund manager to strategically inject capital from time to time to achieve the more aggressive portfolio and consumption behavior. That is, for the optimal solution $(\theta^*,c^*)$ in the Merton's problem, the control triplet $(\theta^*_t, c^*_t, A_t\equiv 0)$ without capital injection does not attain the optimality in \eqref{eq:w} even when $Z_t\equiv 0$.  Hence, the allowance of capital injection significantly affects the optimal decision making. These interesting observations are rigorously verified later in items (i) and (ii) in Remark \ref{compareM}.

Stochastic control problems with minimum floor constraints have been studied in different contexts, see among \cite{Karoui05}, \cite{KM06}, \cite{Giacinto11}, \cite{Sekine12}, and \cite{CYZ20} and references therein. In previous studies, the minimum guaranteed level is usually chosen as constant or deterministic level and some typical techniques to handle the floor constraints are to introduce the option based portfolio or the insured portfolio allocation. 
When there exists a stochastic benchmark $Z_t$, it is actually observed in this paper that one can not find any admissible control $(\theta_t, c_t, A_t\equiv 0)$ such that the constraint $Z_t\leq V_t^{\theta, c}$ is satisfied, i.e., the classical Merton's problem under the benchmark constraint $Z_t\leq V_t^{\theta, c}$ is actually not well defined. To dynamically outperform the stochastic benchmark, the capital injection is crucially needed and our problem formulation in \eqref{eq:w} is a reasonable and tractable one. The detailed elaboration of this observation is given in item (iii) of Remark \ref{compareM}.

We also note that some previous studies on stochastic control problems with the running maximum cost can be found in \cite{BaIshii89}, \cite{BDR1994}, \cite{BPZ15}, \cite{Weerasinghe16} and \cite{Kroner18}, where the viscosity solution approach usually plays the key role. We will adopt the equivalent stochastic control problem with a reflected state process as proposed in \cite{BoLiaoYu21} and rigorously develop a non-standard convex duality theorem, new to the literature, in the context of state reflections.


\section{Equivalent Control Problem and Duality Theorem}\label{sec:aux}

In this section, we formulate and study a more tractable equivalent stochastic control problem, which is mathematically equivalent to the unconstrained optimal control problem \eqref{eq_orig_pb}.

To formulate the equivalent stochastic control problem, we will introduce a new controlled state process to replace the process $V^{\theta,c}=(V_t^{\theta,c})_{t\geq 0}$ given in \eqref{eq:wealth2}. To this end, let us first define the difference process by
$D_t:=Z_t-V_t^{\theta,c}+\mathrm{v}-z,~\forall t\geq 0$
with $D_0=0$.  For any $x\in\R_+$, introduce the running maximum process of $D=(D_t)_{t\geq0}$ given by $L_t:=x\vee \sup_{s\leq t}D_s-x\geq 0$ for $t\geq0$, and $L_0=0$. We then propose a new controlled state process $X=(X_t)_{t\geq 0}$ taking values on $\R_+$, which is defined as the reflected process $X_t:=x+L_t-D_t$ for $t\geq 0$ that satisfies the following SDE with reflection:
\begin{align}\label{state-X}
X_t =x+\int_0^t\theta_s^{\top}\mu ds+\int_0^t\theta_s^{\top}\sigma dW_s  -\int_0^t c_s ds-\int_0^t \mu_Z (Z_s)ds-\int_0^t \sigma_Z (Z_s)dW_s^{\gamma}+ L_t
\end{align}
with the initial value $X_0=x\in\R_+$. For the notational convenience, we have omitted the dependence of $X=(X_t)_{t\geq0}$ on the control $(\theta,c)$. In particular, the process $L=(L_t)_{t\geq0}$ is referred to as the local time of the state process $X$, which increases at time $t$ if and only if $X_t=0$, i.e., $x+L_t=D_t$. We will change the notation from $L_t$ to $L_t^X$ from this point onwards to emphasize its dependence on the new state process $X$ given in \eqref{state-X}.

With the above preparations, let us consider the following stochastic control problem given by, for $(x,z)\in\R_+\times \R$,
\begin{align}\label{eq:u0}
\begin{cases}
\displaystyle v(x,z):=\sup_{(\theta,c)\in\mathbb{U}^{\rm r}}J(x,z;\theta,c)\\[0.6em]
\displaystyle\qquad\quad:=\sup_{(\theta,c)\in\mathbb{U}^{\rm r}}\Ex\left[\int_0^\infty e^{-\rho t} U(c_t)dt- \beta \int_0^{\infty} e^{-\rho t}dL_t^X\Big|X_0=x,Z_0=z \right],\\[1em]
\displaystyle \text{subject~to~the state process}~(X,Z)~\text{satisfies the dynamics~\eqref{state-X}~and~\eqref{eq:Zt}}.
\end{cases}
\end{align}
Here, the admissible control set $\mathbb{U}^{\rm r}$ is specified as the set of $\Fx$-adapted control processes $(\theta,c)=(\theta_t,c_t)_{t\geq0}$ such that the reflected SDE \eqref{state-X} has a unique strong solution.  It is not difficult to observe the equivalence between \eqref{eq_orig_pb} and \eqref{eq:u0} in the following sense:
\begin{lemma}\label{lem:equivalence}
For value functions $\mathrm{w}$ defined in \eqref{eq_orig_pb} and  $v$ defined in \eqref{eq:u0}, we have ${\rm w}(\mathrm{v},z)=v((\mathrm{v}-z)^+,z)$ for all $(\mathrm{v},z)\in\R_+\times \R$. Moreover, the value function $x\to v(x,z)$ defined in \eqref{eq:u0} is non-decreasing and $\left|v(x_1,z)-v(x_2,z)\right|\leq \beta |x_1-x_2|$ for all $(x_1,x_2,z)\in\R_+^2\times \R$.
\end{lemma}

Next, we will solve the stochastic control problem \eqref{eq:u0} by developing a convex duality approach. Assume that the value function $v(x,z)$ is smooth enough, it follows from Lemma \ref{lem:equivalence} that $0\leq v_x(x,z)\leq \beta$ for all $(x,z)\in\R_+\times \R$. Based on this observation and the fact of the state reflection of $X=(X_t)_{t\geq0}$ at boundary $0$, we introduce the following reflected dual process $Y=(Y_t)_{t\geq 0}$ taking values on $(0,\beta]$ in the form:
\begin{align}\label{eq:Yt}
dY_t=\rho Y_tdt-\mu^{\top}\sigma^{-1}Y_tdW_t-dL_t^Y,~~t>0,\quad Y_0=y\in(0,\beta],
\end{align}
where the process $L^Y=(L_t^Y)_{t\geq0}$ is a continuous and non-decreasing process (with $L_t^{Y}=0$) which increases on the time set $\{t\geq 0;~Y_t=\beta\}$ only. That is, we adopt and modify the state price density process in the Black-Scholes model by adding the upper reflection of $Y$ at boundary $\beta$.

Evidently, as both primal state process $X$ and the dual process $Y$ exhibit reflections, the conventional budget constraint for admissible consumption no longer holds. We need to carefully revise the characterization of admissible controls using the above reflected dual process, which is stated in the next important result. 

\begin{lemma}[Characterization of admissible portfolio and consumption]\label{lem:budget-constraint}
For any strategy pair $(\theta,c)\in\mathbb{U}^r$, consider the corresponding controlled state process $(X,Z)=(X_t,Z_t)_{t\geq 0}$ given by \eqref{state-X} and \eqref{eq:Zt} with initial value $(X_0,Z_0)=(x,z)\in\R_+\times \R$, and the reflected dual process $Y=(Y_t)_{t\geq 0}$ given by \eqref{eq:Yt} with $Y_0=y\in(0,\beta]$. Then, we have
\begin{align}\label{eq:budget}
\Ex\left[\int_0^{\infty}e^{-\rho t}\left(c_tY_t+F(Z_t)Y_t\right)dt+\int_0^{\infty}e^{-\rho t}X_tdL_t^Y-\int_0^{\infty}e^{-\rho t}Y_tdL_t^X\right]\leq xy,
\end{align}
where the function $F:\R\to \R$ is defined by $F(z):=\mu_Z(z)-\sigma_Z(z)\gamma^{\top}\sigma^{-1}\mu\geq 0$ for $z\in\R$.
\end{lemma}

\begin{remark}\label{rem:budget}
Clearly, compared to the conventional budget constraint for admissible consumption in Merton's problem, our characterization in \eqref{eq:budget} introduces additional complexity. Specifically, it not only incorporates the term arising from the benchmark process $Z = (Z_t)_{t \geq 0}$ but also includes reflection terms associated with both the primal process $X=(X_t)_{t\geq0}$ and the dual process $Y=(Y_t)_{t\geq0}$. Only when the benchmark process $Z \equiv 0$ (constant zero benchmark) and the relative importance parameter $\beta \to +\infty$ (capital injection is not allowed), the local time process $L_t^Y =0$ for all $t\geq 0$ (no reflection), and hence the inequality \eqref{eq:budget} reduces to the classical budget constraint as in \cite{Karatzas1987}:
\begin{align*}
\mathbb{E}\left[\int_0^{\infty} e^{-\rho t} c_t Y_t dt\right] \leq xy.
\end{align*}
Despite the apparent complexity of the inequality constraint in \eqref{eq:budget}, we conjecture that the optimal portfolio/consumption controls $(\theta^*,c^*)$ can be expressed in terms of the dual process $Y=(Y_t)_{t\geq0}$. Furthermore, under the optimal control, the optimal primal process $X^*=(X^*_t)_{t\geq0}$ hits the boundary $0$ if and only if the dual process $Y=(Y_t)_{t\geq0}$ hits the boundary $\beta$, and the inequality constraint tightens to the equality constraint:
\begin{align}\label{eq:budgetopt}
\Ex\left[\int_0^{\infty}e^{-\rho t}\left(c^*_tY_t+F(Z_t)Y_t\right)dt-\beta \int_0^{\infty}e^{-\rho t}dL_t^{X^{*}}\right]=xy.
\end{align}
Thus, we have, for any $(x,z)\in\R_+^2$,
\begin{align}
v(x,z)&=\Ex\left[\int_0^\infty e^{-\rho t} U(c_t^*)dt- \beta \int_0^{\infty} e^{-\rho t}dL_t^{X^*} \right]+xy-xy\nonumber\\
&=\Ex\left[\int_0^\infty e^{-\rho t} U(c_t^*)dt- \beta \int_0^{\infty} e^{-\rho t}dL_t^{X^*} \right]+xy\nonumber\\
&\quad-\Ex\left[\int_0^{\infty}e^{-\rho t}\left(c^*_tY_t+F(Z_t)Y_t\right)dt-\beta \int_0^{\infty}e^{-\rho t}dL_t^{X^{*}}\right]\nonumber\\
&=\Ex\left[\int_0^\infty e^{-\rho t} \big((U(c_t^*)dt- c^*_tY_t)-F(Z_t)Y_t\big)dt\right]+xy,
\end{align}
which motivates us to consider the following dual problem.
\end{remark}

Introduce the dual problem and define the corresponding dual value function $\hat{v}(y,z):(0,\beta]\times\R\to\R$ by
\begin{align}\label{eq:dual-func}
\hat{v}(y,z):=\Ex\left[\int_0^\infty e^{-\rho t} \left(\Phi(Y_t)-F(Z_t)Y_t\right)dt\right],
\end{align}
where  the function $\Phi(\cdot)$ is the Legendre-Fenchel (LF) transform of the utility function $U(\cdot)$, i.e., 
\begin{align}\label{eq:dualtyPhiU}
\Phi(x):=\sup_{c\in\R_+}\left(U(c)-cx\right),\quad \forall x>0.
\end{align}
To close the duality gap between the value function $v(x,z)$ and the dual function $\hat{v}(y,z)$, we impose the following assumptions on the general utility function $U(\cdot)$ and model parameters:
\begin{itemize}
\item[$(\bm{A_{o}})$] The utility function $U(\cdot)\in C^3(\R_+)$ is strictly increasing and strictly concave,  satisfying $U'(\cdot)>0$ and  $U'(\infty)=\lim_{x\to\infty}U'(x)=0$. Moreover, there exists a constant $p<1$ such that
\begin{align*}
\left|\Phi(x)\right|+\left|x\Phi'(x)\right|+\left|x^2\Phi''(x)\right|+\left|x^3\Phi'''(x)\right|\leq C\left(1+x^{\frac{p}{p-1}}\right),\quad \forall x\in(0,\beta],
\end{align*} 
where $C>0$ is a constant independent of $x$.

\item[$(\bm{A_{\rho}})$]The discount factor $\rho>\rho_0$, where $\rho_0>0$ is a constant depending on $p<1$ in $(\bm{A_{o}})$ defined by
\begin{align}\label{boundrho0}
    \rho_0:=
    \begin{cases}
   \max\{\sup_{z\in\R}\mu_Z'(z),\alpha p/(1-p)\},&\text{if}~p\in (0,1),\\[0.5em]
  \max\{\sup_{z\in\R}\mu_Z'(z),0\},&\text{if}~p\leq 0
    \end{cases}
    \end{align}
    with $\alpha:=\mu^{\top}(\sigma\sigma^{\top})^{-1}\mu/2$.
\end{itemize}
 The classical power utility function $U(x)=\frac{1}{p}x^p$ for $x>0$ and $p<1$ (when $p=0$, it is understood that $U(x)=\ln x$ for $x>0$) satisfies the above assumption $(\bm{A_{o}})$.

The following convex duality theorem is the main result of this paper, which shows the strong dual relationship between the value function $v(x,z)$ and the dual function $\hat{v}(y,z)$, and also provides the optimal control for the problem \eqref{eq:u0} in terms of the dual process.

\begin{theorem}[Duality Theorem]\label{thm:duality}
Let Assumptions $(\bm{A_{Z}})$, $(\bm{A_{o}})$ and $(\bm{A_{\rho}})$ hold. Then, for $(x,z)\in\R_+\times \R$,
\begin{align}\label{eq:duality-v}
v(x,z)=\inf_{y\in(0,\beta]}\left(\hat{v}(y,z)+xy\right).
\end{align}
Furthermore, let us introduce the control pair $(\theta^*,c^*)=(\theta^*_t,c^*_t)_{t\geq 0}$ given by
\begin{align}\label{eq:optimal-control-Yt}
\begin{cases}
\displaystyle \theta^*_t=(\sigma\sigma^{\top})^{-1}\left(\mu Y_t\hat{v}_{yy}(Y_t,Z_t)-\sigma_Z(Z_t)\sigma\gamma \hat{v}_{yz}(Y_t,Z_t)+\sigma_Z(Z_t)\sigma\gamma \right),\\[0.6em]
\displaystyle c^*_t=I(Y_t),\quad \forall t\geq 0.
\end{cases}
\end{align}
Here, $I(\cdot)$ is the inverse function of $U'(\cdot)$, the process $Y=(Y_t)_{t\geq0}$ is given by \eqref{eq:Yt} with $Y_0=y^*(x,z)$, and $y^*(x,z)$ is the function determined by $ -\hat{v}_y(y^*(x,z),z)=x$. Then, $(\theta^*,c^*)=(\theta_t^*,c_t^*)_{t\geq0}\in \mathbb{U}^r$  is an optimal control pair for problem \eqref{eq:u0}, that is
\begin{align}\label{eq:optimality-control}
v(x,z)=\sup_{(\theta,c)\in\mathbb{U}^{\rm r}}J(x,z;\theta,c)=J(x,z;\theta^*,c^*),\quad \forall (x,z)\in\R_+\times \R.
\end{align}
\end{theorem}

Using the inverse transform, the following corollary characterizes the optimal strategy in feedback form via the primal value function $v(x,z)$:
\begin{corollary}\label{coro:optimal-control}
Let Assumptions $(\bm{A_{Z}})$, $(\bm{A_{o}})$ and $(\bm{A_{\rho}})$ hold. Define the following optimal feedback control functions by, for any $(x,z)\in\R_+\times \R$,
\begin{align}\label{eq:feedbackfcn}
\begin{cases}
\displaystyle \theta^*(x,z):=-(\sigma\sigma^{\top})^{-1}\frac{\mu v_x(x,z)+\sigma_Z(z)\sigma\gamma (v_{xz}(x,z)-v_{xx}(x,z))}{v_{xx}(x,z)},\\[0.6em]
\displaystyle c^*(x,z):=I( v_{x}(x,z)).
\end{cases}
\end{align}
Consider the controlled state process $(X^*,Z)=(X_t^*,Z_t)_{t\geq 0}$ with feedback controls $\theta^*=(\theta^*(X^*_{t},Z_{t}))_{t\geq 0}$ and  $c^*=(\theta^*(X^*_{t},Z_{t}))_{t\geq 0}$. 
Then, the pair $(\theta^*,c^*)=(\theta_t^*,c_t^*)_{t\geq 0}\in\mathbb{U}^{\rm r}$ is an optimal control that, for all admissible  $(\theta,c)\in\mathbb{U}^{\rm r}$, we have
\begin{align*}
\Ex\left[\int_0^{\infty} e^{-\rho t} U(c_t)dt- \beta \int_0^{\infty} e^{-\rho t}dL_t^X \right]\leq v(x,z),\quad \forall (x,z)\in\R_+\times \R,
\end{align*}
where the equality holds when $(\theta,c)=(\theta^*,c^*)$.
\end{corollary}

The following lemma shows that the  expected largest shortfall (expectation of the total capital injection) is always finite and positive.
\begin{lemma}\label{lem:inject-captial}
Let Assumptions $(\bm{A_{Z}})$, $(\bm{A_{o}})$ and $(\bm{A_{\rho}})$  hold. Assume that the set ${\cal O}_Z=\{z\in\R;~F(z):=\mu_Z(z)-\sigma_Z(z)\gamma^{\top}\sigma^{-1}\mu> 0\}\neq\varnothing$. Then, we have
\begin{itemize}
\item[{\rm(i)}] The  expected largest shortfall (expectation of the total capital injection) under the optimal strategy $(\theta^*,c^*)$ provided in Corollary \ref{coro:optimal-control} is finite that there exists some function $P(x,z):\R_+\times \R\to \R_+$ such that
\begin{align}\label{eq:dAstarfinite}
\Ex\left[\int_0^{\infty} e^{-\rho t}dA^{\theta^*,c^*}_t \right]\leq P(x,z)<+\infty,\quad \forall (x,z)\in\R_+\times \R.
\end{align}
\item[{\rm(ii)}] The expected largest shortfall (expectation of the total capital injection) under any admissible $(\theta,c)\in\mathbb{U}^r$ is positive that there exists some function $L(x,z):\R_+\times {\cal O}_Z\to (0,+\infty)$ such that, for all $(\theta,c)\in\mathbb{U}^r$,
\begin{align}\label{eq:dAstarpositive}
\Ex\left[\int_0^{\infty} e^{-\rho t}dA^{\theta,c}_t \right]\geq L(x,z)>0,\quad \forall (x,z)\in\R_+\times {\cal O}_Z.
\end{align}
\end{itemize}
 Here,  the largest shortfall (optimal capital injection) under the admissible strategy $(\theta,c)$ is given by
$A_t^{\theta,c} =0\vee \sup_{s\leq t}(Z_{s}-V_{s}^{\theta,c})$ for $t\geq0$.
\end{lemma}

\begin{remark}\label{remA}
First, Lemma \ref{lem:inject-captial}-(i) provides an upper bound of the expected optimal capital injection, i.e. the expectation of the discounted total capital injection is always finite, which is an important fact to support that our problem formulation in \eqref{eq:w} is well defined as it excludes the possibility of requiring the injection of infinite capital. Recall that 
$A^{*}_t =\sup_{s\leq t}(V_{s}^{\theta^*,c^*}-Z_s)^{-}$ can also be understood as the largest shortfall of the wealth below the benchmark process, Lemma \ref{lem:inject-captial}-(i) shows that the expected largest shortfall is finite when there is no actual capital injection in the fund management.

On the other hand, Lemma \ref{lem:inject-captial}-(ii) provides a positive lower bound of the expected optimal capital injection, which implies that the capital injection is always necessary to meet the dynamic benchmark floor constraint $Z_t \le A^*_t + V^{\theta^*,c^*}_t$ for all $t\geq 0$. As the capital injection is needed, the admissible control space of the portfolio and consumption pair $(\theta,c)$ is enlarged from the admissible control space in \cite{Merton1971} under no-bankruptcy constraint. Indeed, due to the positive capital injection, we note that the controlled wealth process $V^{\theta^*, c^*}$ may become negative as the fund manager is more risk-taking. To further elaborate that our problem formulation is well defined in the sense that the wealth process $V^{\theta^*, c^*}$ will remain at a reasonable level, we can show that the expectation of the discounted wealth process is always bounded below by a constant. Indeed, we have $V^{\theta^*,c^*}_t\geq Z_t-A_t^*\geq -A_t^*,~ \forall t>0$.
Integration by parts yields that, for all $t>0$,
$\int_0^{t} e^{-\rho s}dA_s^{*} = e^{-\rho t}A_{t}^{*}-x + \rho\int_0^{t}  e^{-\rho s}A_s^*ds$.
It follows from  Lemma \ref{lem:inject-captial} that
\begin{align*}
\Ex\left[e^{-\rho t}A_{t}^{*}\right]\leq \Ex\left[\int_0^{t} e^{-\rho s}dA_s^{*}\right]+x\leq \Ex\left[\int_0^{\infty} e^{-\rho s}dA_s^{*}\right]+x<+\infty.
\end{align*}
This also implies that $\Ex[e^{-\rho t} V^{\theta^*,c^*}_t]>-\infty$ for any $t>0$ and $(x,z)\in\R_+^2$. In other words, the expectation of the discounted wealth has a finite credit line although $V^{\theta^*, c^*}_t$ may become negative.
\end{remark}

\begin{remark}
We focus on the infinite-horizon optimal tracking problem \eqref{eq_orig_pb} or \eqref{eq:w}, as it allows semi-explicit expressions for the value function and optimal feedback functions under CRRA utility (see Section \ref{sec:CRRA}). These closed-form results facilitate numerical examples and discussion on financial implications. While our duality approach can also be directly extended to the finite-horizon case. Moreover, in the finite-horizon setting, the condition $\rho > \rho_0$ on the discount rate in Assumption $(\bm{A_{o}})$ is no longer required.
\end{remark}

\section{The Case with CRRA Utility and GBM Benchmark}\label{sec:CRRA}

In this section, we focus on the CRRA utility function given by
\begin{align}\label{eq:Ui}
U(x)=\begin{cases}
\displaystyle \frac{1}{p}x^{p}, & p<1\ \text{and}\ p\neq 0,\\
\displaystyle \ln x, & p=0,
\end{cases}
\end{align}
 where the risk-aversion level of the fund manager is represented by $1-p\in(0,+\infty)$. It can be easily checked that the CRRA utility function \eqref{eq:Ui} satisfies Assumption $(\bm{A_{o}})$.  We also specify the benchmark process $Z=\left(Z_t\right)_{t\geq 0}$ as the GBM satisfying
\begin{align}\label{eq:Zt-GBM}
d Z_t=\mu_Z Z_t d t+\sigma_Z Z_t d W_t^{\gamma}, \quad Z_0=z\geq0,
\end{align}
with the return rate $\mu_Z\in\R$ and the volatility $\sigma_Z\geq0$. In this case, Assumptions $(\bm{A_{Z}})$ and $(\bm{A_{\rho}})$ can be simplified to that
\begin{itemize}
\item[$(\bm{A_{Z}'})$] the return rate of benchmark process satisfies $\mu_Z\geq \eta:=\sigma_Z\gamma^{\top}\sigma^{-1}\mu$.
\item[$(\bm{A_{\rho}'})$] the discount factor $\rho>\rho_0$ with $\rho_0>0$ being a constant depending on $p<1$ in $(\bm{A_{o}})$ defined by
\begin{align}\label{boundrho0-CRRA}
    \rho_0:=
    \begin{cases}
        \max\{\mu_Z,\alpha p/(1-p)\},&\text{if}~p\in (0,1),\\[0.5em]
        \max\{\mu_Z,0\},&\text{if}~p\leq 0.
    \end{cases}
    \end{align}
\end{itemize}

Under the CRRA utility and GBM benchmark, we can in fact derive the explicit expression of the dual function, and then get  the optimal strategy of portfolio and consumption in the semi-analytical feedback form in terms of the primal state variables.

\begin{proposition}\label{prop:sol-HJB}
Let Assumptions $(\bm{A_{Z}'})$ and  $(\bm{A_{\rho}'})$ hold. The dual value function in \eqref{eq:dual-func} admits the explicit expression that,  for $(y,z)\in (0,\beta]\times\mathbb{R}_+$,
\begin{align}\label{sol:hat-v}
\hat{v}(y,z)=
\begin{cases}
\displaystyle \frac{(1-p)^3}{p(\rho(1-p)-\alpha p)}y^{-\frac{p}{1-p}}+\frac{(1-p)^2}{\rho(1-p)-\alpha p}\beta^{-\frac{1}{1-p}}y+z\left(y-\frac{\beta^{-(\kappa-1)}}{\kappa}y^{\kappa}\right),& p\neq 0,\\[1em]
\displaystyle -\frac{1}{\rho} \ln y-\frac{2}{\rho}+\frac{\alpha}{\rho^2}+\frac{1}{\rho \beta} y+z\left(y-\frac{\beta^{-(\kappa-1)}}{\kappa}y^{\kappa}\right),&p=0.
\end{cases}
\end{align}
Here, the constant $\kappa$ denotes the positive root of the quadratic equation $\alpha \kappa^2+(\rho-\eta-\alpha)\kappa+\mu_Z-\rho=0$ that $\kappa:=(-(\rho-\eta-\alpha)+\sqrt{(\rho-\eta-\alpha)^2+4\alpha(\rho-\mu_Z)})/(2\alpha)>0$ with $\alpha= \mu^{\top}(\sigma \sigma^{\top})^{-1} \mu/2$ and $\eta=\sigma_Z\gamma^{\top}\sigma^{-1}\mu$.
\end{proposition}

Similar to Corollary 3.7 in \cite{DLPY22}, the value function $v$ of the primal problem \eqref{eq:u0} can be recovered via the inverse transform, which admits a semi-analytical expression involving an implicit function. It follows that  $v(x,z)=\inf_{y\in(0,\beta]}(\hat{v}(y,z)+yx)$. Then, $x=g(y,z):=-\hat{v}_y(y,z)$. Let us define $f(\cdot,z)$ as the inverse of $g(\cdot,z)$, and hence $v(x,z)=\hat{v}(f(x,z),z)+xf(x,z)$. Here, for $p<1$, we have that $f(x,z)$ can be uniquely determined by
\begin{align}\label{f1}
x=\frac{(1-p)^2}{\rho(1-p)-\alpha p}\left(f(x,z)^{-\frac{1}{1-p}}-\beta^{-\frac{1}{1-p}}\right)+z\left(\beta^{-(\kappa-1)}f(x,z)^{\kappa-1}-1\right).
\end{align}
We then have the following result which directly follows from Corollary \ref{coro:optimal-control} and Proposition \ref{prop:sol-HJB}.
\begin{corollary}\label{coro:optimal}
Let Assumptions $(\bm{A_{Z}'})$ and  $(\bm{A_{\rho}'})$ hold. Then, the value function $v$ of the primal problem \eqref{eq:u0} is given by, for all $(x,z)\in\R_+^2$,
\begin{align}\label{sol:v}
v(x,z)=
\begin{cases}
\displaystyle \frac{(1-p)^2}{p(\rho(1-p)-\alpha p)}f(x,z)^{-\frac{p}{1-p}}-\frac{1-\kappa}{\kappa}\beta^{-(\kappa-1)}zf (x,z)^{\kappa},\quad\quad p<1,~p\neq 0,\\[1em]
\displaystyle -\frac{1}{\rho} \ln f(x,z)-\frac{1}{\rho}+\frac{\alpha}{\rho^2}-\frac{1-\kappa}{\kappa}\beta^{-(\kappa-1)}zf (x,z)^{\kappa},\qquad\qquad\quad  p=0.
\end{cases}
\end{align}
Furthermore, the optimal feedback control function is given by, for all $(x,z)\in\R_+^2$,
\begin{align}\label{feedbackcontrol}
&\theta^*(x,z)=
 (\sigma\sigma^{\top})^{-1}\mu \left(\frac{1-p}{\rho(1-p)-\alpha p}f(x,z)^{-\frac{1}{1-p}}+(1-\kappa) \beta^{-(\kappa-1)}zf(x,z)^{\kappa-1}\right)\nonumber\\
&\qquad\qquad+(\sigma\sigma^{\top})^{-1}\sigma_Z\sigma\gamma z \beta^{-(\kappa-1)}f(x,z)^{\kappa-1},\\
&c^*(x,z)= f(x,z)^{\frac{1}{p-1}}.
\end{align}
In the special case when $p=-\kappa/(1-\kappa)$, where $\kappa$ is given in Proposition \ref{prop:sol-HJB}, the optimal feedback control functions admit the following explicit expressions that, for all $(x,z)\in\R_+^2$,
\begin{align}
\theta^*(x,z)&=
\left[(\sigma\sigma^{\top})^{-1}\mu \left(\frac{1-p}{\rho(1-p)-\alpha p}+(1-\kappa) \beta^{-(\kappa-1)}z\right)+(\sigma\sigma^{\top})^{-1}\sigma_Z\sigma\gamma z\beta^{-(\kappa-1)}\right]\nonumber\\
&\quad\times \frac{(\rho(1-p)-\alpha p)(x+z)+(1-p)^2\beta^{-\frac{1}{1-p}}}{(1-p)^2+(\rho(1-p)-\alpha p)\beta^{\frac{1}{1-p}}z}, \\
c^*(x,z)&=\frac{(\rho(1-p)-\alpha p)(x+z)+(1-p)^2\beta^{-\frac{1}{1-p}}}{(1-p)^2+(\rho(1-p)-\alpha p)\beta^{\frac{1}{1-p}}z}.
\end{align}
\end{corollary}

\begin{remark}\label{rem0}
In view of \eqref{f1}, when $x=0$, we have $f(0,z)=\beta$ for all $z\geq0$. Consequently, in the case when $d=1$, Corollary \ref{coro:optimal} with the setting of $\gamma=1$ and $\mu>0$ yields that
\begin{align}\label{eq:optimal-x0}
c^*(0,z)=\beta^{\frac{1}{p-1}}>0,\quad \theta^*(0,z)= \frac{\mu}{\sigma^2} \left(\frac{1-p}{\rho(1-p)-\alpha p}\beta^{\frac{1}{p-1}}+(1-\kappa) z\right)+\frac{\sigma_Z }{\sigma}z>0.
\end{align}
This implies that, when the state level $X_t^*=0$ (described as in Corollary \ref{coro:optimal-control}) at time $t>0$, both the optimal portfolio $\theta^*_t$ and the optimal consumption $c^*_t$ are strictly positive. That is, at the extreme case when the wealth process $V_t^{*}$ equals the benchmark $Z_t$, the allowance of capital injection motivates the fund manager to be more risk seeking by strategically choosing positive consumption to attain a higher expected utility.
\end{remark}

The next result provides more accurate estimations of the expected largest shortfall (expectation of the total optimal capital injection) under the CRRA utility and GBM benchmark.
\begin{corollary}\label{coro:inject-captial-CRRA-GBM}
Let Assumptions $(\bm{A_{Z}'})$ and $(\bm{A_{\rho}'})$  hold.  Then, we have
\begin{itemize}
\item[{\rm(i)}] The expected largest shortfall (expectation of the total capital injection) under the optimal strategy $(\theta^*,c^*)$ in Corollary \ref{coro:optimal-control} is finite and admits the precise expression: for $(x,z)\in\R_+^2$,
\begin{align}\label{eq:injection-CRRA}
&\Ex\left[\int_0^{\infty} e^{-\rho t}dA^{\theta^*,c^*}_t \right]= \frac{1-p}{\rho(1-p)-\alpha p}\beta^{-\frac{2-p}{1-p}}f(x,z)+\frac{1-\kappa}{\kappa}\beta^{-\kappa}zf (x,z)^{\kappa}. 
\end{align}.
\item[{\rm(ii)}] The expected largest shortfall (expectation of the total capital injection) under any admissible $(\theta,c)\in\mathbb{U}^r$ is positive that 
\begin{align}\label{eq:dAstarpositive-CRRA}
\Ex\left[\int_0^{\infty} e^{-\rho t}dA^{\theta,c}_t \right]\geq z\frac{1-\kappa}{\kappa}\left(1+\frac{x}{z}\right)^\frac{\kappa}{\kappa-1}>0,\quad \forall (x,z)\in\R_+\times (0,\infty).
\end{align}
\end{itemize}
 Here,  the largest shortfall (capital injection) under the admissible strategy $(\theta,c)$ is given by
$A_t^{\theta,c} =0\vee \sup_{s\leq t}(Z_{s}-V_{s}^{\theta,c})$ for $t\geq0$.
\end{corollary}

We next present some structural properties of the optimal pair $(\theta^*,c^*)=(\theta^*_t,c_t^*)_{t\geq0}$, and we shall consider the case when $d=1$ and the return rate $\mu>r=0$. The next result characterizes the asymptotic behavior of the optimal portfolio-wealth ratio and the  consumption-wealth ratio obtained in Corollary~\ref{coro:optimal} when the initial wealth level tends to infinity. The proof of the next lemma is provided in Section~\ref{sec:proof}.
\begin{lemma}\label{lem:asymptotic-control}
Let Assumptions $(\bm{A_{Z}'})$ and  $(\bm{A_{\rho}'})$ hold. Consider the optimal feedback control functions $\theta^*(x,z)$ and $c^*(x,z)$ for $(x,z)\in\R_+^2$ provided in Corollary~\ref{coro:optimal}. Then, for any $z>0$ fixed,  we have
\begin{align}\label{eq:limtheta}
\lim_{x\to +\infty}\frac{\theta^*(x,z)}{x}&=
\begin{cases}
\displaystyle \frac{\mu}{\sigma^2(1-p)}, &p>p_1,\\
\displaystyle \frac{\mu}{\sigma^2(1-p_1)}+\frac{\sigma_Z}{\sigma}\frac{C^*(p_1)\beta^{\frac{1}{1-p_1}}z}{1+C^*(p_1)\beta^{\frac{1}{1-p_1}}z},&p=p_1,\\[0.8em]
\displaystyle \frac{\mu}{\sigma^2(1-p_1)}+\frac{\sigma_Z}{\sigma},&p<p_1,
\end{cases}
\end{align}
and
\begin{align}\label{eq:limc}
\lim_{x\to +\infty}\frac{c^*(x,z)}{x}&=
\begin{cases}
\displaystyle ~~~~~~~~~C^*(p), &p>p_1,\\[0.6em]
\displaystyle \frac{C^*(p_1)}{1+C^*(p_1)\beta^{-\frac{1}{1-p_1}}z},&p=p_1,\\[0.8em]
~~~~~~~~~~0, & p<p_1.
\end{cases}
\end{align}
Here, the constant $C^*(p):=\frac{\rho(1-p)-\alpha p}{(1-p)^2}$ coincides with the constant limit in the Merton solution as in \cite{Merton1971}, and $p_1\in(-\infty,0)$ is defined by $p_1:=-\frac{\kappa}{1-\kappa}$.
\end{lemma}

\begin{remark}\label{compareM}
Based on Corollary \ref{coro:inject-captial-CRRA-GBM} and Lemma \ref{lem:asymptotic-control}, we have some further comparison discussions between our optimal tracking problem \eqref{eq:w} and the main result in \cite{Merton1971}:
\begin{itemize}
\item[{\rm(i)}] In the case $\mathrm{v}\geq z\geq0,~\mu_Z=\sigma_Z=0$ and $\beta=+\infty$, (i.e., the benchmark process is always zero and the cost of capital injection is infinity), then the optimal tracking problem \eqref{eq:w} degenerates into the classical Merton's problem. Indeed, in this degenerate case, the equation $\alpha \kappa^2+(\rho-\eta-\alpha)\kappa+\mu_Z-\rho=0$ in Proposition \ref{prop:sol-HJB} reduces to $\alpha \kappa^2+(\rho-\alpha)\kappa-\rho=0$, which yields $\kappa=1$. As $\kappa=1$ and $\beta=+\infty$ in Eq. \eqref{f1}, we can see that, for all $x\geq 0$,
\begin{align}\label{eq:z=0}
x=\frac{(1-p)^2}{\rho(1-p)-\alpha p}f(x,z)^{-\frac{1}{1-p}}.
\end{align}
Then, it follows from \eqref{eq:z=0} and Corollary \ref{coro:optimal} that
\begin{align}\label{eq:Merton}
\theta^*(x,z)=\theta^{\rm{Mer}}(x):=\frac{\mu}{\sigma^2(1-p)}x,\quad
c^*(x,z)=c^{\rm{Mer}}(x):=\frac{\rho(1-p)-\alpha p}{(1-p)^2}x,
\end{align}
which is exactly the optimal solution of the classical Merton's problem. Moreover, using \eqref{eq:Merton}, we can easily see that the local time $L_t=0$ for all $t\geq 0$, thus there is no capital injection in this case.

\item[{\rm(ii)}] In the case $\mathrm{v}\geq z\geq0$, $\mu_Z=\sigma_Z=0$ and $\beta\in(0,+\infty)$, the benchmark is constant $z$, but the fund manager is still allowed to inject capital and the cost of capital injection is finite, we find that the optimal portfolio and consumption strategies differ from those in the classical Merton's problem. In fact, as in the above case (i), we still have $\kappa=1$ in Eq. \eqref{f1}. Hence, for all $x\geq 0$,
\begin{align}\label{eq:zgeq0}
x=\frac{(1-p)^2}{\rho(1-p)-\alpha p}\left(f(x,z)^{-\frac{1}{1-p}}-\beta^{-\frac{1}{1-p}}\right).
\end{align}
Then, it follows from \eqref{eq:z=0} and Corollary \ref{coro:optimal} that
\begin{align}\label{eq:strategyzgeq0}
\begin{cases}
\displaystyle \theta^*(x,z)=\theta^{\rm{Mer}}(x)+\underbrace{\frac{\mu}{\sigma^2}\frac{1-p}{\rho(1-p)-\alpha p}\beta^{-\frac{1}{1-p}}}_{\text{positive adjustment by capital injection}},\\
\displaystyle c^*(x,z)=c^{\rm{Mer}}(x) +\underbrace{\qquad\beta^{-\frac{1}{1-p}}\qquad}_{\text{positive adjustment by capital injection}}.
\end{cases}
\end{align}
Although both the Merton problem without capital injection and our problem in \eqref{eq:w} are solvable, due to the encouragement of risk-taking from the capital injection, the fund manager in our problem \eqref{eq:w} adopts more aggressive optimal portfolio and consumption strategies with additional positive adjustment terms as shown in \eqref{eq:Merton} and \eqref{eq:strategyzgeq0}. It is also observed that these positive adjustments are independent of the wealth level $x$. In particular, the adjustment terms are decreasing w.r.t. the cost parameter $\beta$. When $\beta$ tends to infinity, it is clear that both adjustment terms of capital injection will be vanishing and we have
$\lim_{\beta\to\infty}\theta^*(x,z)=\theta^{\rm{Mer}}(x)~\text{and}~\lim_{\beta\to\infty}c^*(x,z)=c^{\rm{Mer}}(x),~\forall x\geq0.$ See further discussions on their financial implications in Section \ref{sec:num}.

\item[{\rm(iii)}] In the case $\mathrm{v}\geq z>0$ and $\mu_Z>\eta$, the capital injection is indeed necessary for the agent in our optimal tracking problem \eqref{eq:w}. Corollary \ref{coro:inject-captial-CRRA-GBM} shows this fact under any admissible strategy. 
That is, if we consider the classical Merton's problem by requiring the strict benchmark outperforming constraint $V_t\geq Z_t$, a.s., for all $t\geq 0$, the admissible set of Merton problem under the outperforming constraint is empty. Therefore, it is necessary to consider the relaxed benchmark outperforming constraint using the strategic capital injection such that the wealth process is allowed to be negative from time to time.
\end{itemize}
\end{remark}

 In the setting of utility maximization of consumption without benchmark tracking, \cite{CK1996} has discussed the concavity of the optimal feedback consumption function in terms of the wealth level induced by the income uncertainty. It was pointed out therein that the concavity of consumption can imply several interesting economic insights including the real-life observations that the financial risk-taking is often strongly related to wealth. However, it has also been shown in various extended models that the optimal consumption function may turn out to be convex within some wealth intervals when the consumption performance and risk-taking are affected by the past-average dependent habit formation (see \cite{Liu23}), the consumption drawdown constraint (see \cite{ABY} and \cite{LYZ}) or the possible negative terminal debt (see \cite{ChenV}).  The following proposition presents the monotonicity and concavity/convexity of the optimal feedback control functions $x\to \theta^*(x,z)$ and $x\to c^*(x,z)$ in our new problem formulation.

\begin{proposition}\label{lem:prop}
Let Assumptions $(\bm{A_{Z}'})$ and  $(\bm{A_{\rho}'})$ hold. Consider the optimal feedback control functions $\theta^*(x,z)$ and $c^*(x,z)$ for $(x,z)\in\R_+^2$ provided in Corollary~\ref{coro:optimal}. Then, for any $z>0$ fixed,  we have
\begin{itemize}
\item[{\rm(i)}]$x\to c^*(x,z)$ is increasing;
\item[{\rm(ii)}] If $p\in(p_1,1)$, $x\to c^*(x,z)$ is strictly convex; if $p\in(-\infty,p_1)$, $x\to c^*(x,z)$ is strictly concave; and, if $p=p_1$, then $c^*(x,z)$ is linear in $x$;
\item[{\rm(iii)}] For the correlative coefficient $\gamma=1$, $x\to \theta^*(x,z)$ is increasing;
\item[{\rm(iv)}] For the correlative coefficient $\gamma=1$, if $p\in(-\infty,p_1)\cup(p_2,1)$, $x\to \theta^*(x,z)$ is strictly convex; if $p\in(p_1,p_2)$, $x\to \theta^*(x,z)$ is strictly concave; and, if $p=p_1$ or $p=p_2$, $\theta^*(x,z)$ is linear in $x$.
\end{itemize}
Here, $p_1$ is given in Lemma \ref{lem:asymptotic-control} and the critical point $p_2>p_1$ is given by $p_2:=\frac{\sigma\sigma_Z-\mu \kappa}{\mu(1-\kappa)+\sigma\sigma_Z}\in(-\infty,1).$
\end{proposition}

In our extended Merton's problem, the risk aversion level from the utility function has been distorted by the relaxed benchmark tracking constraint. Indeed, the allowance of strategic capital injection not only enlarges the set of admissible controls, but also incentivizes the fund manager to be more risk-taking in choosing aggressive portfolio and consumption plans. From Proposition \ref{lem:prop}, one can observe that when the fund manager is very risk averse such that $p\leq p_1$, the risk aversion attitude from the utility function plays the dominant role and hence the optimal consumption functions exhibits concavity as observed in \cite{CK1996}; but when the fund manager is much less risk averse or close to risk neutral, i.e., $p>p_1$, the risk-taking component from the capital injection starts to distort the fund manager's decision making, leading to convex optimal consumption function. In this case, when the wealth increases, the fund manager is willing to inject more capital to achieve the increasing marginal consumption. Some illustrative plots of different convexity results with respect to the variable $x$ are reported in the next section.

\section{Numerical Examples and Financial Implications}\label{sec:num}

In this section, under the CRRA utility and GBM benchmark process in Section~\ref{sec:CRRA}, we present some numerical examples to illustrate some other quantitative properties and financial implications of the optimal feedback control functions and the expectation of the discounted total capital injection. To ease the discussions, we only consider $d=1$ in all examples and focus primarily on the capital injection model formulation.

We first discuss some financial implications of the adjustment impact by capital injection in Remark \ref{compareM} through the next few numerical examples. Let us take the cost parameter $\beta=1$ for simplicity. Firstly, let us consider the case when $\mu_Z=\sigma_Z=0$, i.e., the benchmark is constant $z\in\R_+$.  It follows from \eqref{eq:zgeq0} that, the optimal portfolio-wealth ratio and the consumption-wealth ratio $(\frac{\theta^*(x,z)}{x},\frac{c^*(x,z)}{x})$ admit the following expression that, for $(x,z)\in\R_+^2$,
    \begin{align}\label{eq:strategyzgeq1}
\begin{cases}
\displaystyle \frac{\theta^*(x,z)}{x}=\frac{\mu}{\sigma^2(1-p)}+x^{-1}\frac{\mu}{\sigma^2}\frac{1-p}{\rho(1-p)-\alpha p},\\[0.6em]
\displaystyle \frac{c^*(x,z)}{x}=C^*(\mu,\sigma,p) +x^{-1}.
\end{cases}
\end{align}
We can observe from \eqref{eq:strategyzgeq1} that the portfolio-wealth ratio and consumption-wealth ratio are no longer constant comparing with the classical Merton's solution. Instead, the adjustment impacts by capital injection per wealth are decreasing in wealth and they are independent of the constant benchmark level $z\in\mathbb{R}_+$. We also have from \eqref{eq:strategyzgeq1} the following financial implications:
\begin{itemize}
\item The wealth level has adverse effect on the adjustment impact. When the  wealth level is sufficiently high, the adjustment impact becomes negligible. This can be explained by the fact that, for the constant benchmark $Z_t\equiv z\geq0$ (i.e., $\mu_Z=\sigma_Z=0$), the relaxed tracking constraint can be easily achieved with a higher wealth level without requesting frequent capital injection. In the extreme case, it also holds that $\lim_{x\to\infty}(\frac{\theta^*(x,z)}{x},\frac{c^*(x,z)}{x})=(\frac{\mu}{\sigma^2(1-p)},C^*(\mu,\sigma,p))$, and   ${\rm w}(x,z)={\rm w}^{\rm Mer}(x,z)-{\rm w}^{\rm cost}(x,z)\approx{\rm w}^{\rm Mer}(x,z)$ as $x\rightarrow \infty$, where ${\rm w}^{\rm Mer}(x,z):=\Ex\left[\int_0^{\infty}e^{-\rho t}U(c^*(X_t^*,z))dt\right]$ and ${\rm w}^{\rm cost}(x,z):=\Ex\left[\int_0^{\infty}e^{-\rho t}dA_t^*\right]$ (see Figure \ref{fig:cost-utility}-(a)). In addition, for any constant benchmark $z\in\mathbb{R}_+$, the optimal portfolio and consumption strategies are independent of $z$, leading to the same expected utility of consumption.

\item When the wealth level $x$ is low, a reasonable amount of capital injection is needed to fulfill the tracking constraint with respect to the constant benchmark. The fund manager needs to balance the trade-off between the utility of consumption (${\rm w}^{\rm Mer}(x,z)$) and the cost of capital injection (${\rm w}^{\rm cost}(x,z)$) to obtain an optimal profit (${\rm w}(x,z)$) (see Figure \ref{fig:cost-utility}-(b)). Moreover, the adjustment term raised by the capital injection in \eqref{eq:strategyzgeq1} is positive, indicating that the fund manager behaves more aggressively in both portfolio and consumption plans comparing with the Merton solution under no bankruptcy constraint. This can be explained by the fact that the possible capital injection (or the tolerance of the positive shortfall when the wealth falls below the constant benchmark) can significantly incentivize the fund manager to attain a higher expected utility as long as the cost of capital injection (or the expected largest shortfall) stays relatively small comparing with the increment in expected utility.

\begin{figure}
\centering
 \subfigure[]{
 \includegraphics[width=6cm]{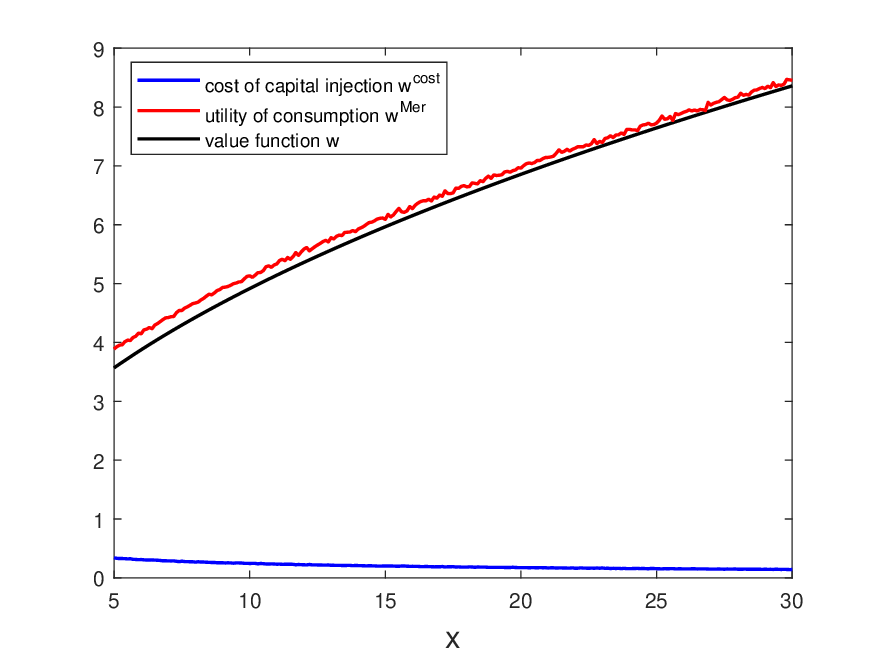}
 }\hspace{-8mm}
 \subfigure[]{
 \includegraphics[width=6cm]{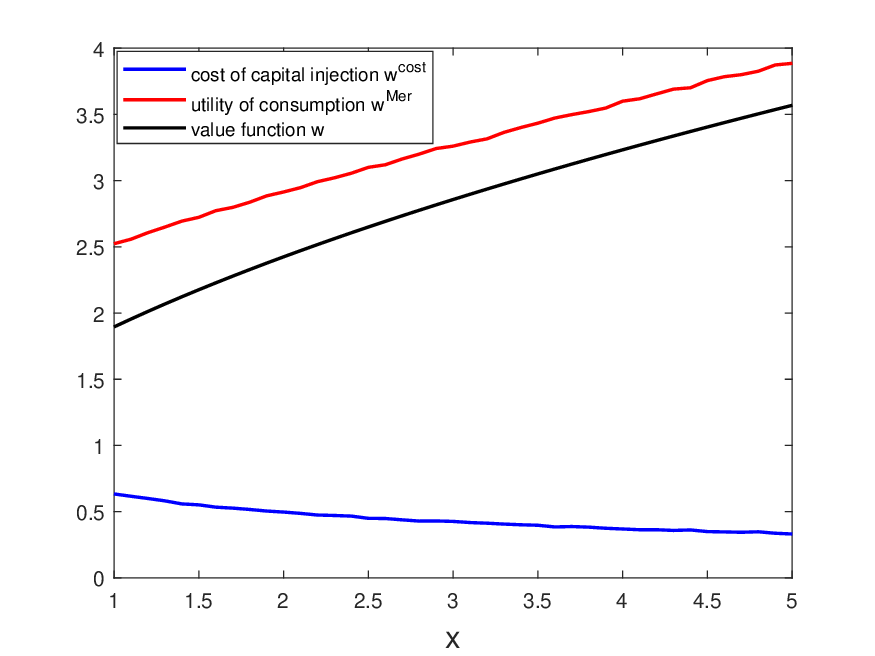}
 }\hspace{-8mm}
 \caption{The optimal value function $x\to {\rm w}(x,z)$, the utility of consumption $x\to {\rm w}^{\rm Mer}(x,z)$ and the cost of capital injection $x\to {\rm w}^{\rm cost}(x,z)$ .  The model parameters are set as $z=1,\rho=1,~\mu=0.5,~\sigma=1,~\mu_Z=0,~\sigma_Z=0,~p=0.5,~\gamma=1,~\beta=1$.}\label{fig:cost-utility}
\end{figure}
\end{itemize}

When the benchmark process $Z=(Z_t)_{t\geq0}$ is a GBM (i.e., $\mu_Z>0$ and $\sigma_Z>0$), the trade-off between the utility maximization and the goal of tracking becomes much more sophisticated, which may considerably rely on the current wealth level, the performance of the benchmark as well as the risk aversion level of the fund manager. In view of the complicated expression of the optimal feedback functions in \eqref{feedbackcontrol}, we are not able to conduct any clean quantitative comparison between our solution and the Merton's solution.

In order to conclude some interpretable financial implications from the optimal feedback controls, we next only consider and discuss the asymptotic case when the wealth level tends to infinity. As stated in Lemma \ref{lem:asymptotic-control}, as the wealth level $x$ goes extremely large, the asymptotic optimal portfolio-wealth ratio and the consumption-wealth ratio $(\frac{\theta^*(x,z)}{x},\frac{c^*(x,z)}{x})$ admit the expressions that
    \begin{align}\label{eq:limthetaliadd}
\begin{cases}
\displaystyle \left(\frac{\mu}{\sigma^2(1-p)},C^*(\mu,\sigma,p)\right),\qquad 1-p<{\rm CRA}(\mu,\sigma,\mu_Z,\sigma_Z),\\
\displaystyle \left(\frac{\mu}{\sigma^2(1-p_1)}+\frac{\sigma_Z}{\sigma}\frac{C^*(\mu,\sigma,p_1)z}{1+C^*(\mu,\sigma,p_1)z},\frac{C^*(\mu,\sigma,p_1)}{1+C^*(\mu,\sigma,p_1)z}\right), \quad 1-p={\rm CRA}(\mu,\sigma,\mu_Z,\sigma_Z),\\
\displaystyle \left(\frac{\mu}{\sigma^2(1-p_1)}+\frac{\sigma_Z}{\sigma},0\right),\qquad\quad 1-p>{\rm CRA}(\mu,\sigma,\mu_Z,\sigma_Z),
\end{cases}
\end{align}
where the critical risk averse (CRA) level ${\rm CRA}(\mu,\sigma,\mu_Z,\sigma_Z)$ is defined by
\begin{align}\label{eq:CRA}
{\rm CRA}(\mu,\sigma,\mu_Z,\sigma_Z):=1-p_1=\frac{1}{1-\kappa}>1
\end{align}
with $\kappa=\frac{\sqrt{(\rho-\eta-\alpha)^2+4\alpha(\rho-\mu_Z)}-(\rho-\eta-\alpha)}{2\alpha}\in(0,1)$, $\alpha=\frac{\mu^2}{2\sigma^2}$, and $\eta=\frac{\sigma_Z}{\sigma}\gamma\mu$.

Next, we examine some quantitative properties of the CRA level ${\rm CRA}(\mu,\sigma,\mu_Z,\sigma_Z)$ w.r.t. the return rate $\mu_Z$ and the volatility $\sigma_Z$ of benchmark process. Note that the mapping $\mu_Z\to \kappa:=\frac{\sqrt{(\rho-\eta-\alpha)^2+4\alpha(\rho-\mu_Z)}-(\rho-\eta-\alpha)}{2\alpha}$ is decreasing. Therefore, the CRA level $\mu_Z\to{\rm CRA}(\mu,\sigma,\mu_Z,\sigma_Z)$ is decreasing for any $\sigma_Z\geq0$ fixed (see Figure \ref{fig:CRA-muZ-sigmaZ}-(a)). In other words, the higher the return rate of benchmark index, the lower the CRA level.  For any $\mu_Z\in\R$ fixed, $\sigma_Z\to {\rm CRA}(\mu,\sigma,\mu_Z,\sigma_Z)$ is non-decreasing (see Figure \ref{fig:CRA-muZ-sigmaZ}-(b)). We next check the extreme case when $\mu_Z=\sigma_Z=0$ (i.e., the constant benchmark case) and find that the resulting CRA level ${\rm CRA}(\mu,\sigma,\mu_Z,\sigma_Z)$ tends to $+\infty$ as $\mu_Z,\sigma_Z\to0$. Indeed, when $\mu_Z=\sigma_Z=0$, we have the parameter $\kappa=1$ as the parameter $\eta=0$. Then, in lieu of the definition of CRA level, it obviously holds that ${\rm CRA}(\mu,\sigma,0,0)=\frac{1}{1-1}=+\infty$. This is consistent with our previous discussion that when the wealth level is sufficiently high, the impact by the capital injection becomes negligible for the constant benchmark. 

\begin{figure}
\centering
  \subfigure[]{
        \includegraphics[width=6cm]{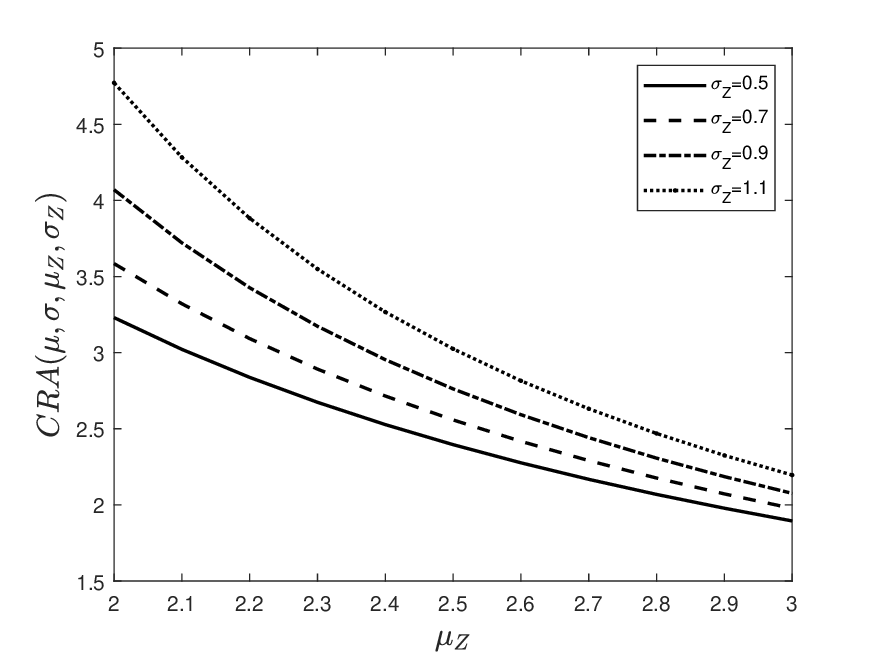}
    }\hspace{-8mm}
  \subfigure[]{
        \includegraphics[width=6cm]{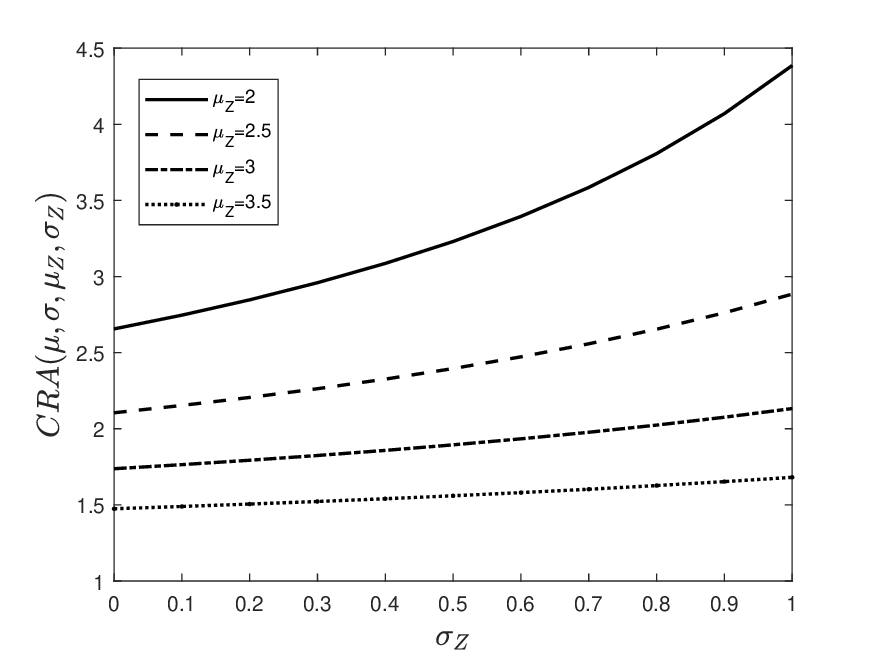}
    }
 \caption{(a): The CRA level $\mu_Z\to \text{CRA}(\mu,\sigma,\mu_Z,\sigma_Z)$. (b):  The CRA level $\sigma_Z\to \text{CRA}(\mu,\sigma,\mu_Z,\sigma_Z)$ . The model parameters are set as $\rho=5,~\mu=1,~\sigma=1,~p=-1,~\gamma=1,~\beta=1$.}\label{fig:CRA-muZ-sigmaZ}
\end{figure}

Moreover, a key observation here is that the asymptotic behavior in \eqref{eq:limthetaliadd} depends purely on the risk aversion parameter and the performance of the benchmark. We now summarize the detailed financial implications as below:
\begin{itemize}
\item[(i)] If the fund manager is less risk averse such that $1-p<{\rm CRA}(\mu,\sigma,\mu_Z,\sigma_Z)$, the fund manager's optimal portfolio-wealth ratio and the consumption-wealth ratio coincide with the classical Merton's limit $\left(\frac{\mu}{\sigma^2(1-p)},C^*(\mu,\sigma,p)\right)$ as $x\rightarrow\infty$. This can be explained by the fact that the fund manager with a low risk aversion, being aware of the extremely large wealth level, will be more aggressive in investing in the risky asset. In turn, the resulting large wealth process from the financial market can stably  outperform the benchmark process most of the time, yielding the capital injection almost negligible. As a consequence, the fund manager's asymptotic consumption plan also behaves like the counterpart in the Merton's solution without benchmark tracking.

\item[(ii)] When the fund manager's risk averse level $1-p$ equals or is higher than the CRA level ${\rm CRA}(\mu,\sigma,\mu_Z,\sigma_Z)$, she will withhold the asymptotic consumption plan comparing with the case of low risk aversion or the Merton's asymptotic consumption. At the same time, it is interesting to see that the fund manager actually chooses a more aggressive asymptotic investment plan comparing with the case of low risk aversion, which is almost counter-intuitive. However, we note that the trade-off becomes more severe for the high risk averse fund manager, who would concern more on the cost of capital injection that drives her to invest more and consume less, as a way to avoid a large amount of capital injection. In this way, the high risk averse fund manager would hope that the resulting wealth process can stay at a high level outperforming the benchmark so that the expected capital injection can be maintained at a low level. As the high risk averse fund manager would put more wealth into the risky asset, the associated optimal consumption-wealth ratio is also restrained and smaller than the counterpart in the Merton's solution.

  \item[(iii)]  For a fixed return rate of benchmark index, the volatility $\sigma_Z$ of benchmark index has a significant impact on the optimal asymptotic portfolio-wealth ratio when the fund manager has higher risk aversion than the CRA level ${\rm CRA}(\mu,\sigma,\mu_Z,\sigma_Z)$. The higher the volatility $\sigma_Z$ of benchmark process, the more the fund manager invests in the risky asset (see Figure \ref{fig:optimal-theta-sigmaZ}), hoping the wealth process from the financial market can outperform the more fluctuating benchmark. On the other hand, we note that the higher the return rate $\mu_z$ of the benchmark, the smaller the CRA level, which yields that the high risk averse fund manager will be more likely to increase her investment in the risky asset, again hoping the gain from the risky asset to beat the high benchmark return.

  \item[(iv)] When the benchmark process is deterministic (i.e., $\sigma_Z=0$), the optimal asymptotic portfolio-wealth ratio and consumption-wealth ratio $(\frac{\theta^*(x,z)}{x},\frac{c^*(x,z)}{x})$ are decreasing in terms of the risk averse parameter $1-p$. In fact, when $\sigma_Z=0$, the limit of $(\frac{\theta^*(x,z)}{x},\frac{c^*(x,z)}{x})$ admit the simplified expressions that
      \begin{align}\label{eq:limthetalisigmaZ=0add}
\begin{cases}
\displaystyle \left(\frac{\mu}{\sigma^2(1-p)},C^*(\mu,\sigma,p)\right), &1-p<{\rm CRA}(\mu,\sigma,\mu_Z,0),\\[0.8em]
\displaystyle \left(\frac{\mu}{\sigma^2(1-p_1)},\frac{C^*(\mu,\sigma,p_1)}{1+C^*(\mu,\sigma,p_1)z}\right),&1-p={\rm CRA}(\mu,\sigma,\mu_Z,0),\\[0.8em]
\displaystyle \left(\frac{\mu}{\sigma^2(1-p_1)},0\right),&1-p>{\rm CRA}(\mu,\sigma,\mu_Z,0),
\end{cases}
\end{align}
where we recall that $1-p_1={\rm CRA}(\mu,\sigma,\mu_Z,0)$. The left panel (a) (resp. the right panel (b)) of Figure \ref{fig:optimal-p-new} displays the optimal portfolio-wealth ratio (resp. consumption-wealth ratio) w.r.t. the risk averse parameter $1-p$ for a fixed large initial wealth level $x$ under the different return rates of benchmark $\mu_Z=0.3$, $0.4$ and $0.8$.
\begin{figure}[h]
\centering
 \includegraphics[width=6cm]{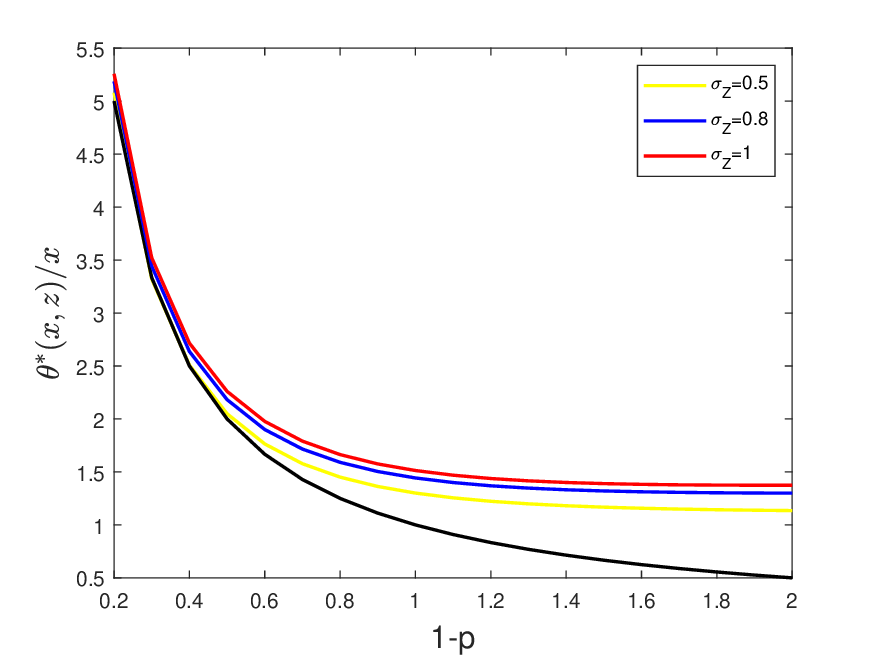}
 \caption{The sensitivity of the asymptotic optimal portfolio-wealth ratio w.r.t. $1-p$. The model parameters are set as $(x,z)=(5,1),~\rho=3,~\mu=1,~\sigma=1,~\mu_Z=2,~\sigma_Z=0.5,~\gamma=1$.}\label{fig:optimal-theta-sigmaZ}
\end{figure}
\begin{figure}[h]
\centering
  \subfigure[]{
        \includegraphics[width=6cm]{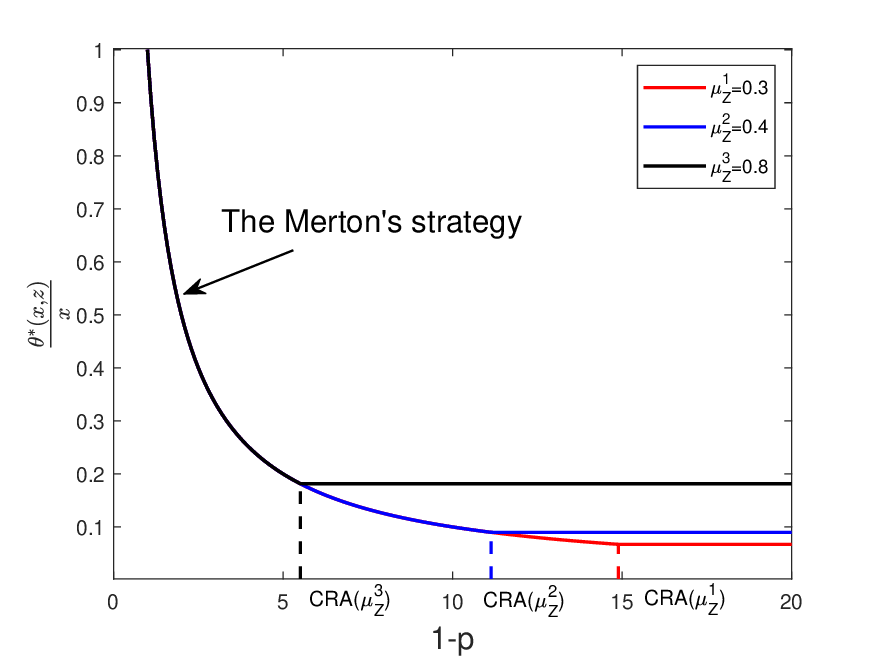}
    }\hspace{-8mm}
  \subfigure[]{
        \includegraphics[width=6cm]{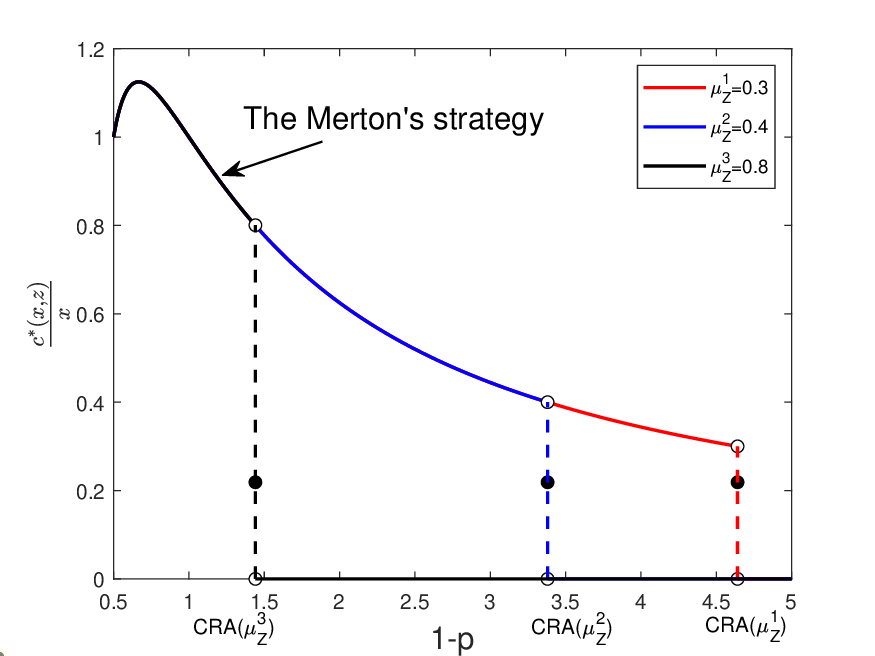}
    }
 \caption{(a): The sensitivity of the asymptotic optimal portfolio-wealth ratio w.r.t. $1-p$. (b): The sensitivity of the asymptotic optimal consumption-wealth ratio w.r.t. $1-p$. The model parameters are set as $(x,z)=(20,1),~\rho=5,~\mu=1,~\sigma=1,~\sigma_Z=0,~\gamma=1$.}\label{fig:optimal-p-new}
\end{figure}
It is observed that, for each fixed return rate $\mu_Z$ of the benchmark, the optimal portfolio-wealth ratio is continuously decreasing w.r.t. $1-p$; while the optimal consumption-wealth ratio is strictly decreasing (it jumps down at the CRA level ${\rm CRA}(\mu,\sigma,\mu_Z,0)$). This observation is similar to that in the Merton's case: the fund manager invests less in the risky asset and consumes less if she is more risk averse. Moreover, the fund manager will implement a Merton's portfolio strategy {\it locked at} the CRA level ${\rm CRA}(\mu,\sigma,\mu_Z,0)$ once she is more risk averse than the CRA level. However, when the fund manager is less risk averse than the CRA level, she will execute the classical Merton's strategy depending on her current risk averse level $1-p$. Note that the CRA level $\mu_Z\to{\rm CRA}(\mu,\sigma,\mu_Z,0)$ is decreasing. Therefore, the higher the return rate of the benchmark process, the lower the CRA level. This implies that, if the return rate of the benchmark process is very high, it is more likely that the high risk averse fund manage would follow the asymptotic behavior $(\frac{\mu}{\sigma^2(1-p_1)},0)$ in the Merton's solution with the locked risk aversion level $1-p_1$ (regardless of the true risk aversion level $1-p$ from the fund manager's utility function) and the asymptotic portfolio-wealth ratio also becomes larger as $1-p_1={\rm CRA}(\mu,\sigma,\mu_Z,0)$ is smaller.
\end{itemize}

In what follows, to numerically illustrate the convexity or concavity of the optimal feedback portfolio and consumption functions under different risk aversion in Proposition \ref{lem:prop}, we plot different cases in Figures \ref{fig:theta} and \ref{fig:c} respectively. More precisely, Figure \ref{fig:theta} displays the optimal portfolio w.r.t. the wealth level $x$ under different choices of $p$ in which the return rate of the risky asset is set to be low (resp. high) when $\mu=0.1$ (resp. $\mu=1$) in the left panel (a) (resp. the right panel (b)) of Figure \ref{fig:theta}. In both panels of Figure \ref{fig:theta}, the optimal portfolio is increasing in wealth for all risk aversion parameters $1-p$ chosen, similar to the classical Merton solution. However, as a contrast, due to the capital injection and the goal of benchmark tracking, our optimal feedback portfolio $\theta^*(x,1)$ is only linear in wealth at two critical risk aversion parameters $1-p_1$ and $1-p_2$ ($p_2>p_1$). In particular, when the risk aversion parameter $1-p$ falls between $1-p_2$ and $1-p_1$, our optimal  portfolio feedback function is concave in $x$ (see the left panel (a) of Figure \ref{fig:theta}); while for the risk aversion parameter $1-p$ falls outside of the risk aversion interval $[1-p_2,1-p_1]$, the optimal portfolio feedback function is strictly convex in $x$ (see the left panel (b) of Figure \ref{fig:theta}). This is precisely reflected in the claim (iv) of Proposition \ref{lem:prop}.

\begin{figure}[ht]
\centering
  \subfigure[]{
        \includegraphics[width=6cm]{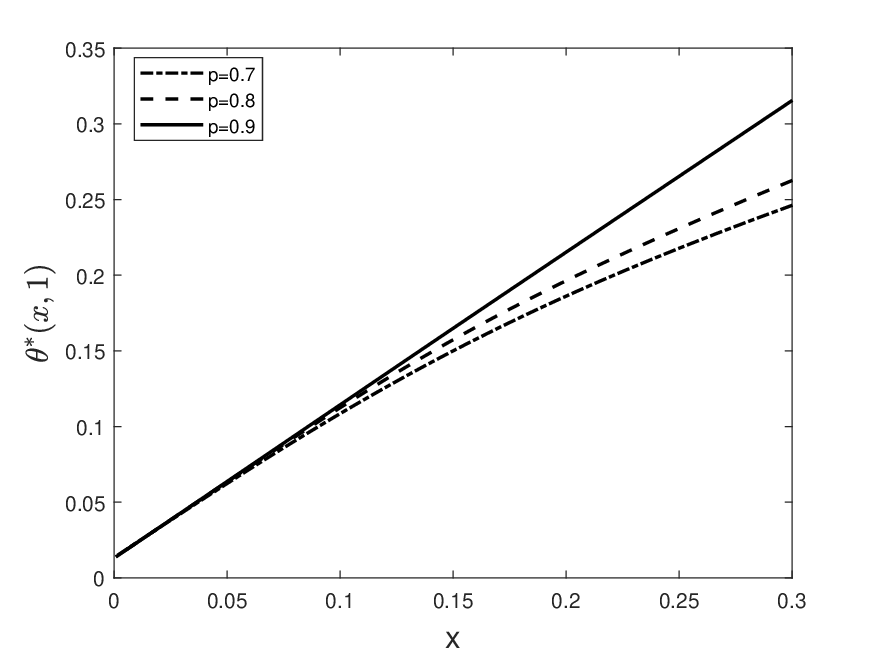}
    }\hspace{-8mm}
  \subfigure[]{
        \includegraphics[width=6cm]{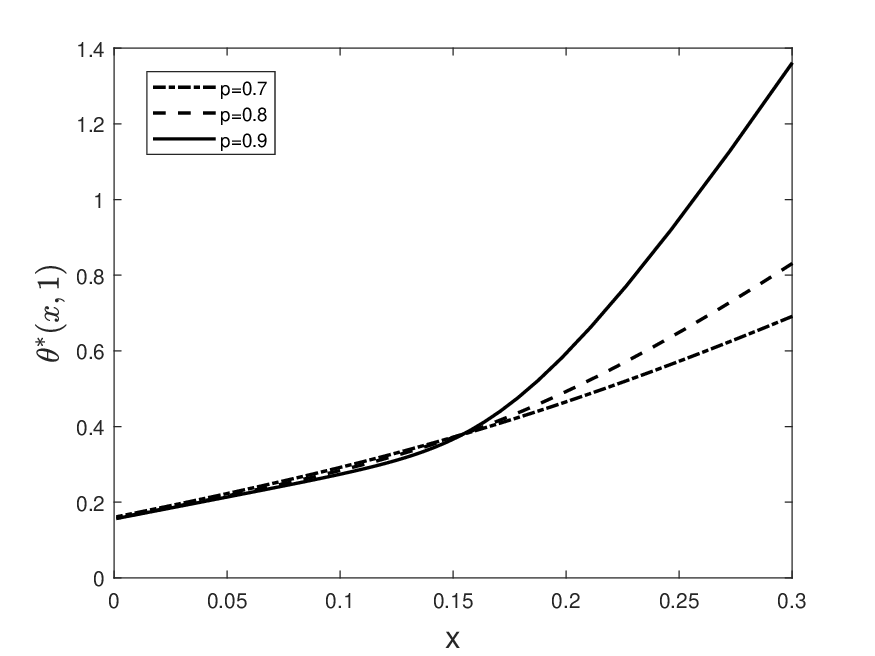}
    }
 \caption{The optimal portfolio $x\to \theta^*(x,1)$.  The model parameters are set as  $\rho=7,~\sigma=1,~\mu_Z=1,~\sigma_Z=1,~\gamma=1,~\beta=3$ and  $\mu=0.1$ in panel (a), $\mu=1$ in panel (b).}\label{fig:theta}
\end{figure}

\begin{figure}[ht]
\centering
  \subfigure[]{
        \includegraphics[width=6cm]{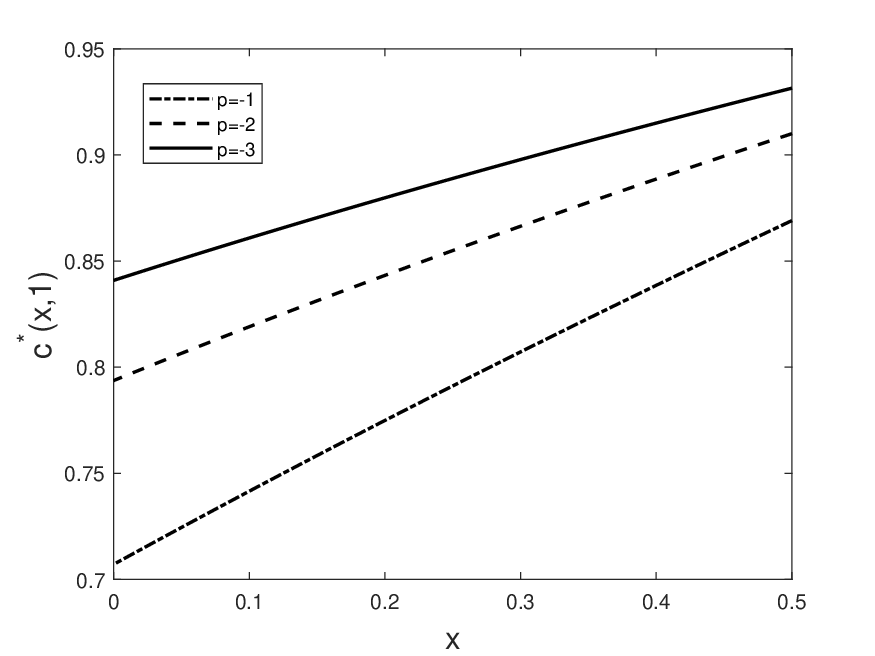}
    }\hspace{-8mm}
  \subfigure[]{
        \includegraphics[width=6cm]{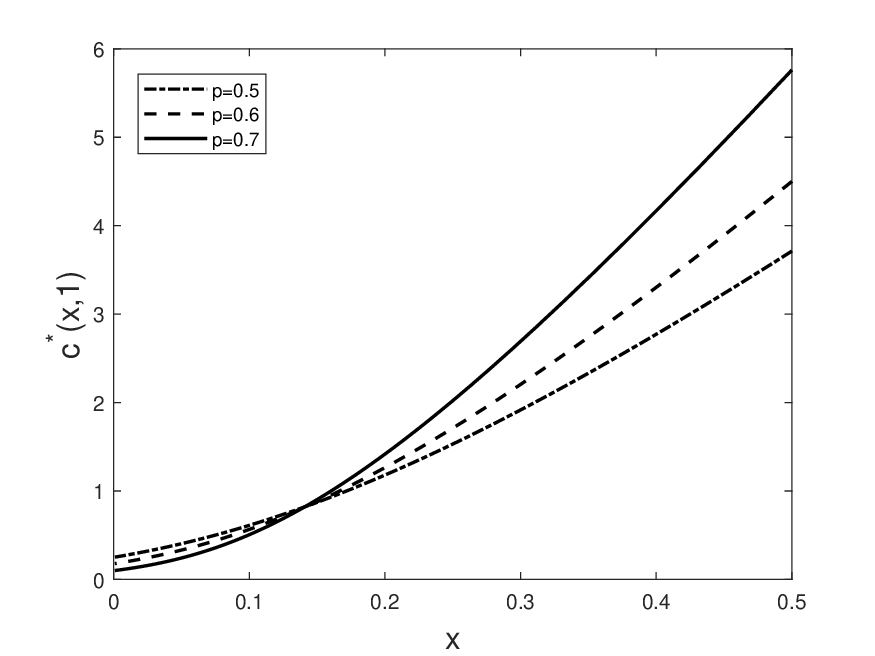}
    }
 \caption{The optimal consumption $x\to c^*(x,1)$. The model parameters are set as  $\rho=7,~\mu=1,~\sigma=1,~\sigma_Z=1,~\gamma=1,~\beta=2$ and  $\mu_Z=6.5$ in panel (a), $\mu_Z=2$ in panel (b).}\label{fig:c}
\end{figure}

Figure \ref{fig:c} presents the optimal consumption w.r.t. the wealth level $x$ under different risk aversion parameter $1-p$ in which the return rate of benchmark process is set to be high (resp. low) when $\mu_Z=6.5$ (resp. $\mu_Z=2$) in the left panel (a) (resp. the right panel (b)). The monotonicity of the optimal feedback consumption with respect to wealth is the same as in the Merton's solution. However, our optimal consumption is only linear in wealth when the risk averse level of the fund manager equals the CRA level ${\rm CRA}(\mu,\sigma,\mu_Z,\sigma_Z)$. When the fund manager has a risk aversion level $1-p>\text{CRA}$, her optimal consumption feedback function is strictly concave in $x$ (see the panel (a) of Figure \ref{fig:c}); while when the fund manager has a risk aversion level $1-p<\text{CRA}$, her optimal consumption feedback function is strictly convex in $x$ (see the panel (b) of Figure \ref{fig:c}). These observations are consistent with the theoretical findings in the claim (ii) of Proposition \ref{lem:prop}.

We also discuss the financial implications for consumption behavior from the perspective of the expected largest shortfall. To examine the relationship between the concavity (or convexity) of the feedback consumption function and the resulting volatility of consumption and expected largest shortfall, we plot these values in Figure \ref{fig:variance-consumption} for different risk aversion parameter $1-p$. In our case, $p_1 = -3.386$. According to Proposition \ref{lem:prop}-(ii), the optimal consumption feedback function is strictly concave in wealth $x$ for $p = -5$ and $p = -4$, but strictly convex for $p = -2$ and $p = -1$. As shown in Figure \ref{fig:variance-consumption}, both the variance of optimal consumption and the expected largest shortfall are smaller in the case of concave feedback consumption functions than those in the case of convex feedback consumption functions. This is consistent with the intuition that a higher risk-averse agent opts for a consumption strategy with lower volatility, resulting in a smaller expected value of the optimal capital injection (expected largest shortfall). Consequently, the expected utility of consumption plays the dominating role in the objective function, which leads to a concave feedback consumption function in the auxiliary state variable $x$.

\begin{figure}[ht]
\centering
  \subfigure[]{
        \includegraphics[width=6cm]{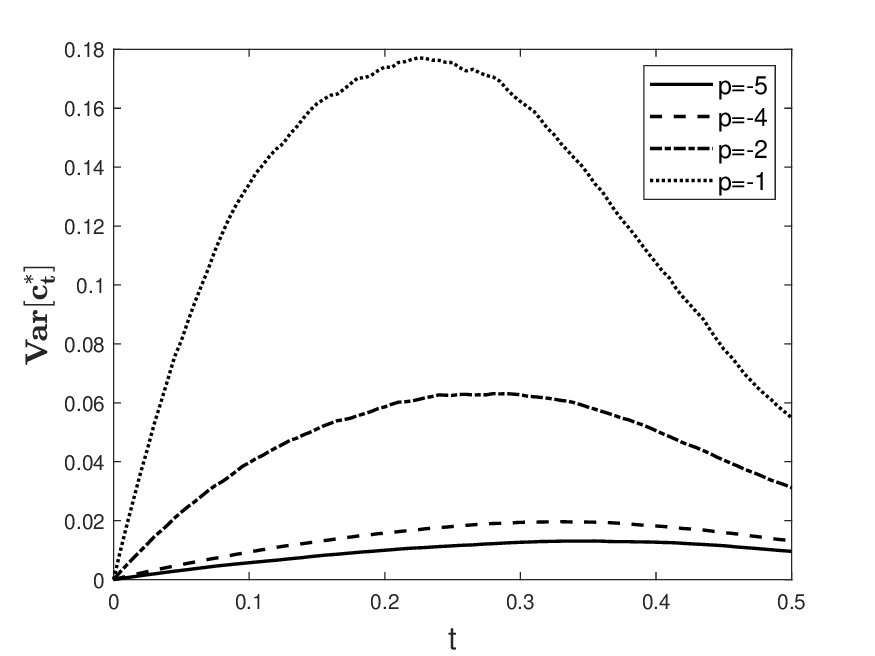}
    }\hspace{-8mm}
  \subfigure[]{
        \includegraphics[width=6cm]{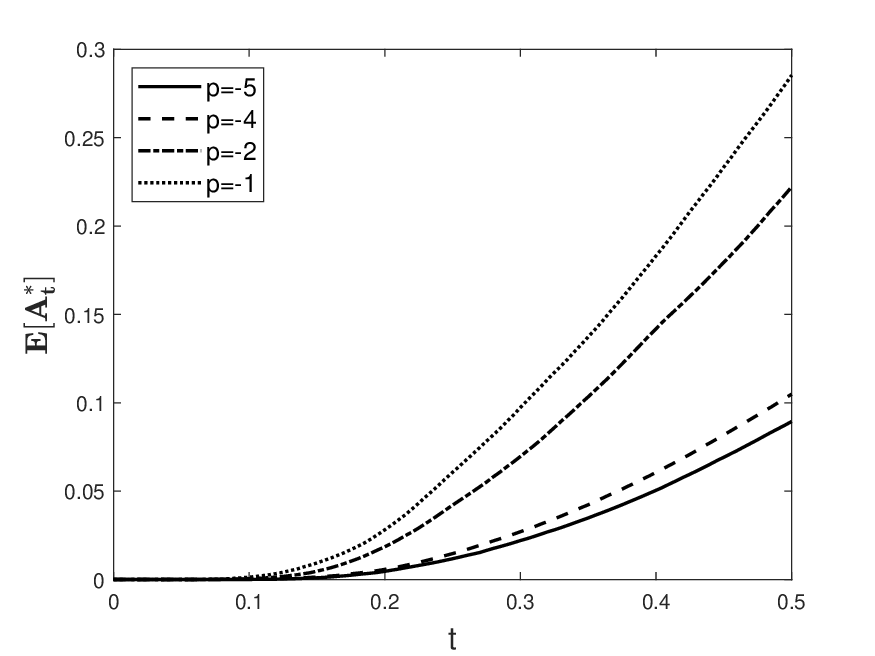}
    }
 \caption{(a): The variance value of optimal consumption. (b):  The expected largest shortfall. The model parameters are set as  $(x,z)=(1,0.5),\rho=5,~\mu=1,~\sigma=1,~\mu_Z=2,~\sigma_Z=1,~\gamma=1,~\beta=1$.}\label{fig:variance-consumption}
\end{figure}

We next conduct some additional numerical examples on sensitivity analysis of the optimal feedback portfolio and consumption w.r.t. the return parameter $\mu_Z$ of the benchmark process and the cost parameter of capital injection $\beta$.
 We illustrate in Figures \ref{fig:muZ} the sensitivity of the optimal feedback portfolio and consumption as well as the expectation of the discounted capital injection with respect to the return parameter $\mu_Z$ in the benchmark dynamics. As expected, the expectation of the discounted capital injection is a decreasing function of the wealth variable $x$. More importantly, being consistent with the intuition, it is shown in Figure \ref{fig:muZ} that when the benchmark process has a higher return, the fund manager will invest more in the risky asset and inject more capital to outperform the targeted benchmark, and meanwhile will strategically reduce the consumption amount due to the pressure of fulfilling the benchmark floor constraint.
\begin{figure}[h]
\centering
  \subfigure[]{
        \includegraphics[width=5.5cm]{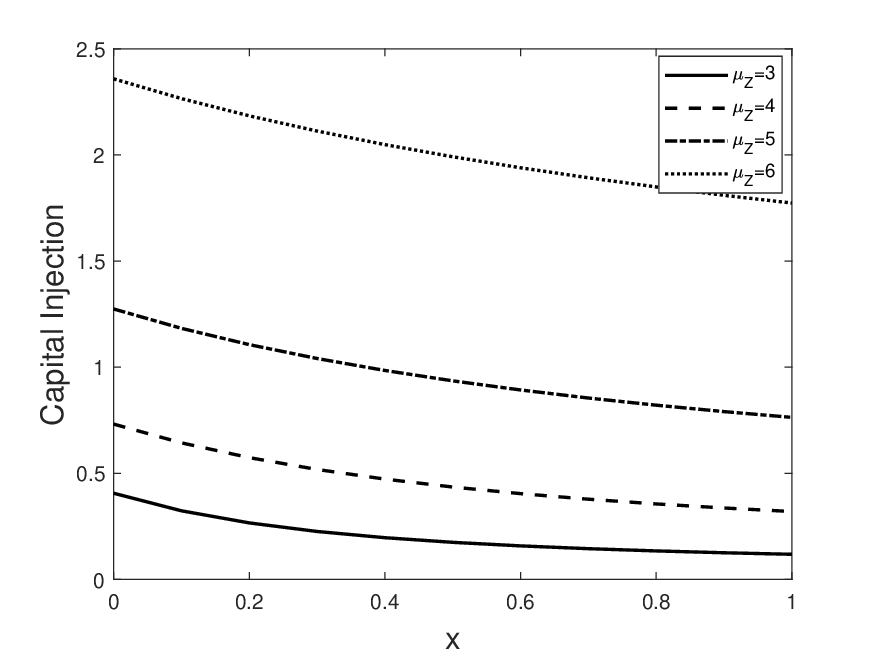}
        \label{fig:value-muZ}
    }\hspace{-8mm}
  \subfigure[]{
        \includegraphics[width=5.5cm]{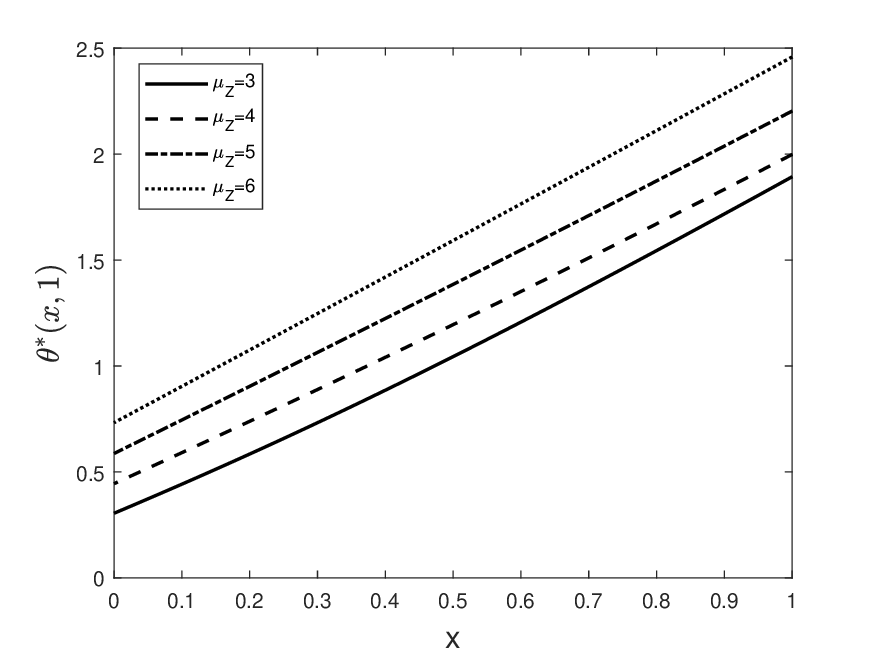}
        \label{fig:theta-muZ}
    }\hspace{-8mm}
  \subfigure[]{
        \includegraphics[width=5.5cm]{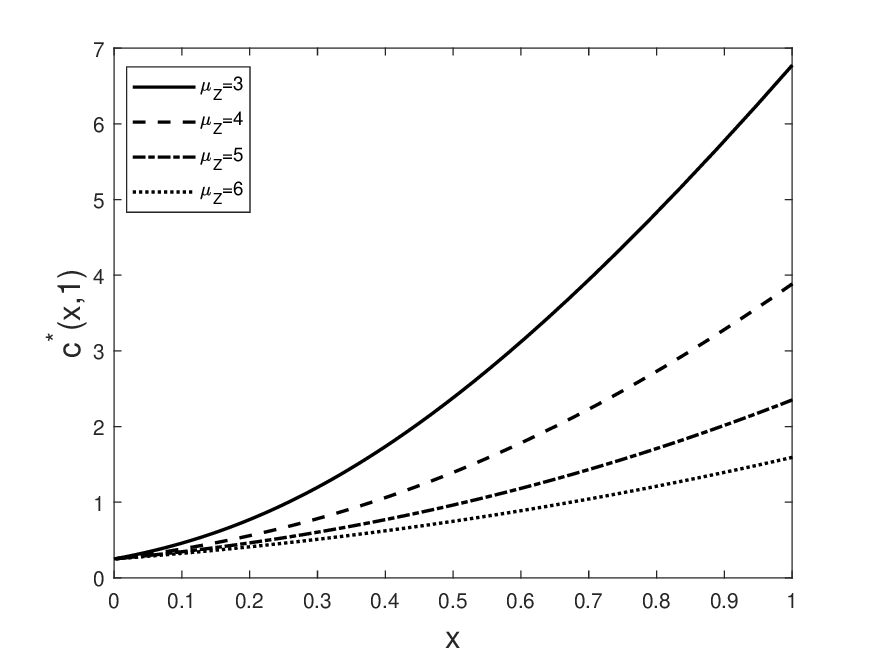}
        \label{fig:c-muZ}
    }
 \caption{(a): The expectation of the total optimal discounted capital injection.  (b): The optimal portfolio $x\to \theta^*(x,1)$. (c): The optimal consumption $x\to c^*(x,1)$. The model  parameters are set as $z=1,~\rho=8,~\mu=1,~\sigma=1,~\sigma_Z=1,\gamma=1, ~p=0.5,~\beta=2$.}\label{fig:muZ}
\end{figure}

 We then plot in Figure \ref{fig:beta} the sensitivity of the optimal feedback portfolio and consumption and the expectation of the capital injection with respect to the cost parameter $\beta$. As $\beta$ increases, the fund manager is more hindered to inject capital as shown in the panel (a) of  Figure \ref{fig:beta}, and hence will strategically suppress the consumption plan to fulfill the benchmark constraint. Meanwhile, from panel (b) of Figure \ref{fig:beta}, it is observed that the fund manager will also reduce the investment in the risky asset, which can be explained by the reduced volatility of the controlled wealth process that may help to avoid unnecessary capital injection in the tracking of the benchmark.

\begin{figure}[h]
\centering
  \subfigure[]{
        \includegraphics[width=5.5cm]{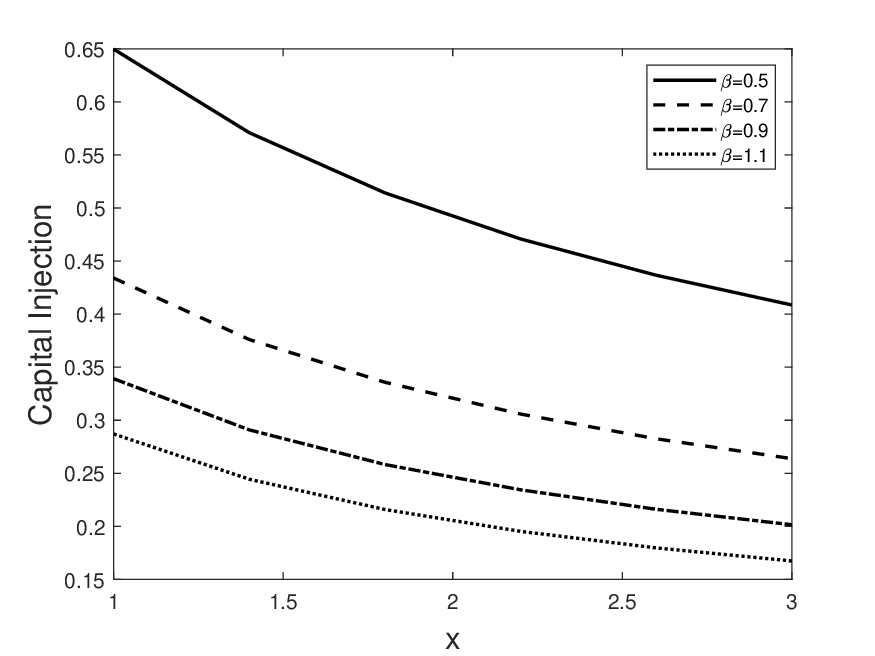}
    }\hspace{-8mm}
  \subfigure[]{
        \includegraphics[width=5.5cm]{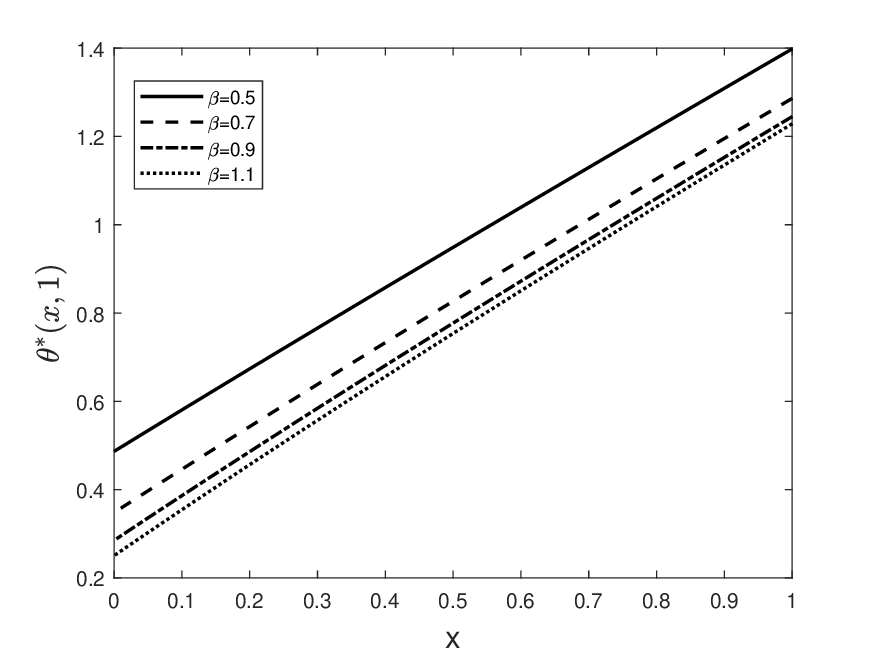}
    }\hspace{-8mm}
  \subfigure[]{
        \includegraphics[width=5.5cm]{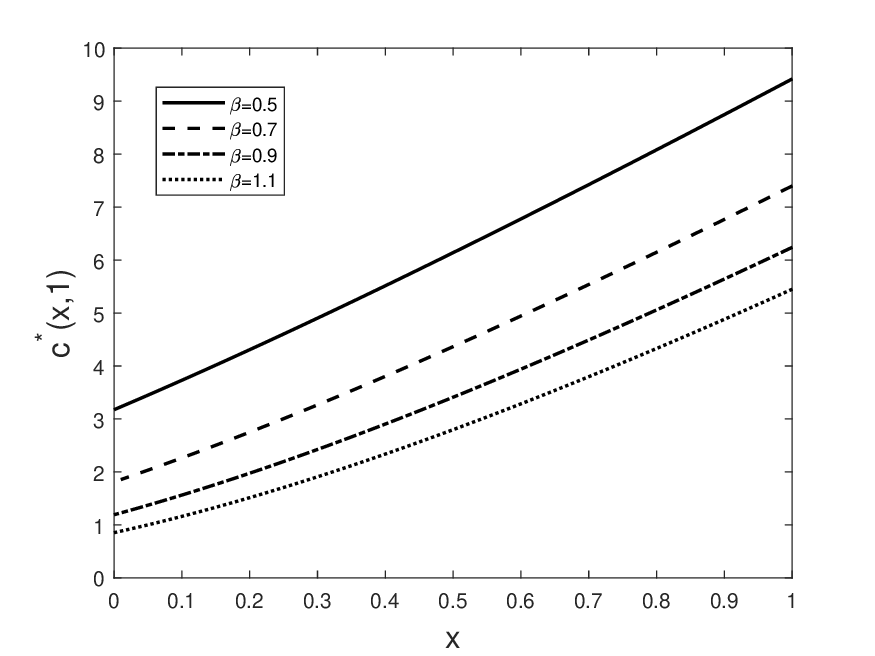}
    }
 \caption{\small (a):  The expectation of the total optimal discounted capital injection.  (b): The optimal portfolio $x\to \theta^*(x,1)$. (c): The optimal consumption $x\to c^*(x,1)$. The model parameters are set as  $z=1,~\rho=5,~\mu=0.5,~\sigma=1,~\mu_Z=2,~\sigma_Z=1,~\gamma=1,~p=0.4$.}\label{fig:beta}
\end{figure}

Next, with zero benchmark process ($Z_t \equiv 0$), we present simulated sample paths in Figure \ref{fig:simulation} to illustrate five key processes: the optimal wealth $V_t^*$, the optimal capital injection $A_t^*$, the auxiliary state process $X_t^*$, and the optimal portfolio-consumption strategies $(\theta_t^*, c_t^*)$ for $t\in[0,1]$. Two important observations can be drawn: 
\begin{itemize}
    \item[(i)] We see from Figure \ref{fig:simulation}  that the wealth process $V^*$ coincides exactly with the auxiliary state process $X^*$ up to $t_0$,  the first hitting time of $V^*$ at zero. During this phase ($t\leq t_0$),  no capital injection is needed, and both processes exhibit identical co-movement with the optimal portfolio and consumption strategies. However, after $V^*$ reaches the zero threshold, the trajectories of $(\theta^*, c^*)$ start to follow the pattern of $X^*$ instead of $V^*$ (see the right panel (b) of Figure \ref{fig:simulation}). This phenomenon again reflects the fact that our optimal controls can only be expressed as the feedback form in terms of the auxiliary state process $X_t$ but not in terms of the original wealth process $V_t$. Indeed, we have the relationship (when benchmark is zero) that
\begin{align*}
X_t = x + (V_t^{\theta,c} - \mathrm{v}) + \sup_{s\leq t}\left(-x - (V_s^{\theta,c} - \mathrm{v})\right)^+,
\end{align*}
which shows that the optimal control $(\theta^*,c^*)$
in Corollary \ref{coro:optimal-control} actually has the path-dependent structure in terms of the wealth process $V_t$ that will make the decision making intractable based on the direct study of the control problem using the original wealth process. This justifies the main advantage of working with the auxiliary state process $X_t$ in the present paper, which significantly simplifies the problem and enables us to numerically illustrate some quantitative properties of the optimal control $(\theta^*,c^*)$ in feedback form.

\item[(ii)] Figure \ref{fig:simulation}-(b)  reveals that the optimal portfolio $\theta^*_t$ and consumption $c^*_t$ naturally satisfies the positive constant subsistence level in the life-cycle. The subsistence level occurs whenever the auxiliary state process $X_t^*$ reaches zero. This observation aligns with our theoretical analysis in Remark \ref{rem0}, which characterizes control behavior when the auxiliary state level $X^*=0$. Specifically, with zero benchmark  ($Z_t \equiv 0$), it follows from \eqref{eq:optimal-x0} that
\begin{align}\label{eq:optimal-x0-z0}
 \theta^*(0,0) = \frac{\mu}{\sigma^2} \frac{1-p}{\rho(1-p)-\alpha p}\beta^{\frac{1}{p-1}} > 0,\quad c^*(0,0) = \beta^{\frac{1}{p-1}} > 0.
\end{align}

\end{itemize}
The observation-(ii) essentially relates to the literature on subsistent portfolio constraint or consumption constraint, which has been documented in many empirical and economic studies. \cite{Behr13} demonstrates that minimum-variance portfolios with explicit allocation floors yield significantly lower out-of-sample volatility compared to conventional unconstrained strategies. The concept of constant consumption floors, commonly termed subsistence consumption, has been rigorously analyzed in development economics. \cite{Chatterjee1999} develops a theoretical framework quantifying the impact of minimum consumption requirements on wealth distribution dynamics and economic growth rates, calibrated using household-level data from rural Indian communities. Their results show that the effect of minimum consumption requirement may be quantitatively important. \cite{Alvarez-Pelaez05} investigates wealth inequality propagation mechanisms in a calibrated one-sector growth model where household consumption cannot fall below a positive level each period. This model is calibrated to match some key aggregate statistics of the U.S. economy. Other relevant studies can be found in \cite{Zimmerman03}, \cite{Jensen08}, \cite{Achury12} among others.

Notably, most theoretical studies along this direction enforce the subsistence constraint on the admissible control to achieve the portfolio minimum floor (see \citealt{Shirakawa1994}; \citealt{Best2000}) or the subsistent consumption behavior (see \citealt{Sethi1992};  \citealt{Shin2011}; \citealt{Kim18}). Accordingly, the initial wealth therein needs to stay above some threshold to ensure that the problem is well-defined. Our model differs fundamentally from them -- the positive lower bounds on optimal portfolio and consumption naturally arise from the tracking formulation by allowing capital injection. Our formulation allows all levels of wealth in fund management to attain the subsistent constant consumption level, similar to the previous studies. Moreover, Equation \eqref{eq:optimal-x0-z0} also explicitly captures the minimum positive level $(\theta^*(0,0),c^*(0,0))$ in terms of the cost parameter $\beta$, thereby making the calibration of the cost parameter $\beta$ possible by observing the real-life data of the subsistent consumption behavior from the fund manager.

\begin{figure}[ht]
\centering
  \subfigure[]{
        \includegraphics[width=6cm]{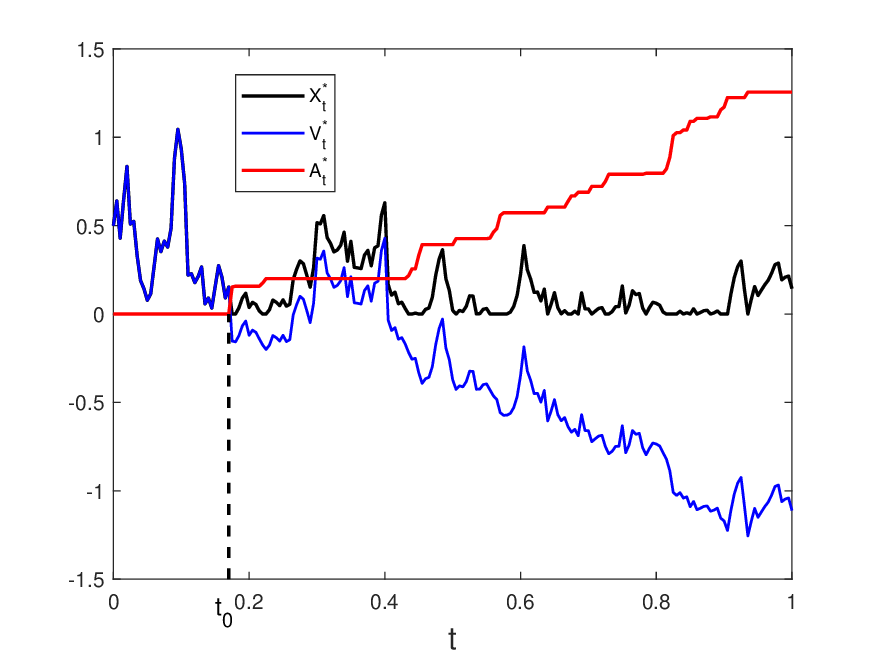}
    }\hspace{-8mm}
  \subfigure[]{
        \includegraphics[width=6cm]{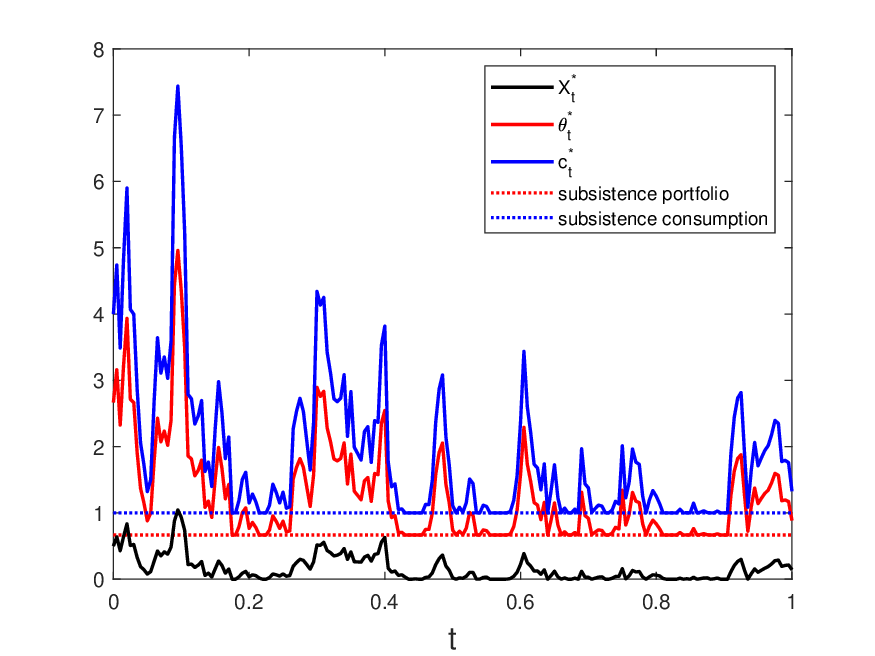}
    }
 \caption{\small (a): Sample paths of $t\to V_t^*$, $t\to X_t^*$ and $t\to A_t^*$ via Monte Carlo simulation. (b) Sample paths of $t\to \theta_t^*$ and $t\to c_t^*$  via Monte Carlo simulation.  The model parameters are set as $(x,z)=(0.5,0)$,~ $\rho=5,~\mu=2,~\sigma=1,~\mu_Z=0,~\sigma_Z=0,~\gamma=1,~\beta=1,~p=0.5$.}\label{fig:simulation}
\end{figure}

Figure \ref{fig:simulation-Merton} presents a comparative analysis of mean and  variance of  optimal wealth processes as well as the mean of optimal portfolio and consumption processes between our problem and the Merton problem.  Evidently, the injected capital supports the more aggressive portfolio and consumption behavior. Panels (c) and (d) of Figure \ref{fig:simulation-Merton} show that the expected optimal portfolio and consumption processes in our problem are obviously higher than the ones in the Merton problem. On the other hand, the mean of wealth process $\Ex[V^*]$ in our problem may become negative and under-performs $\Ex[V^{\text{Mer}}]$ across all horizon, as shown in panel (a) of Figure \ref{fig:simulation-Merton}. Meanwhile, the variance of wealth process $\text{Var}[V^*]$ in our problem also consistently exceeds $\text{Var}[V^{\text{Mer}}]$, as illustrated in panel (b) of Figure \ref{fig:simulation-Merton}.

\begin{figure}[ht]
\centering
  \subfigure[]{
        \includegraphics[width=5.5cm]{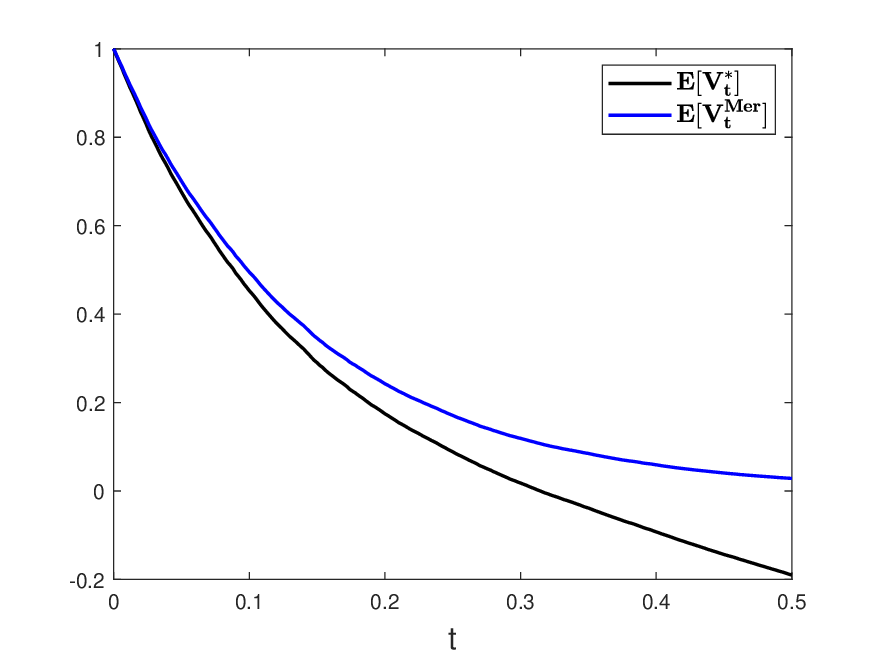}
    }\hspace{-8mm}
  \subfigure[]{
        \includegraphics[width=5.5cm]{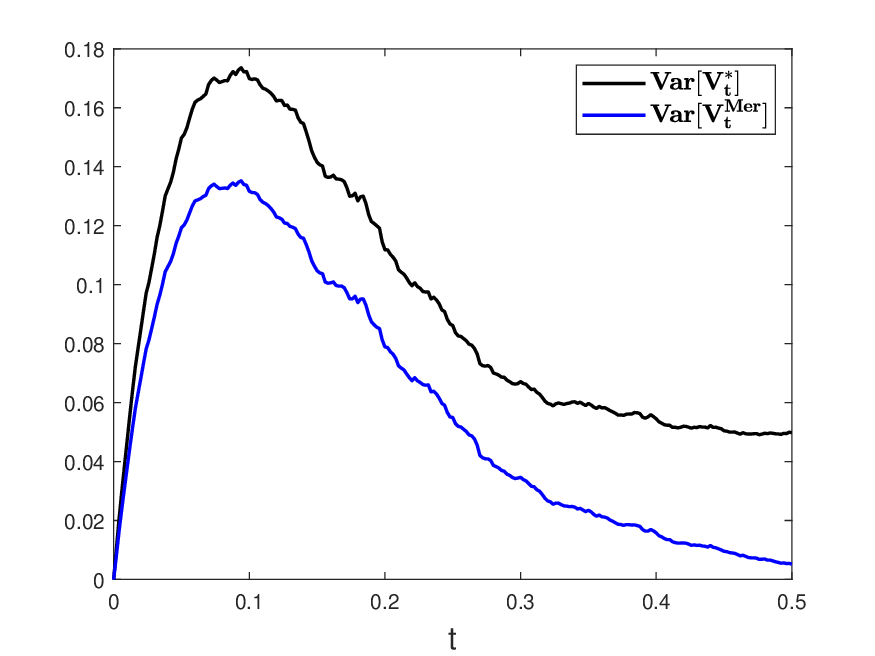}
    }
    
      \subfigure[]{
        \includegraphics[width=5.5cm]{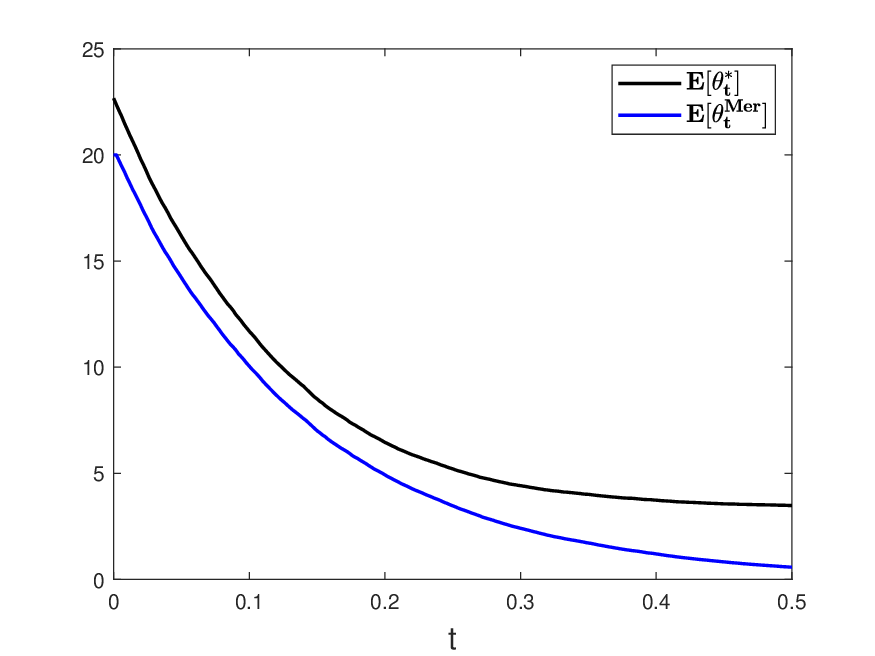}
    }\hspace{-8mm}
  \subfigure[]{
        \includegraphics[width=5.5cm]{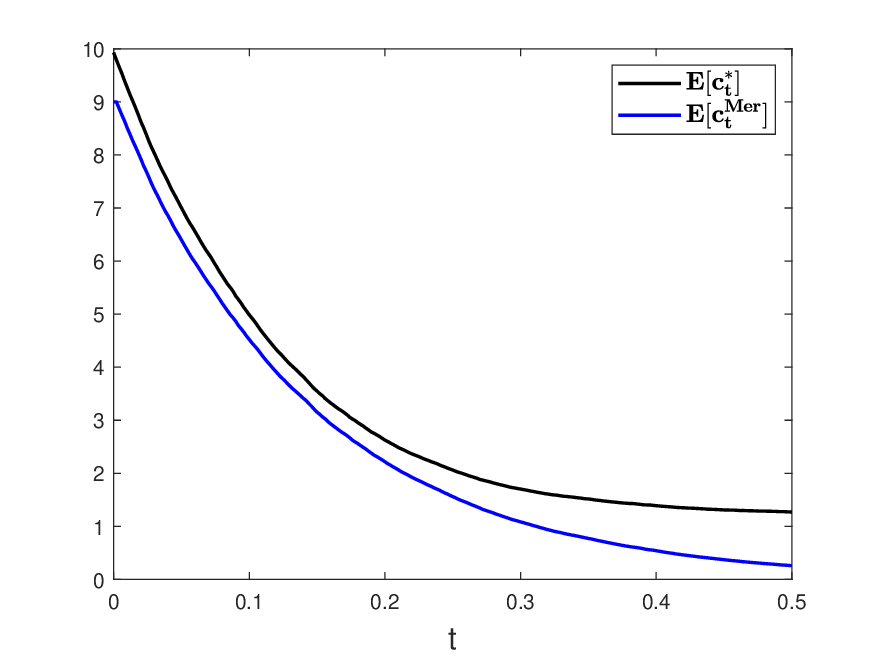}
    }
 \caption{\small (a): The mean value of optimal wealth process under our optimal tracking problem and the Merton problem . (b) The variance value of optimal wealth process under our optimal tracking problem and the Merton problem.  (c): The mean value of optimal portfolio process under our optimal tracking problem and the Merton problem . (d) The mean value of optimal consumption process under our optimal tracking problem and the Merton problem . The model parameters are set as $(x,z)=(1,0.5)$,~ $\rho=5,~\mu=0.1,~\sigma=0.1,~\mu_Z=0.2,~\sigma_Z=0.1,~\gamma=1,~\beta=1,~p=0.5$.}\label{fig:simulation-Merton}
\end{figure}

Next, we mainly discuss the impact of the risk aversion attitude on the optimal investment and consumption strategies via both numerical and empirical analysis. For our empirical results, we choose the S\&P 500 index as the benchmark process and select the Sony Group Corporation (SONY) as the risky asset\footnote{The data is retrieved from \url{https://finance.yahoo.com} with the period of January 1, 2023 to  January 2, 2024.}. Based on the historical data, we calibrate the return rate and volatility parameters in both our Black-Scholes stock price model \eqref{stockSDE} and the Black-Scholes benchmark process \eqref{eq:Zt} by using the approach of maximum likelihood estimation (MLE) (c.f. \citealt{Brigo2009}). Here, we use the daily returns from January 1, 2023 through December 31, 2023. Denote respectively by $S_0,S_1,\ldots,S_n$ and $Z_0,Z_1,\ldots,Z_n$, with $n=249$ (note that there were $250$ trading days during the time period). We can define the sequence $(Y_i)_{i\geq 1}$ on the log stock prices that
\begin{align*}
Y_i:=\ln(S_i)-\ln(S_{i-1}),\quad \forall i=1,\ldots,n.
\end{align*}
For the sample average of the log stock price sequence $(Y_i)_{i\geq 1}$ given by $\overline{Y}:=\frac{1}{n}\sum_{i=1}^n Y_i$, we estimate the return parameter $\mu$ and the volatility parameter $\sigma$ of the stock by using the following estimators with $\Delta t=1$ (day):
\begin{align*}
\widehat{\mu}=\frac{\overline{Y}}{\Delta t}+\frac{1}{2}\widehat{\sigma}^2,\quad \widehat{\sigma}=\sqrt{\frac{1}{n\Delta t} \sum_{i=1}^n (Y_i-\overline{Y})^2}.
\end{align*}
In a similar fashion, we can estimate the return rate parameter $\mu_Z$ and the volatility parameter $\sigma_Z$ of the benchmark process by using the following estimators with $\Delta t=1$ (day):
\begin{align*}
\widehat{\mu}_Z=\frac{\overline{Y}^Z}{\Delta t}+\frac{1}{2}\widehat{\sigma}_Z^2,\quad \widehat{\sigma}_Z=\sqrt{\frac{1}{n\Delta t} \sum_{i=1}^n \left(Y_i^Z-\overline{Y}^Z\right)^2}.
\end{align*}
Here, the sample average $\overline{Y}^Z:=\frac{1}{n}\sum_{i=1}^n Y_i^Z$, where $(Y_i^Z)_{i\geq1}$ is  the log benchmark sequence defined by
\begin{align*}
 Y_i^Z:=\ln(Z_i)-\ln(Z_{i-1}),\quad \forall i=1,\ldots,n.    
\end{align*}
By implementing the maximum likelihood estimation (MLE) (c.f. \citealt{Brigo2009}) with the historical data, we then obtain the estimated values of the parameters as shown in Table \ref{table:parameters}:
\begin{table}[h]   
\begin{center}   
\caption{Estimated parameters of the risky asset and benchmark process using MLE.}  
\label{table:parameters} 
\begin{tabular}{ m{6cm}<{\centering} |m{6cm}<{\centering}}   
\hline\hline Estimated parameters & Estimated values\\
\hline $\widehat{\mu}$ & $9.7399\times 10^{-4}$\\   
\hline  $\widehat{\sigma}$ &0.0158  \\ 
\hline   $\widehat{\mu}_Z$ & $9.2137\times 10^{-4}$ \\  
\hline   $\widehat{\sigma}_Z$ &0.0082 \\  
\hline\hline
\end{tabular}   
\end{center}   
\end{table}
With the estimated parameters above, Table \ref{table:portfolio} then presents the implied optimal portfolio strategy on January 2, 2024 (the first trading day in 2024) and Table \ref{table:consumption} shows the implied optimal consumption strategy, where the price of the initial level of benchmark process $z=4742.83$, the risk aversion level $1-p$ varies in the set $\{5,4,3,2,0.9,0.85,0.8,0.75\}$ and the initial wealth level $x\in\{1,4,10\}$. We set the discount rate $\rho=1$ and the cost parameter of capital injection $\beta=5$.
\begin{table}[h]   
\begin{center}   
\caption{The optimal (feedback) portfolio strategy $\theta^*(x,z)$ ($\times 10^{3}$).}  
\label{table:portfolio} 
\begin{tabular}{ m{1.2cm}<{\centering} m{1.2cm}<{\centering}m{1.2cm}<{\centering}m{1.2cm}<{\centering}m{1.2cm}<{\centering}m{1.2cm}<{\centering}m{1.2cm}<{\centering}m{1.2cm}<{\centering}m{1.2cm}<{\centering}}   
\hline\hline   $1-p$ &5 &4 &3 &2 &0.9 &0.85 &0.8 &0.75\\   
\hline $x=10$ &2.4820 &2.4822 &2.4825 &2.4837 &2.5032 &2.5093 &2.5183 &2.5324\\
  $x=4$ &2.4815 &2.4814 &2.4813 &2.4813 &2.4822 &2.4824 &2.4825 &2.4828\\
   $x=1$ &2.4800 &2.4798 &2.4795 &2.4790 &2.4779 &2.4779 &2.4778 &2.4777\\
   \hline\hline 
\end{tabular}   
\end{center}   
\end{table}

\begin{table}[h]   
\begin{center}   
\caption{The optimal (feedback) consumption rate strategy $c^*(x,z)$.}  
\label{table:consumption} 
\begin{tabular}{ m{1.2cm}<{\centering} m{1.2cm}<{\centering}m{1.2cm}<{\centering}m{1.2cm}<{\centering}m{1.2cm}<{\centering}m{1.2cm}<{\centering}m{1.2cm}<{\centering}m{1.2cm}<{\centering}m{1.2cm}<{\centering}}   
\hline\hline   $1-p$ &5 &4 &3 &2 &0.9 &0.85 &0.8 &0.75\\   
\hline $x=10$ &1.0120 &1.0319  &1.0874 &1.3088  &6.0504  &7.5809  & 9.8646 &13.4391\\
  $x=4$ &0.9492 &0.9354 &0.9170 &0.9022 &1.0621 &1.0937 &1.1319 &1.1783\\
   $x=1$ &0.8011 &0.7558 & 0.6843   & 0.5601 &0.2773 &0.2578 &0.2375 & 0.2166\\
   \hline\hline 
\end{tabular}   
\end{center}   
\end{table}

Recall that the solution in \cite{Merton1971} suggests that a more risk-averse agent would invest less in the risky asset, see also  \cite{Broell07} and \cite{Xia11} for their theoretical conclusions under general utilities. However, the empirical study in \cite{Wang2021} illustrates that a larger risk aversion may induce higher investment if the proportion of less risk-averse investors in the population is sufficiently small. Empirical results in \cite{Chacko2005} also reveal that the consumption can either increase or decrease with respect to risk aversion, depending on the elasticity of intertemporal substitution of consumption. In the present paper, we can show by Figure \ref{fig:optimal-p} on various plots of optimal feedback functions that the risk taking induced by the capital injection also leads to the similar phenomenon that the optimal (feedback) portfolio $\theta^*(x,z)$ and the optimal (feedback) consumption rate $c^*(x,z)$ are not necessarily monotone in the risk aversion parameter $1-p$, which actually depend on different auxiliary state variable regimes. From our empirical results in Table \ref{table:portfolio} and Table \ref{table:consumption}, we have the consistent observations that the optimal portfolio and optimal consumption might be increasing or decreasing in $1-p$ depending on different auxiliary state variable levels.

\begin{figure}[h]
\centering
  \subfigure[]{
        \includegraphics[width=6cm]{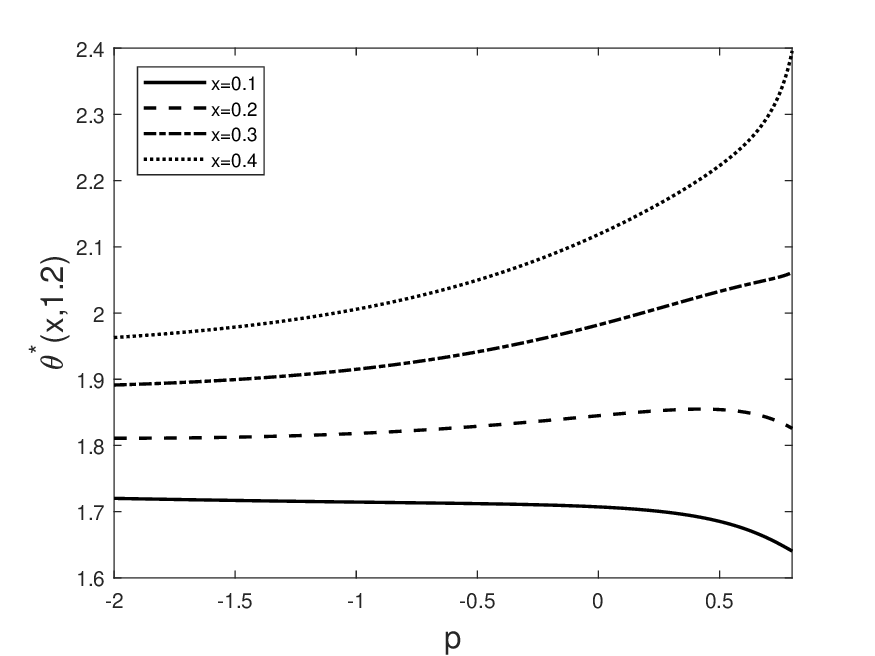}
    }\hspace{-4mm}
  \subfigure[]{
        \includegraphics[width=6cm]{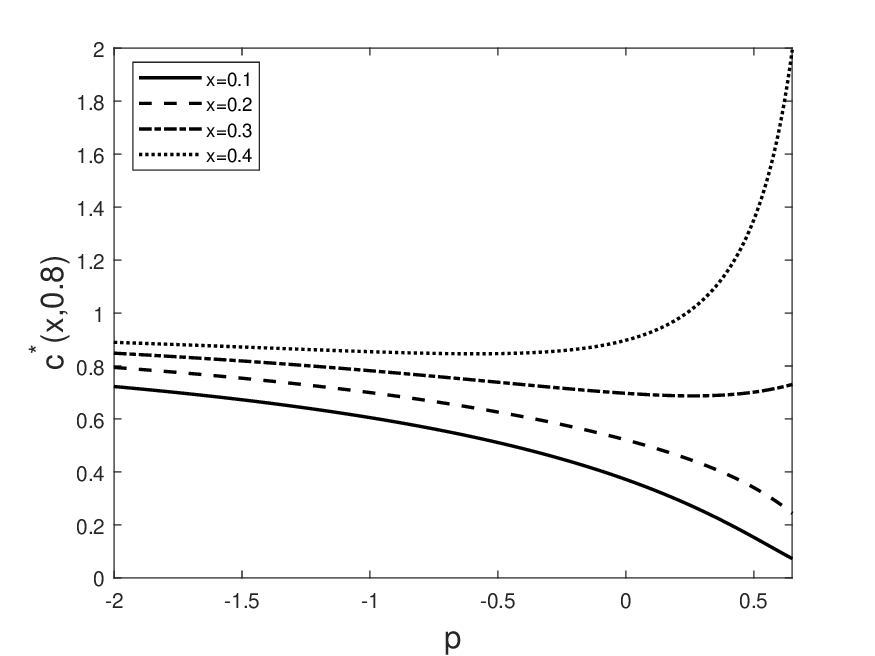}
    }
 \caption{(a): The optimal (feedback) portfolio $p\to \theta^*(x,1.2)$ as the risk aversion parameter $p$ varies. The model parameters are set as  $\rho=5,~\mu=1,~\sigma=1,~\mu_Z=3,~\sigma_Z=1,~\gamma=1,~\beta=2$. (b): The optimal (feedback) consumption rate $p\to c^*(x,0.8)$ as the risk aversion parameter $p$ varies. The model parameters are set as  $\rho=5,~\mu=1,~\sigma=1,~\mu_Z=2,~\sigma_Z=1,~\gamma=1,~\beta=4$.}\label{fig:optimal-p}
\end{figure}

Indeed, in our proposed new formulation, the risk aversion is distorted by the incentivized risk-taking from the possible capital injection to fulfill the benchmark constraint. When the auxiliary state variable level is relatively large (but not too large), the less risk averse (i.e. as $p$ tends to $1$) fund manager would invest and consume more than the highly risk averse fund manager. This can be explained by the fact that the low risk averse fund manager would take more risk by investing in the risky asset when the wealth level is healthy, which in turn leads to a higher consumption when the benchmark tracking can be maintained. However, when the auxiliary state variable level is very low, the less risk averse fund manager will put more concern on the cost of capital injection from its trade-off with the expected utility. To avoid high capital injection, the low risk averse fund manager will strategically reduce the portfolio and consumption plan to maintain the benchmark tracking.
\begin{figure}[h]
\centering
  \subfigure[]{
        \includegraphics[width=5.5cm]{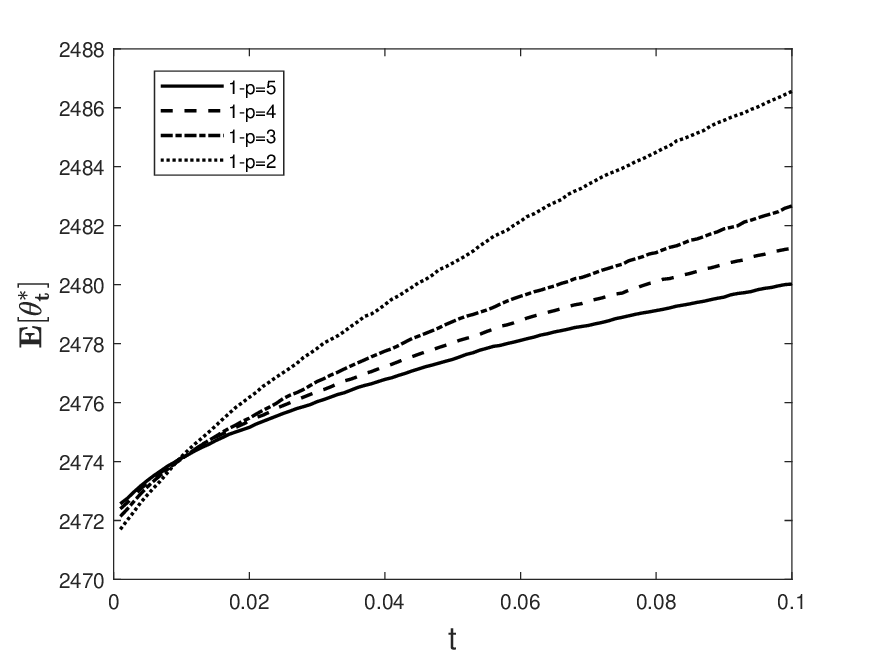}
    }\hspace{-8mm}
  \subfigure[]{
        \includegraphics[width=5.5cm]{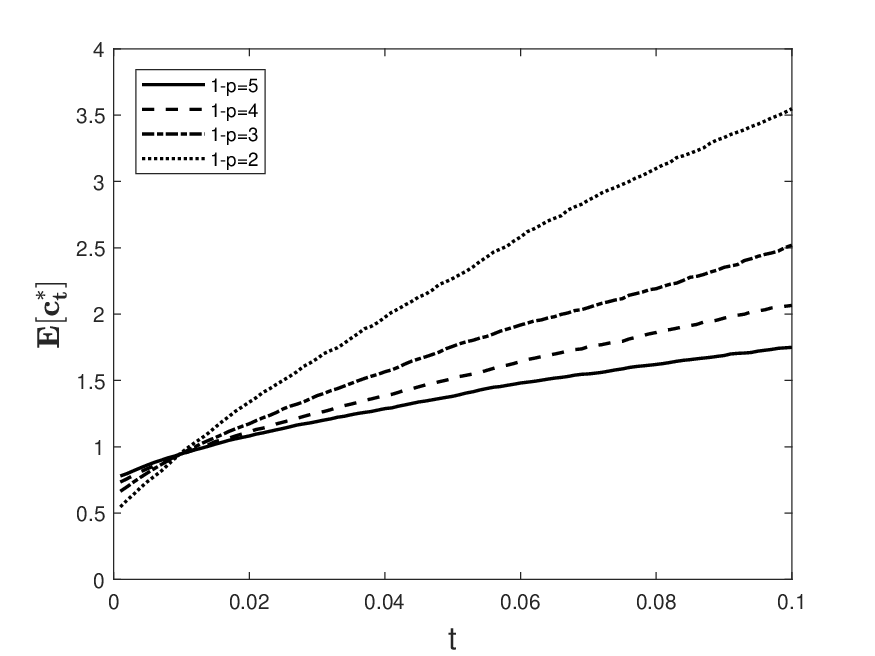}
    }

 \subfigure[]{
        \includegraphics[width=5.5cm]{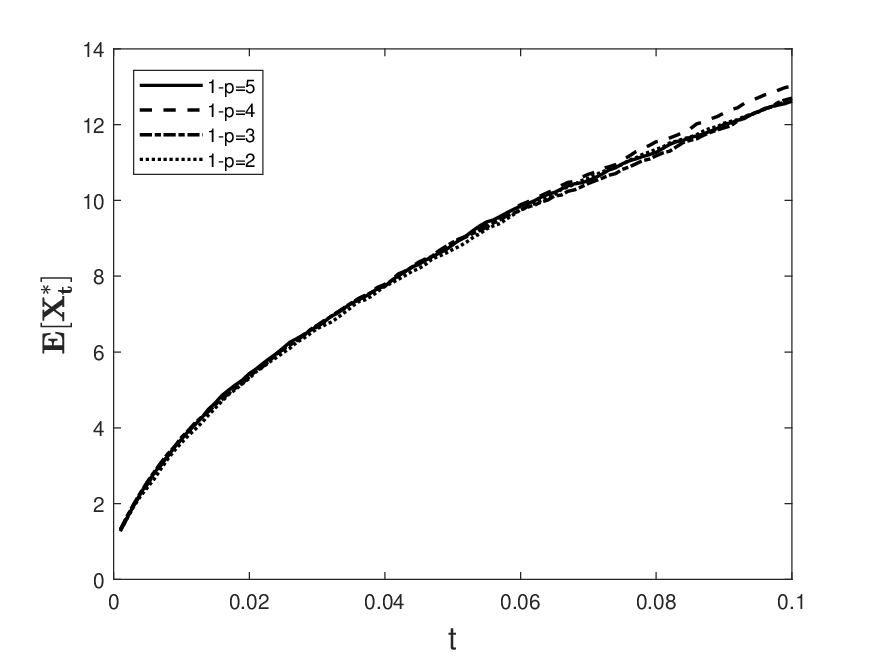}
    }
     \subfigure[]{
        \includegraphics[width=5.5cm]{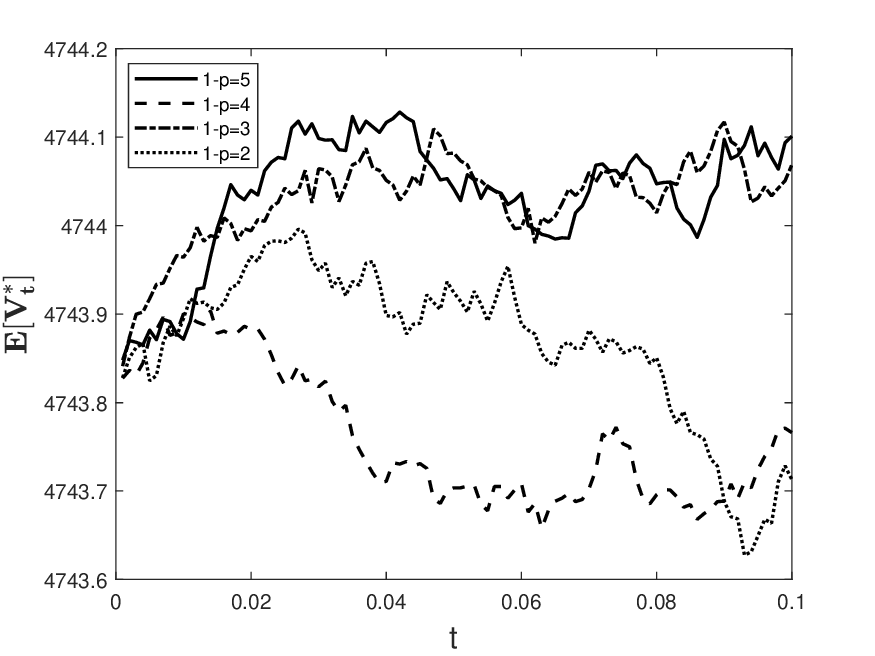}
    }
\caption{(a): The mean value of optimal portfolio $t\to \Ex[\theta^*_t]$. (b): The mean value of optimal consumption $t\to \Ex[c^*_t]$ under different choices of  risk
aversion level $1-p$. (c):  The mean value of auxiliary process $t\to \Ex[X^*_t]$. (d):  The mean value of wealth process $t\to \Ex[V^*_t]$. The model parameters are set as  $\rho=1,~\mu=9.7398\times 10^{-4},~\sigma=0.0158,~\mu_Z=9.2137\times 10^{-4},~\sigma_Z=0.0082,~\gamma=1,~\beta=5$. The initial level of wealth process, auxiliary state process and benchmark process are set as $(\mathrm{v},x,z)=(4743.83,1,4742.83)$. }\label{fig:optimal-x-1}
\end{figure}

To echo with the previous interesting observations, we also simulate the mean value of the optimal portfolio $t\to \Ex[\theta^*_t]$, the mean value of the optimal consumption rate $t\to \Ex[c^*_t]$, the mean value of auxiliary state process $t\to \Ex[X^*_t]$ and the mean value of wealth process $t\to \Ex[V^*_t]$ as functions of time $t$ via the Monte Carlo method (see Figure \ref{fig:optimal-x-1}). The model parameters are set to be the empirically estimated values in Table \ref{table:parameters} and we focus on the case with a low initial wealth level $x=1$. The panel (c) of Figures \ref{fig:optimal-x-1} shows that the expected auxiliary state  level is increasing along the time $t$ while the expected wealth level can be either increasing or decreasing. From panels (a) and (b), at the early stage of the investment horizon, both mean values of the optimal portfolio and consumption are increasing in the risk aversion parameter $1-p$, due to the fact the expected auxiliary state level is relatively low. Later on, when time $t$ is sufficiently large and the accumulated expected auxiliary state  becomes large, the monotonicity of the mean values of the optimal portfolio and consumption with respect to the risk aversion parameter $1-p$ overturns, which perfectly matches with previous observations in Table \ref{table:portfolio} and Table \ref{table:consumption} as well as the plots of optimal feedback functions in Figure \ref{fig:optimal-p}.

\section{Proofs}\label{sec:proof}

This section collects all proofs of auxiliary and main results in previous sections.

\begin{proof}[Proof of Lemma \ref{lem:equivalence}]

For the first claim, let $x=(\mathrm{v}-z)^+$. In view of \eqref{eq_orig_pb} and \eqref{eq:u0}, it suffices to show that
\begin{align}\label{eq:equivalence-1}
v(x,z)=\sup_{(\theta,c)\in\mathbb{U}^{\rm r}}\ \Ex\left[ \int_0^{\infty} e^{-\rho t} U(c_t)dt-\beta\int_0^{\infty} e^{-\rho t} d\left(0\vee \sup_{s\leq t}(Z_{s}-V_{s}^{\theta,c})\right)\right].
\end{align}
For any $(\theta,c)\in\mathbb{U}^{\rm r}$, if $\mathrm{v}\leq z$, i.e., $x=0$, then we have
\begin{align*}
L_t^X:=x\vee \sup_{s\leq t}D_s-x=0\vee \sup_{s\leq t}D_s=\sup_{s\leq t}D_s=\sup_{s\leq t}\left(Z_t-V_t^{\theta,c}\right)+\mathrm{v}-z,
\end{align*}
which yields that $dL_t^X=d\sup_{s\leq t}(Z_t-V_t^{\theta,c})=d(0\vee \sup_{s\leq t}(Z_{s}-V_{s}^{\theta,c}))$. On the other hand, if $\mathrm{v}> z$, i.e., $x=\mathrm{v}-z$, note that
\begin{align*}
L_t^X:=x\vee \sup_{s\leq t}D_s-x=0\vee \left(\sup_{s\leq t}D_s-x\right)=0\vee \left(\sup_{s\leq t}Z_t-V_t^{\theta,c}\right).
\end{align*}
Thus, the first claim holds.

For the second claim, with $(x,z)\in\R_+^2$, let $L^{x,z,\theta,c}=(L_t^{x,z,\theta,c})_{t\geq0}$ be the local time process of $X^x$ with $(X_0^x,Z_0^z)=(x,z)$ and $L_0^{x,z,\theta,c}=0$ under the control strategy $(\theta,c)\in\mathbb{U}^r$. Using the solution representation of ``the Skorokhod problem", we obtain that, for all $s\geq0$,
{\small\begin{align}\label{eq:skorokhod}
\begin{cases}
\displaystyle L_s^{x,z,\theta,c}=\sup_{\ell\in[0,s]}\left(x+\int_0^{\ell}\theta_{r}^{\top}\mu dr+\int_0^{\ell}\theta_{r}^{\top}\sigma dW_{r}  -\int_0^{\ell} c_{r} dr -\int_0^t \mu_Z Z_{\ell}^zd{\ell}-\int_0^t \sigma_Z Z_{\ell}^zdW_{\ell}^{\gamma}\right)^-,\\[1.2em]
\displaystyle Z_s^{z}=z\exp\left[\left(\mu_Z-\frac{1}{2}\sigma_Z^2\right)s+\sigma_Z W^{\gamma}_s\right].
\end{cases}
\end{align}}
By this, we have $x\to L_s^{x,z,\theta,c}$ is non-increasing. Moreover,  it holds that, $\Px$-a.s.
\begin{align}\label{eq:LiphatL0}
\sup_{s\geq0}\left|L_s^{x_1,z,\theta,c}-L_s^{x_2,z,\theta,c}\right|\leq |x_1-x_2|,\quad \forall (x_1,x_2)\in\R_+^2.
\end{align}
For any $\epsilon>0$, denote by $(\theta^{\epsilon}(x,z),c^{\epsilon}(x,z)))\in\mathbb{U}^{\rm r}$ the $\epsilon$-optimal control strategy for \eqref{eq:u0} with the initial state $(x,z)$. In other words, it holds that
\begin{align}\label{epsiloncontrol}
v(x,z)\leq J(x,z;\theta^{\epsilon}(x,z),c^{\epsilon}(x,z))+\epsilon.
\end{align}
Then, for any $x_1>x_2\geq0$, we have from \eqref{epsiloncontrol} that
\begin{align}\label{eq:diffVy1y2}
v(x_1,z)-v(x_2,z) &\geq  J(x_1,z;\theta^{\epsilon}(x_2,z),c^{\epsilon}(x_2,z))-J(x_2,z;\theta^{\epsilon}(x_2,z),c^{\epsilon}(x_2,z))-\epsilon\nonumber\\
&=-\beta\Ex\left[\int_0^{\infty}  e^{-\rho s}d(L_s^{x_1,z,\epsilon}-L_s^{x_2,z,\epsilon})\right]-\epsilon,
\end{align}
where $L_s^{x,z,\epsilon}$ for $s\in\R_+$ is the local time process with $X^x_0=x$, $Z^z_0=z$ and $L_0^{x,z,\epsilon}=0$ under the $\epsilon$-optimal control strategy $(\theta^{\epsilon}(x,z),c^{\epsilon}(x,z))$. Thus, integration by parts yields that, for any $T>0$,
\begin{align*}
\int_0^{T} e^{-\rho s}dL_s^{x,z,\epsilon} = e^{-\rho T}L_{T}^{x,z,\epsilon} + \rho\int_0^{T}  L_s^{x,z,\epsilon}e^{-\rho s}ds.
\end{align*}
Using the fact $L_s^{x_1,z,\epsilon}-L_s^{x_2,z,\epsilon}\leq0$ whenever $x_1>x_2\geq0$, it follows from  Monotone Convergence Theorem (MCT) that
\begin{align}\label{eq:diffLref}
&\Ex\left[\int_0^{\infty} e^{-\rho s}d(L_s^{x_1,z,\epsilon}-L_s^{x_2,z,\epsilon})\right]=\lim_{T\to \infty}\Ex\left[\int_0^{T} e^{-\rho s}d(L_s^{x_1,z,\epsilon}-L_s^{x_2,z,\epsilon})\right]\nonumber\\
&\quad=\lim_{T\to \infty}\left\{\Ex\left[e^{-\rho T}(L_{T}^{x_1,z,\epsilon}-L_{T}^{x_2,z,\epsilon})\right]+ \rho \Ex\left[\int_0^{T} e^{-\rho s}(L_s^{x_1,z,\epsilon}-L_s^{x_2,z,\epsilon})ds\right]\right\}\leq 0.
\end{align}
Hence, we have from \eqref{eq:diffVy1y2} that $v(x_1,z)-v(x_2,z)\geq -\epsilon$. As $\epsilon>0$ is arbitrary, we get $v(x_1,z)\geq v(x_2,z)$. This concludes that $x\to v(x,z)$ is non-decreasing. On the other hand, it follows from \eqref{eq:LiphatL0} and \eqref{eq:diffLref} that
\begin{align}\label{eq:betaLip-1}
&\left|v(x_1,z)-v(x_2,z)\right|\nonumber\\
 &\leq \beta \sup_{(\theta, c)\in\mathbb{U}^{\rm r}}\left|\Ex\left[\int_0^{\infty} e^{-\rho s}d(L_s^{x_2,z,\theta,c}-L_s^{x_1,z,\theta,c})\right]\right|\nonumber\\
&\leq \beta\sup_{(\theta, c)\in\mathbb{U}^{\rm r}} \lim_{T\to \infty}\Ex\left[e^{-\rho T}|L_{T}^{x_2,z,\theta,c}-L_{T}^{x_1,z,\theta,c}| + \rho \int_0^{T}e^{-\rho s} |L_s^{x_2,z,\theta,c}-L_s^{x_1,z,\theta,c}|ds\right]\nonumber\\
&\leq \beta\lim_{T\to\infty} \Ex\left[e^{-\rho T}|x_1-x_2| + \rho|x_1-x_2| \int_0^{T} e^{-\rho s}ds\right]\nonumber\\
&=\beta|x_1-x_2|.
\end{align}
Thus, we complete the proof of the lemma. 
\end{proof}

\begin{proof}[Proof of Lemma \ref{lem:budget-constraint}]
Applying It\^{o}'s lemma to $e^{-\rho t}X_tY_t$, together with \eqref{state-X} and \eqref{eq:Yt}, we get that
\begin{align}\label{eq:XY}
d(e^{-\rho t}X_tY_t)
&=-e^{-\rho t}\left(c_tY_t+F(Z_t)Y_t\right)dt-e^{-\rho t}X_tdL_t^Y+e^{-\rho t}Y_tdL_t^X\nonumber\\
&\quad+e^{-\rho t}\left[\left(-\mu^{\top}\sigma^{-1}X_tY_t+\theta_t^{\top}\sigma X_tY_t\right)dW_t-\sigma_Z (Z_t)Y_tdW_t^{\gamma}\right].
\end{align}
For any $T>0$, integrating w.r.t. $t$ on both sides of \eqref{eq:XY} from 0 to $T$, we arrive at
\begin{align}\label{eq:Mt}
&e^{-\rho T}X_TY_T+\int_0^T e^{-\rho t}(c_tY_t+F(Z_t)Y_t)dt+\int_0^T e^{-\rho t}X_tdL_t^Y\\
&\quad=xy+\int_0^T e^{-\rho t}\left[\left(-\mu^{\top}\sigma^{-1}X_tY_t+\theta_t^{\top}\sigma X_tY_t\right)dW_t-\sigma_Z (Z_t)Y_tdW_t^{\gamma}\right]+\int_0^Te^{-\rho t}Y_tdL_t^X:=M_T.\nonumber
\end{align}
It follows from Assumption $(\bm{A_{Z}})$ and $X_t,Y_t>0$, $\Px$-a.s., for $t\geq 0$ and the non-decreasing property of $t\to L_t^X$ that $M=(M_t)_{t\geq 0}$ is a nonneagtive process. For any $n\geq1$, define the stopping time by
\begin{align*}
\tau_n:=\inf\left\{t\geq 0;~\int_0^t\left|(\mu+\theta)^{\top}\sigma^{-1}X_tY_t\right|^2ds\geq n~\text{or}~\int_0^{t}|\sigma_Z (Z_t)Y_t|^2ds\geq n\right\}.
\end{align*}  
Then, Fatou's lemma and MCT yield that
\begin{align}\label{eq:Mt-Fatou}
\Ex[M_T]=\Ex\left[\liminf_{n\to \infty}M_{T\wedge \tau_n}\right]\leq \liminf_{n\to \infty}\Ex\left[M_{T\wedge \tau_n}\right]&=xy+\liminf_{n\to \infty}\Ex\left[\int_0^{T\wedge \tau_n}e^{-\rho t}Y_tdL_t^X\right]\nonumber\\
&=xy+\Ex\left[\int_0^{T}e^{-\rho t}Y_tdL_t^X\right].
\end{align}
Letting $T$ tend to infinity on both sides of \eqref{eq:Mt-Fatou} and using \eqref{eq:Mt} and MCT, we obtain the desired characterization \eqref{eq:budget}.
\end{proof}

Before providing the proof of Theorem \ref{thm:duality}, we first present some technical preparations on the dual process and the dual value function. 
\begin{lemma}\label{lem:transversality}
Let Assumptions $(\bm{A_{Z}})$, $(\bm{A_{o}})$ and $(\bm{A_{\rho}})$ hold. Then
\begin{itemize}
    \item[{\rm(i)}] for the constant $p<1$ in $(\bm{A_{o}})$, we have $\lim_{T\to \infty}\Ex\left[e^{-\rho T} Y_T^{\frac{p}{p-1}}\right]=0$;
    \item[{\rm(ii)}] the dual function $\hat{v}\in C^2((0,\beta]\times\R)$ and  $\hat{v}_y\in C^2((0,\beta]\times\R)$;
    \item[{\rm(iii)}]  for any  $z\in\R$, the dual function $y\to \hat{v}(y,z)$ is strictly decreasing and strictly convex.
\end{itemize}
\end{lemma}

\begin{proof}
For the case $p<0$, the result obviously holds as $Y_T^{\frac{p}{p-1}}\leq \beta^{\frac{p}{p-1}}$ a.s., for all $T\geq 0$. It is then sufficient to consider the case $p\in(0,1)$.  For any $t\geq 0$,  define $R_t=-\ln(Y_t/\beta)$ , i.e. $Y_t=\beta e^{-R_t}$. It follows from the It\^o's lemma that $R=(R_t)_{t\geq0}$, taking values on $[0,\infty)$, is a reflected process that
\begin{align}\label{eq:Rt}
R_t^{r}=r+ (\alpha-\rho)t+\mu^{\top}{\sigma}^{-1}W_t+L_t^{R},
\end{align}
where $r:=-\ln(y/\beta)$ and $L^{R}=(L_t^{R})_{t\geq 0}$ is a continuous and non-decreasing process that increases only on the time set $\{t\geq 0;~R_t=0\}$ with $L_0^{R}=0$.  Then, we have
\begin{align*}
0\leq \liminf_{T\to \infty}e^{-\rho T}\Ex\left[Y_T^{\frac{p}{p-1}}\right]&\leq \limsup_{T\to \infty}e^{-\rho T}\Ex\left[Y_T^{\frac{p}{p-1}}\right]=\limsup_{T\to \infty}\beta^{\frac{p}{p-1}} \Ex\left[e^{-\rho T+{\frac{p}{1-p}}R_T}\right].
\end{align*}
Applying It\^o's lemma to $e^{-\rho T+\frac{p}{1-p}R_T}$ and taking expectation, we obtain that
\begin{align}\label{eq:Ito-R}
\Ex\left[e^{-\rho T+\frac{p}{1-p}R_T}\right]&=e^{\frac{p}{1-p}r}+\frac{p}{1-p}\Ex\left[\int_0^Te^{-\rho t}dL_t^R\right]+\frac{\alpha p-(1-p)\rho}{(1-p)^2}\int_0^T \Ex\left[e^{-\rho t+\frac{p}{1-p}R_t}\right] dt\nonumber\\
&\leq e^{\frac{p}{1-p}r}+\frac{p}{1-p}e^{-r}+\frac{\alpha p-(1-p)\rho}{(1-p)^2}\int_0^T \Ex\left[e^{-\rho t+\frac{p}{1-p}R_t}\right] dt
\end{align}
It follows from \eqref{eq:Ito-R} and the Gronwall’s lemma that, for all $T\geq 0$,
\begin{align*}
\Ex\left[e^{-\rho T+\frac{p}{1-p}R_T}\right]
&\leq\left(e^{\frac{p}{1-p}r}+\frac{p}{1-p}e^{-r}\right)\exp\left(-\frac{(1-p)\rho-\alpha p}{(1-p)^2}T\right).
\end{align*}
This yields that, for any $\rho>\alpha p/(1-p)$,
\begin{align*}
\limsup_{T\to \infty}e^{-\rho T}\Ex\left[Y_T^{\frac{p}{p-1}}\right]&\leq\limsup_{T\to \infty}\beta^{\frac{p}{p-1}} \Ex\left[e^{-\rho T+{\frac{p}{1-p}}R_T}\right]=0.
\end{align*}
Thus, we complete the proof of item (i).

Next, we prove (ii) and (iii). To this end, we introduce the function $u(r,z)$ with $(r,z)\in\R_+\times\R$ satisfying $u(r,z)=\hat{v}(\beta e^{-r},z)$. Then, the function $u(r,z)$ has the probabilistic representation that, for all $(r,z)\in\R_+\times\R$,
\begin{align}\label{sol:u}
u(r,z)&=\Ex\left[\int_0^{\infty} e^{-\rho s}\Phi\left(\beta e^{-R_s^r}\right)ds\right]- \beta \Ex\left[ \int_0^{\infty}e^{-\rho s-R_s^r}F(Z_s) ds\right]\nonumber\\
&=:l(r)+ h(r,z).
\end{align}
For the function $l(r)$, using Assumptions $(\bm{A_{o}})$  and $(\bm{A_{p}})$, it follows that, for some constant $K>0$, 
\begin{align}\label{eq:l-func}
|l(r)|&\leq K\Ex\left[\int_0^{\infty} e^{-\rho s}\left(1+\beta^{\frac{p}{p-1}} e^{\frac{p}{1-p}R_s^r}\right)ds\right]\nonumber\\
&=K\left(\frac{1}{\rho}+\frac{(1-p)^2}{\rho(1-p)-\alpha p}\beta^{-\frac{p}{1-p}} e^{\frac{p}{1-p}r}+\frac{p(1-p)}{\rho(1-p)-\alpha p}\beta^{-\frac{p}{1-p}} e^{-r}\right),\quad \forall r\in\R_+.
\end{align}
For the function $h(r,z)$, it follows from Assumptions $(\bm{A_{Z}})$  and $(\bm{A_{p}})$ that
\begin{align}\label{eq:h-func}
|h(r,z)|\leq K(1+z),\quad \forall (r,z)\in\R_+\times\R,
\end{align}
for some constant $K>0$.
Using \eqref{eq:l-func} and \eqref{eq:h-func}, we have  $|\hat{v}(y,z)|<\infty$ for $(y,z)\in(0,\beta]\times\R$, thus the dual function $\hat{v}(y,z)$ is well-defined. 

By a similar calculation as in the proof of Proposition  4.1 in \cite{BoLiaoYu21} and in the proof of Lemma 3.4 in \cite{BHY24}, we can obtain 
\begin{align}
l_r( r)&=-\beta\mathbb{E}\left[\int_0^{\tau_r } e^{-\rho s}\Phi'\left(\beta e^{-R_s}\right) ds\right].\label{eq:l-r-1}\\
l_{rr}(r)&=-\beta\int_{0}^{\infty}\int_{-\infty}^{r} e^{-\rho s}\Phi'(\beta e^{-r+x})\phi(s,x,r)dx ds+\beta\Ex\left[\int_{0}^{\tau_r} e^{-\rho s}\Phi''\left(\beta e^{-R_s}\right)ds\right],\label{eq:l-r-2}\\
l_{rrr}(r)&=2\beta  \int_{0}^{\infty}\int_{-\infty}^{r} e^{-\rho s}\Phi''(\beta e^{-r+x})\phi(s,x,r)dx ds-\beta  \Ex\left[\int_{0}^{\tau_r} e^{-\rho s}\Phi'''\left(\beta e^{-R_s}\right)ds\right]\label{eq:l-r-3}\\
&\quad-\beta  \int_{0}^{\infty} e^{-\rho s}\Phi'(\beta)\phi(s,r,r)ds-\beta \int_{0}^{\infty}\int_{-\infty}^{r} e^{-\rho s}\Phi'(\beta e^{-r+x})\phi_r(s-t,x,r)dx ds.\nonumber
\end{align}
Here, the stopping time $\tau_r$ is defined by
$\tau_r:=\inf \left\{s \geq 0 ;~-\mu^{\top}\sigma^{-1}W_s-\left(\alpha-\rho\right)s=r\right\}$, and the function $\phi(s,x,y)$ is given by 
\begin{align}\label{eq:phi}
\phi(s,x,y)= \frac{2(2 y-x)}{\tilde{\sigma}^2\sqrt{2\tilde{\sigma}^2 \pi s^3}} \exp \left(\frac{\tilde{\mu}}{\tilde{\sigma}} x-\frac{1}{2} \tilde{\mu}^2 s-\frac{(2 y-x)^2}{2\tilde{\sigma}^2 s}\right)
\end{align}
with parameters $\tilde{\mu}:=\sqrt{\alpha/2}-\rho/\sqrt{2\alpha}$ and $\tilde{\sigma}:=-\sqrt{2\alpha}$. Following the proof of Proposition  4.1 in \cite{BoLiaoYu21} and the proof of Lemma 3.4 in \cite{BHY24} again, we can similarly get that $h(r,z)\in C^{2}(\R_+\times \R)$ and $h_r(r,z)\in C^{2}(\R_+\times \R)$ with
\begin{align}
h_r(r,z)&= \beta \Ex\left[ \int_0^{\tau_r}e^{-\rho s-R_s^r}F(Z_s) ds\right],\label{eq:h-r-1}\\
h_{rr}(r,z)&= \beta\Gamma_1 \Ex\left[e^{-\rho \tau_r} F(Z_{\tau_r})\right]-\beta \Ex\left[ \int_0^{\tau_r}e^{-\rho s-R_s^r}F(Z_s) ds\right],\label{eq:h-r-2}\\
h_{rrr}(r,z)&=\beta\Gamma_1^2 \Ex\left[e^{-\rho \tau_r}\left(\mu_Z(Z_{\tau_r})F'(Z_{\tau_r})+\frac{1}{2}\sigma^2_Z(Z_{\tau_r})F''(Z_{\tau_r})\right)\right]+\beta\Gamma_1\Gamma_2\Ex\left[F(Z_{\tau_r})\right]\nonumber\\
&- \beta\Gamma_1\Ex\left[e^{-\rho \tau_r}F(Z_{\tau_r})\right]+\beta \Ex\left[ \int_0^{\tau_r}e^{-\rho s-R_s^r}F(Z_s) ds\right],\label{eq:h-r-3}
\end{align}
where the parameters  $\Gamma_1:=\int_0^{\infty}\frac{1}{\sqrt{4\alpha\pi s}}e^{-\frac{(\rho-\alpha)^2}{4\alpha}s} ds$ and $\Gamma_2:=\int_0^{\infty}\frac{1}{\sqrt{4\alpha\pi s^3}}e^{-\frac{(\rho+\alpha)^2}{4\alpha}s} ds$. From the above equations, we deduce $u\in C^2(\R_+\times\R)$ and  $u_r\in C^2(\R_+\times\R)$, which implies $\hat{v}\in C^2((0,\beta]\times\R)$ and  $\hat{v}_y\in C^2((0,\beta]\times\R)$.

Noting that $\Phi'(x)=-I(x)<0$ and $\Phi''(x)=-1/U''(I(x))>0$ for all $x>0$, we deduce that $l'(r)>0$, $h_r(r,z)\geq 0$ , $l'(r)+l''(r)>0$ and $h_r(r,z)+h_{rr}(r,z)\geq 0$ for any $r\in\R_+$. It also holds that
\begin{align*}
\hat{v}_{y}(y,z)&=-\frac{1}{y}u_r(r,z)=-\frac{1}{y}\left(l_r(r)+h_r(r,z)\right)<0,\\
\hat{v}_{yy}(y,z)&=\frac{1}{y^2}\left(u_r(r,z)+u_{rr}(r,z)\right)=\frac{1}{y}\left(l_r(r)+l_{rr}(r)+h_r(r,z)+h_{rr}(r,z)\right)>0
\end{align*}
with $r=-\ln(y/\beta)$. As a consequence, for any $z\in\R$, the function $y\to \hat{v}(y,z)$ is strictly decreasing and  strictly convex. 
\end{proof}

Using Lemma \ref{lem:transversality}, we can show that the dual function $\hat{v}(y,z)$ given by \eqref{eq:dual-func} satisfies a linear PDE:
\begin{proposition}\label{prop:dual-equation}
Let Assumptions $(\bm{A_{Z}})$, $(\bm{A_{o}})$ and $(\bm{A_{\rho}})$  hold. The dual function $\hat{v}(y,z)$ given by \eqref{eq:dual-func} is a classical solution of the following linear dual PDE that, for $(y,z)\in(0,\beta]\times\R$,
\begin{align}\label{eq:HJB-dual-hatv} 
\begin{cases}
\displaystyle-\rho \hat{v}(y,z)+\rho y \hat{v}_y(y,z)+\alpha y^2 \hat{v}_{y y}(y,z)+\mu_Z (z) \hat{v}_z(y,z)+\frac{\sigma_Z^2(z)}{2}\hat{v}_{z z}(y,z)\\[0.6em]
\displaystyle\qquad\qquad\qquad-\eta (z) y \hat{v}_{y z}(y,z)-F(z) y+\Phi(y)=0;\\[0.4em]
\displaystyle \hat{v}_y(\beta,z)=0
\end{cases}
\end{align}
with $\eta(z):=\sigma_Z(z)\gamma^{\top}\sigma^{-1}\mu$ for $z\in\R$.
Moreover, if $\varphi(y,z)$ is a classical solution of the dual equation \eqref{eq:HJB-dual-hatv} satisfying $|\varphi(y,z)|\leq K(1+z+y^{\frac{p}{p-1}})$ for some constant $K>0$, then it admits the probability representation \eqref{eq:dual-func}.
\end{proposition}

\begin{proof}
By using Lemma \ref{lem:transversality}, \eqref{eq:l-func} and \eqref{eq:h-func}, there exists some $K>0$ such that, for all $(y,z)\in(0,\beta]\times\R$, 
\begin{align}\label{eq:hatv-low}
\hat{v}(y,z)\geq \hat{v}(\beta,z)=u(1,z)=l(1)-K(1+z).
\end{align}
This, together with Assumptions $(\bm{A_{Z}})$ and $(\bm{A_{\rho}})$, yields that
\begin{align}\label{eq:inf-hatv}
\liminf_{T\to \infty}\Ex\left[e^{-\rho T} v\left(Y_T,Z_T\right)\right]&\geq\liminf_{T\to \infty}\Ex\left[e^{-\rho T} \left(l(1)+K(1+Z_T)\right)\right]\nonumber\\
&\geq \liminf_{T\to \infty}Ke^{-(\rho-\sup_{z\in\R}\mu_Z'(z))T}=0.
\end{align}
On the other hand, we have from  \eqref{eq:l-func} and \eqref{eq:h-func} that
\begin{align}\label{eq:hatv-high}
\hat{v}(y,z)\leq K\left(1+y^{\frac{p}{p-1}}\right),\quad \forall (y,z)\in(0,\beta]\times\R
\end{align}
for some constant $K>0$ independent of $y$ and $z$. By Lemma \ref{lem:transversality} and \eqref{eq:hatv-high}, we can get that
\begin{align}\label{eq:sup-hatv}
\limsup_{T\to \infty}\Ex\left[e^{-\rho T} v\left(Y_T,Z_T\right)\right]&\leq\limsup_{T\to \infty}\Ex\left[e^{-\rho T}K\left(1+Y_T^{\frac{p}{p-1}}\right)\right]=0.
\end{align}
Thanks to \eqref{eq:inf-hatv} and \eqref{eq:sup-hatv}, we arrive at
\begin{align}\label{eq:lim-hatv}
\lim_{T\to \infty}\Ex\left[e^{-\rho T} v\left(Y_T,Z_T\right)\right]=0.
\end{align}
Then, the desired result follows from Lemma \ref{lem:transversality}, \eqref{eq:lim-hatv} and a  similar proof of Lemma 3.3 in \cite {BHY24}.
\end{proof}

We are now in a position to prove Theorem \ref{thm:duality}.
\begin{proof}[Proof of Theorem \ref{thm:duality}]
For any $(\theta,c)\in\mathbb{U}^r$, let us consider the resulting state process $(X,Z)=(X_t,Z_t)_{t\geq 0}$ given by \eqref{state-X} and \eqref{eq:Zt}. It follows from Lemma \ref{lem:budget-constraint} that
\begin{align}\label{eq:dual-ineq}
&\Ex\left[\int_0^\infty e^{-\rho t} U(c_t)dt- \beta \int_0^{\infty} e^{-\rho t}dL_t^X\right]-xy+xy\nonumber\\
&\leq \Ex\left[\int_0^\infty e^{-\rho t} U(c_t)dt- \beta \int_0^{\infty} e^{-\rho t}dL_t^X\right]\nonumber\\
&\quad-\Ex\bigg[\int_0^{\infty}e^{-\rho t}\Big(c_tY_t+F(Z_t)Y_t\Big)dt+\int_0^{\infty}e^{-\rho t}X_tdL_t^Y-\int_0^{\infty}e^{-\rho t}Y_tdL_t^X\bigg]+xy\nonumber\\
&=\Ex\left[\int_0^\infty e^{-\rho t}\left(\left( U(c_t)-c_tY_t\right)-F(Z_t)Y_t\right)dt-  \int_0^{\infty} e^{-\rho t}\left((\beta-Y_t)dL_t^X+X_tdL_t^Y\right)\right]+xy\nonumber\\
&\leq \Ex\left[\int_0^\infty e^{-\rho t} \left(\Phi(Y_t)-F(Z_t)Y_t\right)dt\right]+xy=\hat{v}(y,z)+xy.
\end{align}
The equality in the above display holds if and only if 
\begin{align}\label{eq:condition-dual}
\begin{cases}
\displaystyle \Ex\left[\int_0^{\infty}e^{-\rho t}\left(c_tY_t+F(Z_t)Y_t\right)dt+\int_0^{\infty}e^{-\rho t}X_tdL_t^Y-\int_0^{\infty}e^{-\rho t}Y_tdL_t^X\right]= xy,\\[0.9em]
\displaystyle X_t=0\Longleftrightarrow Y_t=\beta ,~~\forall t\geq 0.
\end{cases}
\end{align}

For any $(x,z)\in\R_+\times \R$, let $y^*(x,z)$ be the function satisfying $ -\hat{v}_y(y^*(x,z),z)=x$. As $\hat{v}_{yy}{\color{red}>}0$, $\hat{y}(\beta,z)=0$ and $\lim_{y\to 0}\hat{v}_y(y,z)=-\infty$, we have $y^*(x,z)$ is well-defined. We claim that, if one takes $(\theta,c)=(\theta^*,c^*)$ and $y=y^*(x,z)$, then the conditions in \eqref{eq:condition-dual} hold. Denote by $X^*=(X^*_t)_{t\geq 0}$ the state process under $(\theta^*,c^*)$. Let us introduce the process $\tilde{X}^*=(\tilde{X}^*_t)_{t\geq 0}$ given by
\begin{align*}
\tilde{X}_t^*=-\hat{v}_y(Y_t,Z_t),~~t\geq0,~~ \tilde{X}_0^*=-\hat{v}_y(Y_0,Z_0)=-\hat{v}_y(y^*(x,z),z)=x.   
\end{align*}
For any $t\geq0$, applying It\^o's lemma to $-\hat{v}_y(Y_t,Z_t)$ yields that
\begin{align}\label{eq:tilde-X-1}
d\tilde{X}_t^*&=-\Big(\rho Y_t\hat{v}_{yy}(Y_t,Z_t)+\alpha Y_t^2 \hat{v}_{yyy}(Y_t,Z_t)+\mu_Z(Z_t)\hat{v}_{yz}(Y_t,Z_t)+\frac{1}{2}\sigma_Z^2 (Z_t)\hat{v}_{yzz}(Y_t,Z_t)\\
&-\eta (Z_t) Y_t \hat{v}_{yy z}(Y_t,Z_t)\Big)dt+\mu^{\top}\sigma^{-1}Y_t\hat{v}_{yy}(Y_t,Z_t)dW_t-\sigma_Z (Z_t )\hat{v}_{yz}(Y_t,Z_t)dW_t^{\gamma}+\hat{v}_{yy}(Y_t,Z_t)dL_t^Y.\nonumber
\end{align}
Differentiating w.r.t. $y$ on both sides of \eqref{eq:HJB-dual-hatv}, we get that
\begin{align}\label{eq:tilde-X-2}
&\rho y \hat{v}_{yy}(y,z)+\alpha y^2 \hat{v}_{y yy}(y,z)+\mu_Z (z )\hat{v}_{yz}(y,z)+\frac{\sigma^2_Z(z)}{2}  \hat{v}_{yz z}(y,z)-\eta (z) y \hat{v}_{yy z}(y,z)\nonumber\\
&\qquad\qquad=-(2\alpha y \hat{v}_{y y}(y,z)-\eta (z)  \hat{v}_{y z}(y,z)-F(z) +\Phi'(y)).
\end{align}
It follows from \eqref{eq:optimal-control-Yt}, \eqref{eq:tilde-X-1}, \eqref{eq:tilde-X-2} and $\Phi'(y)=-I(y)$ that
\begin{align}
d\tilde{X}_t^*=(\theta^*_t)^{\top}\mu dt+(\theta^*_t)^{\top}\sigma dW_t  - c^*_t dt- \mu_Z (Z_t)dt- \sigma_Z (Z_t)dW_t^{\gamma}+dL_t^{\tilde{X}^*},~~t>0,
\end{align}
with $L_t^{\tilde{X}^*}=\int_0^t\hat{v}_{yy}(Y_s,Z_s)dL_s^Y=\int_0^t\hat{v}_{yy}(\beta,Z_s)dL_s^Y$ for all $t\geq0$. Note that $\hat{v}_y(y,z)< 0$ for $y\in(0,\beta)$, $\hat{v}_y(\beta,z)= 0$ and $\hat{v}_{yy}(y,z)>0$ for $(y,z)\in(0,\beta]\times\R$. Then, $L^{\tilde{X}^*}=(L_t^{\tilde{X}^*})_{t\geq0}$ is a continuous and non-decreasing process (with $L_0^{\tilde{X}^*}=0$)  which increases on the time set $\{t\geq 0;~\tilde{X}_t^*=0\}$ only. This implies that  $\tilde{X}^*$, taking values on $[0,\infty)$, is a reflected process and $L^{\tilde{X}^*}$ is the local time process of $\tilde{X}^*$. Using the solution representation of ``the Skorokhod problem", we obtain $X^*_t=\tilde{X}_t^*$ for all $t\geq 0$, which verifies the second condition in \eqref{eq:condition-dual}. Thus, to check the first condition in \eqref{eq:condition-dual}, it suffices to show that, for all $(x,z)\in\R_+\times \R$,
\begin{align*}
\Ex\bigg[\int_0^{\infty}e^{-\rho t}\Big(c_t^*Y_t+F(Z_t)Y_t\Big)dt-\beta\int_0^{\infty}e^{-\rho t}dL_t^{X^*}\bigg]= xy^*(x,z),
\end{align*}
which is equivalent to that, for all $(y,z)\in(0,\beta]\times\R_+$,
\begin{align*}
\Ex\bigg[\int_0^{\infty}e^{-\rho t}\Big(I(Y_t)Y_t+F(Z_t)Y_t\Big)dt-\beta\int_0^{\infty}e^{-\rho t}\hat{v}_{yy}(\beta,Z_t)dL_t^Y\bigg]= -y\hat{v}_y(y,z).
\end{align*}
Denote by
\begin{align*}
\varphi(y,z)
=\Ex\bigg[\int_0^{\infty}e^{-\rho t}\Big(I(Y_t)Y_t+F(Z_t)Y_t\Big)dt-\beta\int_0^{\infty}e^{-\rho t}\hat{v}_{yy}(\beta,Z_t)dL_t^Y\bigg| Y_0=y,Z_0=z\bigg].
\end{align*}
In a similar fashion of the proof to Proposition  4.1 in \cite{BoLiaoYu21} and the proof to Lemma 3.3 in \cite{BHY24}, we have that,  the function $\varphi(y,z)$ is the uniqueness classical solution satisfying $|\varphi(y,z)|\leq K(1+z+y^{\frac{p}{p-1}})$ for some constant $K>0$ to the linear PDE with Neumann boundary condition:
\begin{align}\label{eq:y-hatv} 
\begin{cases}
\displaystyle-\rho \varphi(y,z)+\rho y \varphi_y(y,z)+\alpha y^2 \varphi_{y y}(y,z)+\mu_Z (z) \varphi_z(y,z)+\frac{\sigma_Z^2(z)}{2}  \varphi_{z z}(y,z)\\[0.6em]
\displaystyle\qquad\qquad\qquad-\eta (z) y \varphi_{y z}(y,z)+F(z) y+yI(y)=0,\\[0.4em]
\displaystyle \varphi_{y}(\beta,z)=\hat{v}_{yy}(\beta,z),\quad \forall z\in\R.
\end{cases}
\end{align}
On the other hand, for $(y,z)\in(0,\beta]\times \R$ and $r=-\ln(y/\beta)$, by the proof of Lemma \ref{eq:optimal-control-Yt} and Assumption $(\bm{A_{o}})$,  it holds that
\begin{align*}
0\leq -y\hat{v}_{y}(y,z)&=l'(r)+h_r(r,z)\leq \beta\mathbb{E}\left[\int_0^{\tau_r } e^{-\rho s}|\Phi'\left(\beta e^{-R_s}\right)| ds\right]+Kz\nonumber\\
&\leq \beta K\Ex\left[\int_0^{\infty} e^{-\rho s}\left(1+\beta^{\frac{1}{p-1}} e^{\frac{1}{1-p}R_s}\right)ds\right]+K z\nonumber\\
&=\beta K\left(\frac{1}{\rho}+\frac{p(1-p)}{\rho(1-p)-\alpha p}\beta^{-\frac{1}{1-p}} e^{\frac{p}{1-p}r}-\frac{p(1-p)}{\rho(1-p)-\alpha p}\beta^{-\frac{1}{1-p}} e^{-r}\right)+K z\nonumber\\
&\leq K\left(\frac{\beta}{\rho}+\frac{|p|(1-p)}{\rho(1-p)-\alpha p}\beta^{-\frac{p}{1-p}}+\frac{|p|(1-p)}{\rho(1-p)-\alpha p}y^{\frac{p}{1-p}}\right)+K z\nonumber\\
&\leq K\left(1+y^{\frac{p}{1-p}}+z\right),
\end{align*}
where $K>0$ is a constant that may be different from line to line.
By utilizing \eqref{eq:HJB-dual-hatv} and a direct calculation, it can be checked that, the function $-y\hat{v}_y(y,z)$ also satisfies the PDE \eqref{eq:y-hatv}, which yields that $\varphi(y,z)=-y\hat{v}_y(y,z)$ for all $(y,z)\in(0,\beta]\times \R$. By the arbitrariness of $(\theta,c) \in\mathbb{U}^r$ and $(x,z,y)\in\R_+\times \R\times (0,\beta]$, we deduce from \eqref{eq:dual-ineq} that
\begin{align*}
v(x,z)=\sup_{(\theta,c)\in\mathbb{U}^{\rm r}}\Ex\left[\int_0^\infty e^{-\rho t} U(c_t)dt- \beta \int_0^{\infty} e^{-\rho t}dL_t^X\right]\leq \inf_{y\in(0,\beta]}(\hat{v}(y,z)+xy).
\end{align*}
Thus, it holds that $\hat{v}(y^*,z)+xy^*=J(x,z;\theta^*,c^*)\leq v(x,z)\leq \inf_{y\in(0,\beta]}(\hat{v}(y,z)+xy)$, which readily yields the desired result.
\end{proof}

\begin{proof}[Proof of Corollary \ref{coro:optimal-control}]
Define $x^*(y,z)=v_x(\cdot,z)^{-1}(y)$ with $y\to v_x(\cdot,z)^{-1}(y)$ being the inverse function of $x\to v_x(x,z)$. Then, $x^*=x^*(y,z)$ satisfies the equation $v_x(x^*,z)=y$ for all $z\in\R$. We can obtain by some straightforward calculations that
\begin{align*}
    &\hat{v}(y,z)=v(x^*,z)-x^*y,\quad
    \hat{v}_y(y,z)=-x^*,\quad \hat{v}_z(y,z)=v_z(x^*,z),\quad
   \hat{v}_{yy}(y,z)=-\frac{1}{v_{xx}(x^*,z)},\\
   &\hat{v}_{yz}(y,z)=\frac{v_{xz}(x^*,z)}{v_{xx}(x^*,z)},\qquad
   \hat{v}_{zz}(y,z)=v_{zz}(x^*,z)-\frac{v_{xz}(x^*,z)^2}{v_{xx}(x^*,z)}.
\end{align*}
In view of above dual relationship equations, it is easy to check that, the state process $X^*$ under the optimal strategy given by \eqref{eq:optimal-control-Yt} has the same dynamics as the one under the feedback control pair defined in this corollary, which yields that the two strategy pairs are identical. Thus, we complete the proof of the corollary.
\end{proof}

\begin{proof}[Proof of Lemma \ref{lem:inject-captial}]
We first show the item (i). For any $(\mathrm{v},z)\in\R_+^2$, it follows from \eqref{eq:w}, Assumption $(\bm{A_{o}})$ and Lemma \ref{lem:equivalence} that
\begin{align}\label{eq:discountAsatr}
 \beta\Ex\left[\int_0^{\infty} e^{-\rho t}dA^*_t \right]&= \Ex\left[\int_0^{\infty} e^{-\rho t} U(c^*_t)dt\right] -{\rm w}(\mathrm{v},z)-\beta ( z-\mathrm{v})^+\nonumber\\
&=  \Ex\left[\int_0^{\infty} e^{-\rho t} U(I(Y_t))dt\right] -v(x, z)\nonumber\\
&=  \Ex\left[\int_0^{\infty} e^{-\rho t} \left(\Phi(Y_t)-Y_t\Phi'(Y_t)\right)dt\right]-v(x, z)\nonumber\\
&\leq \Ex\left[\int_0^{\infty} e^{-\rho t} \left(|\Phi(Y_t)|+|\Phi'(Y_t)|\right)dt\right]-v(x, z)
\nonumber\\
&\leq 2C\Ex\left[\int_0^{\infty} e^{-\rho t} \left(1+Y_t^{\frac{p}{p-1}}\right)dt\right]-v(x, z)\nonumber\\
&<M\left(1+y^{\frac{p}{p-1}}-v(x,z)\right)=M\left(1+v^{\frac{p}{p-1}}_x(x,z)-v(x,z)\right),
\end{align}
where the last inequality follows from the proof of Lemma \ref{lem:transversality}, and  $M>0$ is a constant depending on $(\mu,\sigma,\mu_Z,\sigma_Z,\gamma,p,\beta)$. Then, we can take the function $P(x,z)=M(1+v^{\frac{p}{p-1}}_x(x,z)-v(x,z))$ and get the inequality \eqref{eq:dAstarfinite}.

Next, we show the item (ii). Consider the stochastic control problem given by
\begin{align*}
\tilde{\mathrm{w}}(\mathrm{v},z):=\inf_{\theta \in\tilde{\mathbb{U}}}\Ex\left[  \int_0^{\infty} e^{-\rho t}d\tilde{A}_t^{\theta}\Big|V_0=\mathrm{v},Z_0=z \right],\quad (\mathrm{v},z)\in\R_+\times\R
\end{align*}
subject to the state processes $(\tilde{V}^{\theta},Z,\tilde{A}^{\theta})=(\tilde{V}_t^{\theta},Z_t,\tilde{A}_t^{\theta})_{t\geq0}$ satisfying
\begin{align}\label{eq:wealth-theta}
\begin{cases}
\displaystyle \tilde{V}_t^{\theta} =\textrm{v}+\int_0^t\theta_s^{\top}\mu ds+\int_0^t \theta_s^{\top}\sigma dW_s,\quad Z_t=z+\int_0^t \mu_Z (Z_s) d s+\int_0^t \sigma_Z (Z_s) d W_s^{\gamma}, \\[1.2em]
\displaystyle \tilde{A}_t^{\theta}=0\vee \sup_{s\leq t}(Z_{s}-\tilde{V}^{\theta}_{s}).
\end{cases}
\end{align}
Here, the admissible control set is defined as $\tilde{\mathbb{U}}:=\{ \theta=(\theta_t)_{t\geq0}$; $\theta$ is $\Fx$-adapted process taking values on $\R^d\}$.
Note that $c_t\geq 0$ for all $t\in\R_+$. Then, it follows from \eqref{eq:wealth2} and \eqref{eq:wealth-theta} that $\tilde{V}_t^{\theta^*}\geq V^{\theta^*,c^*}_t$ for all $t\in\R_+$, and hence
\begin{align}\label{eq:tilde-w}
\Ex\left[\int_0^{\infty} e^{-\rho t}dA^{\theta,c}_t \right]>\Ex\left[\int_0^{\infty} e^{-\rho t}d\tilde{A}_t^{\theta} \right]\geq \inf_{\theta \in\tilde{\mathbb{U}}}\Ex\left[\int_0^{\infty} e^{-\rho t}d\tilde{A}_t^{\theta} \right]=\tilde{\mathrm{w}}(\mathrm{v},z).
\end{align}
Similar to the proof of Lemma \ref{lem:equivalence}, we can get $\tilde{{\rm w}}(\mathrm{v},z)=-\tilde{v}((\mathrm{v}-z)^+,z)+ ( z-\mathrm{v})^+$ for all $(\mathrm{v},z)\in\R_+\times\R$, where the function $\tilde{v}$ is given by
\begin{align}\label{eq:tilde-v}
\begin{cases}
\displaystyle \tilde{v}(x,z):=\sup_{\theta\in\mathbb{U}}\Ex\left[-  \int_0^{\infty} e^{-\rho t}dL_t^X\Big|X_0=x,Z_0=z \right],\\[1em]
\displaystyle \tilde{X}_t =x+\int_0^t\theta_s^{\top}\mu ds+\int_0^t\theta_s^{\top}\sigma dW_s-\int_0^t \mu_Z (Z_s)ds-\int_0^t \sigma_Z (Z_s)dW_s^{\gamma}+ L_t^{\tilde{X}},\\[0.8em]
\displaystyle Z_t=z+\int_0^t \mu_Z (Z_s) d s+\int_0^t \sigma_Z (Z_s) d W_s^{\gamma}.
\end{cases}
\end{align}
As a direct result of Theorem \ref{thm:duality}, we have that
\begin{align*}
\tilde{v}(x,z)=-\Ex\left[\int_0^{\infty}F(Z_t)Y_tdt\right]-x\tilde{v}_x(x,z),
\end{align*}
where $Y=(Y_t)_{t\geq 0}$ is given by \eqref{eq:Yt} with $Y_0= -\tilde{v}_x(x,z)$. Note that $F(z)>0$ for $z\in{\cal O}_Z$, which implies $\tilde{v}(x,z)<0$ for all $(x,z)\in\R_+\times {\cal O}_Z$. Then, we can take the function $L(x,z)=-\tilde{v}((\mathrm{v}-z)^+,z)+ ( z-\mathrm{v})^+$ and get the inequality \eqref{eq:dAstarpositive}, which completes the proof of the lemma.
\end{proof}

\begin{proof}[Proof of Proposition~\ref{prop:sol-HJB}]
Using the probabilistic representation \eqref{eq:dual-func}, we consider the candidate solution admitting the form $\hat{v}(y,z)=f(y)+z g(y)$ for Eq.~\eqref{eq:HJB-dual-hatv}. In particular, the functions $y\to f(r)$ and $y\to g(y)$ satisfy the following equations, respectively:
\begin{align}
 -\rho f(y)+\rho yf'(y)+\alpha y^2f''(y)+\Phi(y)&=0,\label{eq:f}\\[1em]
 (\mu_Z -\rho) g(y)+(\rho-\eta) yg'(y)+\alpha y^2g''(y)-(\mu_Z-\eta) y&=0.\label{eq:g}
\end{align}
By solving Eq.s \eqref{eq:f} and \eqref{eq:g}, we obtain, for $p\in(-\infty,1)$,
\begin{align*}
l(y)&=\begin{cases}
\displaystyle \frac{(1-p)^3}{p(\rho(1-p)-\alpha p)}y^{-\frac{p}{1-p}}+C_1  y+C_2 y^{-\frac{\rho}{\alpha}}, &p\neq 0,\\[1.8em]
\displaystyle \frac{1}{\rho} r-\frac{\ln\beta+2}{\rho}+\frac{\alpha}{\rho^2}+C_1  y+C_2 y^{-\frac{\rho}{\alpha}}, &p=0,
\end{cases}\\
g(y)&=y+C_3 y^{\kappa}+C_4 y^{-\hat{\kappa} },
\end{align*}
where $C_i$ with $i=1,\ldots,4$ are unknown real constants which will be determined later. Above, the the constant $\hat{\kappa}$ is given by
\begin{align*}
    \hat{\kappa}:=\frac{-(\rho-\eta-\alpha)-\sqrt{(\rho-\eta-\alpha)^2+4\alpha(\rho-\mu_Z)}}{2\alpha}<0.
\end{align*}
Using the probability representation \eqref{eq:dual-func}, we look for such functions $f(y)$ and $g(y)$ with $C_2=C_4=0$ and such that the Neumann boundary conditions $f'(\beta)=0$ and $g'(\beta)=0$ holds. This implies that
\begin{align*}
C_1=\frac{(1-p)^2}{\rho(1-p)-\alpha p}\beta^{-\frac{1}{1-p}},\quad
C_3=- \frac{\beta^{1-\kappa}}{\kappa}.
\end{align*}
With the above specified constants $C_i$ with $i=1,\ldots,4$, we can easily verify that $\hat{v}(y,z)=f(y)+zg(y)$ given by \eqref{sol:hat-v} satisfies Eq. \eqref{eq:HJB-dual-hatv}. Thus, we complete the proof of the proposition. 
\end{proof}

\begin{proof}[Proof of Corollary~\ref{coro:inject-captial-CRRA-GBM}]
For any $(\mathrm{v},z)\in\R_+^2$, we have from \eqref{eq:w}, \eqref{f1} and Lemma \ref{lem:equivalence} that
\begin{align}\label{eq:Uc-1}
 \beta\Ex\left[\int_0^{\infty} e^{-\rho t}dA^*_t \right]&= \Ex\left[\int_0^{\infty} e^{-\rho t} U(c^*_t)dt\right] -{\rm w}(\mathrm{v},z)-\beta ( z-\mathrm{v})^+\nonumber\\
&=  \Ex\left[\int_0^{\infty} e^{-\rho t} U(I(Y_t))dt\right] -v(x, z),
\end{align}
where $Y_0=f(x,z)$ and $x=(\mathrm{v}-z)^+$. Using a similar calculation in the proof of Proposition~\ref{prop:sol-HJB}, we have
\begin{align}\label{eq:Uc-2}
&\Ex\left[\int_0^{\infty} e^{-\rho t} U(I(Y_t))dt\right]\nonumber\\
&\quad=\begin{cases}
\displaystyle \frac{(1-p)^2}{p(\rho(1-p)-\alpha p)}f(x,z)^{-\frac{p}{1-p}}+\frac{1-p}{\rho(1-p)-\alpha p}\beta^{-\frac{1}{1-p}}f(x,z),~p<1,~~p\neq 0,\\[1.2em]
\displaystyle -\frac{1}{\rho} \ln f(x,z)-\frac{1}{\rho}+\frac{\alpha}{\rho^2}+\frac{1}{\rho \beta} f(x,z),~~p=0.
\end{cases}
\end{align}
From \eqref{eq:Uc-1}, \eqref{eq:Uc-2} and Proposition~\ref{prop:sol-HJB}, it follows that \eqref{eq:injection-CRRA} holds, which verifies the item (i).

Next, we deal with the item (ii). In fact, the value function $\tilde{v}$ defined by \eqref{eq:tilde-v} satisfies the following HJB equation that, on $\R_+\times(0,\infty)$,
\begin{align}\label{HJB-tilde}
\begin{cases}
\displaystyle\sup_{\theta\in\R^d}\left[\theta^{\top}\mu \tilde{v}_x+\frac{1}{2}\theta^{\top}\sigma\sigma^{\top}\theta \tilde{v}_{xx}+ \theta^{\top} \sigma \gamma\sigma_Z z (\tilde{v}_{x z}-\tilde{v}_{xx})  \right]-\sigma_Z^2 z^2 \tilde{v}_{xz}\\[0.6em]
\displaystyle\qquad\qquad\qquad\qquad\qquad\qquad+\frac{1}{2}\sigma_Z^2 z^2 (\tilde{v}_{xx}+\tilde{v}_{zz})+\mu_Z z (\tilde{v}_z-\tilde{v}_x) =\rho \tilde{v},\\[0.4em]
\displaystyle \tilde{v}_x(0,z)=\beta,~~\forall z>0.
\end{cases}
\end{align}
Applying the first-order condition, we arrive at, on $\R_+\times(0,\infty)$,
\begin{align}\label{HJB-v-tilde}
\begin{cases}
\displaystyle -\alpha \frac{\tilde{v}_x^2}{ \tilde{v}_{xx}}+\frac{1}{2}\sigma_Z^2 z^2\left(\tilde{v}_{z z}- \frac{ \tilde{v}_{x z}^2}{ \tilde{v}_{x x}}\right)-\eta z \frac{ \tilde{v}_x  \tilde{v}_{x z}}{ \tilde{v}_{x x}}+(\eta-\mu_Z)z\tilde{v}_x+\mu_Z z \tilde{v}_z =\rho v,\\[1.3em]
\displaystyle v_x(0,z)=\beta,\quad \forall z>0.
\end{cases}
\end{align}
Here, the coefficients
$\alpha=\frac{1}{2} \mu^{\top}(\sigma \sigma^{\top})^{-1}\mu$ and $\eta=\sigma_Z\gamma^{\top}\sigma^{-1}\mu$. It can be directly verified that
\begin{align}\label{eq:sol-v}
\tilde{v}(x,z)=z\frac{1-\kappa}{\kappa}\left(1+\frac{x}{z}\right)^\frac{\kappa}{\kappa-1}
\end{align}
is a classical solution to HJB equation \eqref{HJB-v-tilde} with $\kappa$ being the constant defined by $\kappa=(-(\rho-\eta-\alpha)+\sqrt{(\rho-\eta-\alpha)^2+4\alpha(\rho-\mu_Z)})/(2\alpha)$.
Thus, we complete the proof of the corollary.
\end{proof}

\noindent{\it Proof of Lemma \ref{lem:asymptotic-control}.}\quad  It follows from the duality relationship that $x=-\hat{v}_y(y,z)$ and
\begin{align}\label{eq:dual-2}
\theta^*(x,z)&= -\frac{\mu}{\sigma^2}\frac{v_x}{v_{xx}}(x,z)-\frac{\sigma_Z}{\sigma}\left(\frac{zv_{xz}}{v_{xx}}(x,z)-z\right)\nonumber\\
&=\frac{\mu}{\sigma^2}f(x,z)\hat{v}_{yy}(f(x,z),z)-\frac{\sigma_Z}{\sigma}(z\hat{v}_{yz}(f(x,z),z)-z).
\end{align}
Hence, we have from \eqref{eq:dual-2} that
\begin{align}
\lim_{x\to +\infty}\frac{\theta^*(x,z)}{x}&=\lim_{x\to +\infty}\frac{\mu f(x,z)\hat{v}_{yy}(f(x,z),z)-\sigma_Z\sigma(z\hat{v}_{yz}(f(x,z),z)-z)}{-\sigma^2\hat{v}_y(f(x,z),z)}\nonumber\\
&=\lim_{y\to 0}\frac{\mu y\hat{v}_{yy}(y,z)-\sigma_Z\sigma z\hat{v}_{yz}(y,z)}{-\sigma^2\hat{v}_y(y,z)}.
\end{align}
By using \eqref{sol:hat-v}, for $p<1$ with $p\neq 0$, it holds that
\begin{align}
\lim_{y\to 0}\frac{y\hat{v}_{yy}(y,z)}{-v_y(y,z)}&=\lim_{x\to +\infty}\frac{\frac{1-p}{\rho(1-p)-\alpha p}y^{-\frac{1}{1-p}}+(1-\kappa) \beta^{-(\kappa-1)}zy^{\kappa-1}}{\frac{(1-p)^2}{\rho(1-p)-\alpha p}(y^{-\frac{1}{1-p}}-\beta^{-\frac{1}{1-p}})+z\left(\beta^{-(\kappa-1)}y^{\kappa-1}-1\right)}\nonumber\\
&=\lim_{y\to 0}\frac{\frac{1-p}{\rho(1-p)-\alpha p}+(1-\kappa) \beta^{-(\kappa-1)}zy^{\kappa-1+\frac{1}{1-p}}}{\frac{(1-p)^2}{\rho(1-p)-\alpha p}(1-\beta^{-\frac{1}{1-p}}y^{\frac{1}{1-p}})+z\left(\beta^{-(\kappa-1)}y^{\kappa-1+\frac{1}{1-p}}-y^{\frac{1}{1-p}}\right)}\nonumber\\
&=\frac{1}{1-p}\lim_{y\to 0}\frac{1+C^*(1-\kappa)(1-p) \beta^{-(\kappa-1)}zy^{\kappa+\frac{p}{1-p}}}{1+zC^*\beta^{-(\kappa-1)}y^{\kappa+\frac{p}{1-p}}},\label{eq:theta-x-1}\\
\lim_{y\to 0}\frac{z\hat{v}_{yz}(y,z)}{-v_y(y,z)}&=\lim_{y\to 0}\frac{(\beta^{-(\kappa-1)}y^{\kappa-1}-1)z}{\frac{(1-p)^2}{\rho(1-p)-\alpha p}(y^{-\frac{1}{1-p}}-\beta^{-\frac{1}{1-p}})+z\left(\beta^{-(\kappa-1)}y^{\kappa-1}-1\right)}\nonumber\\
&=\lim_{y\to 0}\frac{(\beta^{-(\kappa-1)}y^{\kappa+\frac{p}{1-p}}-y^{\frac{1}{1-p}})z}{\frac{(1-p)^2}{\rho(1-p)-\alpha p}(1-\beta^{-\frac{1}{1-p}}y^{\frac{1}{1-p}})+z\left(\beta^{-(\kappa-1)}y^{\kappa+\frac{p}{1-p}}-y^{\frac{1}{1-p}}\right)}\nonumber\\
&=\lim_{y\to 0}\frac{C^*\beta^{-(\kappa-1)}y^{\kappa+\frac{p}{1-p}}z}{1+C^*\beta^{-(\kappa-1)}y^{\kappa+\frac{p}{1-p}}z}.\label{eq:theta-x-2}
\end{align}
By discussing the three cases of $p>p_1$, $p=p_1$ and $p<p_1$, respectively, we can deduce the desired result \eqref{eq:limtheta} following from \eqref{eq:theta-x-1} and \eqref{eq:theta-x-2}. Furthermore, we can easily verify that the limit \eqref{eq:limtheta} also holds for the case with $p=0$.

Next, it follows from the relationship $x=-\hat{u}_y(y,z)$ and \eqref{sol:hat-v} that, for $p<1$ with $p\neq 0$,
\begin{align}\label{eq:c-x}
\lim_{x\to +\infty}\frac{c^*(x,z)}{x}&=\lim_{x\to +\infty}\frac{f(x,z)^{\frac{1}{1-p}}}{-\hat{v}_y(f(x,z),z)}=\lim_{y\to 0}\frac{y^{\frac{1}{1-p}}}{-\hat{v}_y(y,z)}\nonumber\\
&=\lim_{y\to 0}\frac{y^{-\frac{1}{1-p}}}{\frac{(1-p)^2}{\rho(1-p)-\alpha p}(y^{-\frac{1}{1-p}}-\beta^{-\frac{1}{1-p}})+z\left(\beta^{-(\kappa-1)}y^{\kappa-1}-1\right)}\nonumber\\
&=\lim_{y\to 0}\frac{C^*}{(1-\beta^{-\frac{1}{1-p}}y^{\frac{1}{1-p}})+zC^*\left(\beta^{-(\kappa-1)}y^{\kappa+\frac{p}{1-p}}-y^{\frac{1}{1-p}}\right)}\nonumber\\
&=\lim_{y\to 0}\frac{C^*}{1+C^*\beta^{-(\kappa-1)}y^{\kappa+\frac{p}{1-p}}z}.
\end{align}
Thanks to \eqref{eq:c-x}, we can obtain the desired result \eqref{eq:limc} via a direct calculation. Similar proof can be conducted in the case $p=0$. \hfill$\Box$

\begin{proof}[Proof of Proposition \ref{lem:prop}]
We first prove item (i). For $(x,z)\in\R_+^2$, recall that $f(x,z)\in(0,\beta]$ satisfies the following equation:
\begin{align}\label{eq:dual-u3}
\hat{v}_y(f(x,z),z)=-x.
\end{align}
Taking the derivative with respect to $x$ on both sides of \eqref{eq:dual-u3}, we deduce that
\begin{align}\label{eq:dual-u4}
f_x(x,z)=-\frac{1}{\hat{v}_{yy}(f(x,z),z)}.
\end{align}
Then, it follows from \eqref{sol:v} and \eqref{eq:dual-u4} that
\begin{align}\label{eq:derivate-c-x}
\frac{\partial c^*(x,z)}{\partial x}=\frac{\partial f(x,z)^{\frac{1}{p-1}}}{\partial x}=\frac{1}{p-1}f(x,z)^{\frac{1}{p-1}-1}f_x(x,z)=\frac{f(x,z)^{\frac{1}{p-1}-1}}{(1-p)\hat{v}_{yy}(f(x,z),z)}>0.
\end{align}
This implies that $x\to c^*(x,z)$ is increasing.

Next, we deal with item (ii). Taking the derivative with respect to $x$ on both sides of \eqref{eq:dual-u4} again, we obtain
\begin{align}\label{eq:dual-u5}
f_{xx}(x,z)=\frac{1}{\hat{v}^2_{yy}(f(x,z),z)}\hat{v}_{yyy}(f(x,z),z)f_x(x,z)=-\frac{\hat{v}_{yyy}(f(x,z),z)}{\hat{v}^3_{yy}(f(x,z),z)}.
\end{align}
It follows from \eqref{sol:v}, \eqref{eq:derivate-c-x} and \eqref{eq:dual-u5} that
\begin{align}\label{eq:derivate-c-xx}
\frac{\partial^2 c^*(x,z)}{\partial x^2}&=\frac{1}{p-1}f(x,z)^{\frac{2-p}{p-1}}f_{xx}(x,z)+\frac{2-p}{(p-1)^2}f(x,z)^{\frac{3-2p}{p-1}}f_{x}(x,z)^2\nonumber\\
&=-\frac{1}{p-1}f(x,z)^{\frac{2-p}{p-1}}\frac{\hat{v}_{yyy}(f(x,z),z)}{\hat{v}^3_{yy}(f(x,z),z)}+\frac{2-p}{(p-1)^2}f(x,z)^{\frac{3-2p}{p-1}}\frac{1}{\hat{v}^2_{yy}(f(x,z),z)}\nonumber\\
&=\frac{1}{1-p}\frac{f(x,z)^{\frac{3-2p}{p-1}}}{\hat{v}^3_{yy}(f(x,z),z)}\left[f(x,z)\hat{v}_{yyy}(f(x,z),z)+\frac{2-p}{1-p}\hat{v}_{yy}(f(x,z),z)\right].
\end{align}
We then obtain from \eqref{sol:hat-v} and \eqref{eq:derivate-c-xx} that
\begin{align}\label{eq:derivate-c-xx-1}
\frac{\partial^2 c^*(x,z)}{\partial x^2}&=\frac{1-\kappa}{1-p}\frac{f(x,z)^{\frac{3-2p}{p-1}+\kappa-1}}{\hat{v}^3_{yy}(f(x,z),z)}\beta^{-(\kappa-1)}z\left(\kappa+\frac{p}{1-p}\right),
\end{align}
which implies the desired item (ii).

 Next, we proceed to show item (iii). By applying Corollary \ref{coro:optimal-control}, it holds that
\begin{align*}
\theta^*(x,z)&= -\frac{\mu}{\sigma^2}\frac{v_x}{v_{xx}}(x,z)-\frac{\sigma_Z}{\sigma}\left(\frac{zv_{xz}}{v_{xx}}(x,z)-z\right)\nonumber\\
&=\frac{\mu}{\sigma^2}f(x,z)\hat{v}_{yy}(f(x,z),z)-\frac{\sigma_Z}{\sigma}(z\hat{v}_{yz}(f(x,z),z)-z).
\end{align*}
This yields that
\begin{align}\label{eq:theta-de1}
&\frac{\partial \theta^*(x,z)}{\partial x}=f_x(x,z)\left(\frac{\mu}{\sigma^2}\left(\hat{v}_{yy}(f(x,z),z)+f(x,z)\hat{v}_{yyy}(f(x,z),z)\right)
-\frac{\sigma_Z}{\sigma}z\hat{v}_{yyz}(f(x,z),z)\right).
\end{align}
In lieu of \eqref{sol:hat-v}, we can see that
\begin{align}
(\hat{v}_{yy}+y\hat{v}_{yyy})(y,z)&=-\frac{1}{\rho(1-p)-\alpha p}y^{\frac{2-p}{p-1}}-z(1-\kappa)^2\beta^{-(\kappa-1)}y^{\kappa-2},\label{eq:vyy-yyy}\\
\hat{v}_{yyz}(y,z)&=(1-\kappa)\beta^{-(\kappa-1)}y^{\kappa-2}.\label{eq:vyyz}
\end{align}
Then, from \eqref{eq:theta-de1}, \eqref{eq:vyy-yyy} and \eqref{eq:vyyz}, we deduce that
\begin{align}\label{eq:theta-de-x}
\frac{\partial \theta^*(x,z)}{\partial x}&=\frac{\mu}{\sigma^2}\frac{f(x,z)^{\kappa-2}}{\hat{v}_{yy}(f(x,z),z)}\nonumber\\
&\quad\times\left[\frac{1}{\rho(1-p)-\alpha p}f(x,z)^{-(\frac{p}{1-p}+\kappa)}+(1-\kappa)\beta^{-(\kappa-1)}z\left(1-\kappa+\frac{\sigma_Z\sigma}{\mu}\right)\right].
\end{align}
Note that $\frac{\partial \theta^*(x,z)}{\partial x}>0$ for all $x\in\R_+$. Then, we obtain that $x\to \theta^*(x,z)$ is increasing.

Finally, it remains to prove item (iv). In fact, it follows from \eqref{eq:theta-de-x} that
\begin{align}\label{eq:theta-de2}
&\frac{\partial^2 \theta^*(x,z)}{\partial x^2}=-\frac{\mu}{\sigma^2}\frac{1}{\hat{v}^2_{yy}(f(x,z),z)}\nonumber\\
&~\times\left[\frac{1}{\rho(1-p)-\alpha p}\frac{2-p}{p-1}f(x,z)^{\frac{3-2p}{p-1}}+f(x,z)^{\kappa-3}(1-\kappa)(\kappa-2)\beta^{-(\kappa-1)}
z\left(1-\kappa+\frac{\sigma_Z\sigma}{\mu}\right)\right]\nonumber\\
&~+\frac{\mu}{\sigma^2}\frac{\hat{v}_{yyy}(f(x,z),z)}{\hat{v}^3_{yy}(f(x,z),z)}\left[\frac{1}{\rho(1-p)-\alpha p}f(x,z)^{\frac{2-p}{p-1}}+f(x,z)^{\kappa-2}(1-\kappa)\beta^{-(\kappa-1)}z\left(1-\kappa+\frac{\sigma_Z\sigma}{\mu}\right)\right]\nonumber\\
&~=\frac{\mu}{\sigma^2}\frac{1}{\hat{v}^3_{yy}(f(x,z),z)}\left[\frac{1}{\rho(1-p)-\alpha p}f(x,z)^{\frac{3-2p}{p-1}}\left(\frac{2-p}{1-p}\hat{v}_{yy}(f(x,z),z)+f(x,z)\hat{v}_{yyy}(f(x,z),z)\right)\right.\nonumber\\
&~+\left.f(x,z)^{\kappa-3}(1-\kappa)\beta^{-(\kappa-1)}z\left(1-\kappa+\frac{\sigma_Z\sigma}{\mu}\right)
\left((2-\kappa)\hat{v}_{yy}(f(x,z),z)+f(x,z)\hat{v}_{yyy}(f(x,z),z)\right)\right]\nonumber\\
&~=\frac{\mu}{\sigma^2}\frac{\beta^{-(\kappa-1)}(1-\kappa)z}{\hat{v}_{yy}^3(f(x,z),z)} \frac{1-p}{\rho(1-p)-\alpha p}f(x,z)^{-\frac{5-4p}{p-1}+\kappa}\left(\frac{p}{1-p}+\kappa\right)\left(\frac{p}{1-p}+\kappa-\frac{\sigma\sigma_Z}{\mu}\right).
\end{align}
It also holds that
\begin{align}\label{p-p1p2}
\left(\frac{p}{1-p}+\kappa\right)\left(\frac{p}{1-p}+\kappa-\frac{\sigma\sigma_Z}{\mu}\right)
\begin{cases}
\displaystyle >0, &\text{if}~p\in(\infty,p_1)\cup(p_2,1);\\
\displaystyle =0, &\text{if}~p=p_1~\text{or}~p=p_2;\\
\displaystyle <0, &\text{if}~p\in(p_1,p_2).
\end{cases}
\end{align}
The desired result (iv) then follows from \eqref{eq:theta-de2} and \eqref{p-p1p2}. 
\end{proof}

\noindent{\bf Acknowledgement.}\quad The authors are grateful to two anonymous referees and the Associate Editor for their helpful comments and suggestions. L. Bo and Y. Huang are supported by National Natural Science of Foundation of China (grant 12471451), Natural Science Basic Research Program of Shaanxi (grant 2023-JC-JQ-05) and the Shaanxi Fundamental Science Research Project for Mathematics and Physics (grant 23JSZ010). X. Yu is supported by the Hong Kong RGC General Research Fund (GRF) under grants 15304122 and 15306523 and by the Research Centre for Quantitative Finance at the Hong Kong Polytechnic University under grant no. P0042708.

\ \\
\noindent{\bf Declarations:}\quad The authors have no relevant financial or non-financial conflict of interests to declare.              
\ \\ \\
\noindent{\bf Data availability statement:}\quad Data sharing not applicable to this article as no datasets were generated or analyzed during the current study.


\begin{thebibliography}{}
{\small

\bibitem[{Achury et al.(2012)}]{Achury12} Achury C, Hubar S, Koulovatianos C (2012) Saving rates and portfolio choice with subsistence consumption. {\it Rev. Econ. Dyn.} {15}(1): 108-126.

\bibitem[{Alvarez-Pelaez and D{\'\i}az (2005)}]{Alvarez-Pelaez05} Alvarez-Pelaez MJ, D{\'\i}az A (2005) Minimum consumption and transitional dynamics in wealth distribution. {\it J. Monetary Econ.} {52}(3): 633-667.

\bibitem[{Angoshtari et al.(2019)}]{ABY} Angoshtari B, Bayraktar E, Young VR (2019) Optimal dividend distribution under drawdown and ratcheting constraints on dividend rates. {\it SIAM J. Financial Math.} {10}(2): 547-577.

\bibitem[Barles et al.(1994)]{BDR1994} Barles G, Daher C, Romano M (1994) Optimal control on the $L^{\infty}$ norm of a diffusion process. {\it SIAM J. Contr. Optim.} {32}(3): 612-634.

\bibitem[{Barron and Ishii(1989)}]{BaIshii89} Barron EN, Ishii H (1989) The Bellman equation for minimizing the maximum cost. {\it Nonlinear Anal. TMA} {13}(9): 1067-1090.

\bibitem[{Behr et al.(2013)}]{Behr13} Behr P, Guettler A, Miebs F (2013) On portfolio optimization: Imposing the right constraints. {\it J. Bank. Financ.} {37}(4): 1232-1242.

\bibitem[{Best and Hlouskova(2000)}]{Best2000}Best MJ, Hlouskova J (2000) The efficient frontier for bounded assets. {\it Math. Method Oper. Res.} 52: 195-212.



\bibitem[Bo et al.(2021)]{BoLiaoYu21}  Bo L, Liao H, Yu X (2021) Optimal tracking portfolio with a ratcheting capital benchmark. {\it SIAM J. Contr. Optim.} {59}(3): 2346-2380.

\bibitem[Bo et al.(2024)]{BHY24} Bo L, Huang Y, Yu X (2024) Stochastic control problems with state-reflections arising from relaxed benchmark tracking. {\it Math. Oper. Res.} available at \url{https://doi.org/10.1287/moor.2023.0265}

\bibitem[{Bokanowski et al.(2015)}]{BPZ15} Bokanowski O, Picarelli A, Zidani H (2015) Dynamic programming and error estimates for stochastic control problems with maximum cost. {\it Appl. Math. Optim.} 71: 125-163.

\bibitem[{Borell(2007)}]{Broell07} Borell C (2007) Monotonicity properties of optimal investment strategies for log-Brownian asset prices.  {\it Math. Finance} {17}(1): 143-153.

\bibitem[{Boyle and Tian(2007)}]{BoyTian07} Boyle P, Tian W (2007) Portfolio management with constraints.  {\it Math. Finance} {17}(3): 319-343.

\bibitem[Brigo et al.(2009)]{Brigo2009} Brigo D, Dalessandro A, Neugebauer M, Triki F (2009) A stochastic processes toolkit for risk management: Geometric Brownian motion, jumps, GARCH and variance gamma models. {\it J. Risk Manag.  Finan. Institutions} {2}(4): 365-393.


\bibitem[Browne(1999a)]{Browne9} Browne S (1999a) Reaching goals by a deadline: Digital options and continuous-time active portfolio management. {\it Adv. Appl. Probab.} {31}: 551-577.

\bibitem[{Browne(1999b)}]{Browne99} Browne S (1999b) Beating a moving target: Optimal portfolio strategies for outperforming a stochastic benchmark. {\it Finan. Stoch.} {3}: 275-294.


\bibitem[{Browne(2000)}]{Browne00} Browne S (2000) Risk-constrained dynamic active portfolio management {\it Manag. Sci.} {46}(9): 1188-1199.



\bibitem[{Carroll and Kimball(1996)}]{CK1996} Carroll CD, Kimball MS (1996) On the concavity of the consumption function. {\it Econometrica} {64}(4): 981-992.

\bibitem[Chacko and Viceira(2005)]{Chacko2005} Chacko G, Viceira LM (2005) Dynamic consumption and portfolio choice with stochastic volatility in incomplete markets. {\it Rev.  Finan. Stud.} {18}(4): 1369-1402.

\bibitem[{Chatterjee and Ravikumar(1999)}]{Chatterjee1999}Chatterjee S and Ravikumar B (1999) Minimum consumption requirements: Theoretical and quantitative implications for growth and distribution. {\it Macroecon. Dyn.} {3}(4): 482-505.

\bibitem[{Chen and Vellekoop(2017)}]{ChenV} Chen A, Vellekoop M (2017) Optimal investment and consumption when allowing terminal debt. {\it Euro. J. Oper. Res.} 258: 385-397.


\bibitem[{Chow et al.(2020)}]{CYZ20} Chow Y, Yu X, Zhou C (2020) On dynamic programming principle for stochastic control under expectation constraints. {\it J. Optim. Theor. Appl.} {185}(3): 803-818.

\bibitem[{Cox and Huang(1989)}]{Cox1989} Cox JC, Huang CF (1989) Optimal consumption and portfolio policies when asset prices follow a diffusion process. {\it J. Econ. Theor.} {49}(1): 33-83.

\bibitem[{Deng et al.(2022)}]{DLPY22} Deng S, Li X, Pham H, Yu X (2022) Optimal consumption with reference to past spending maximum. {\it Finan. Stoch.} {26}: 217-266.

\bibitem[{Di Giacinto et al.(2011)}]{Giacinto11} Di Giacinto M, Federico S, Gozzi F (2011) Pension funds with a minimum guarantee: a stochastic control approach. {\it Finan. Stoch.} {15}: 297-342.


\bibitem[{Eisenberg and Schmidli(2009)}]{Eisenberg09} Eisenberg J, Schmidli H (2009) Optimal control of capital injections by reinsurance in a diffusion approximation. {\it Bl\"{a}tter der DGVFM} {30}(1): 1-13.

\bibitem[{Eisenberg and Schmidli(2011)}]{Eisenberg11}  Eisenberg J, Schmidli H (2011) Optimal control of capital injections by reinsurance with a constant rate of interest. {\it J.  Appl. Probab.} {48}(3): 733-748.


\bibitem[{El Karoui et al.(2005)}]{Karoui05} El Karoui N, Jeanblanc M, Lacoste V (2005) Optimal portfolio management with American capital guarantee. {\it J. Econ. Dyn. Contr.} {29}: 449-468.

\bibitem[{El Karoui and Meziou(2006)}]{KM06} El Karoui N, Meziou A (2006) Constrained optimization with respect to stochastic dominance: application to portfolio insurance. {\it Math. Finance} {16}(1): 103-117.

\bibitem[{Farkas et al.(2012)}]{Farkas21} Farkas W, Mathys L,  Vasiljevi{\'c}, N (2021) Intra‐Horizon expected shortfall and risk structure in models with jumps. {\it Math. Finance} {31}(2): 772-823.

\bibitem[{Ferrari and Meziou(2019)}]{FS19} Ferrari G, Schuhmann P (2019) An optimal dividend problem with capital injections over a finite horizon. {\it SIAM J. Contr. Optim.} {57}(4): 2686-2719.


\bibitem[{Gaivoronski et al.(2005)}]{Gaivoronski05} Gaivoronski A, Krylov S, Wijst N (2005) Optimal portfolio selection and dynamic benchmark tracking. {\it Euro. J. Oper. Res.} {163}: 115-131.

\bibitem[{Guasoni et al.(2011)}]{GHW11} Guasoni P, Huberman G, Wang ZY (2011) Performance maximization of actively managed funds. {\it J. Finan. Econ.} {101}: 574-595.

\bibitem[{Jensen and Miller(2008)}]{Jensen08}Jensen RT, Miller NH (2008) Giffen behavior and subsistence consumption. {\it Amer. Econ. Rev.} 98(4): 1553-1577.

\bibitem[{Karatzas et al.(1987)}]{Karatzas1987} Karatzas I, Lehoczky JP, Shreve SE (1987) Optimal portfolio and consumption decisions for a ``small investor'' on a finite horizon. {\it SIAM J. Contr. Optim.} {25}(6): 1557-1586.

\bibitem[{Kim and Shin(2018)}]{Kim18} Kim JY, Shin YH (2018) Optimal consumption and portfolio selection with negative wealth constraints, subsistence consumption constraints, and CARA utility. {\it J. Korean Stat. Soc.} 47: 509-519.

\bibitem[{Kr\"{o}ner et al.(2018)}]{Kroner18} Kr\"{o}ner A, Picarelli A, Zidani H (2018) Infinite horizon stochastic optimal control problems with running maximum cost. {\it SIAM J. Contr. Optim.} {56}(5): 3296-3319.


\bibitem[{Li et al.(2023)}]{LYZ} Li X, Yu X, Zhang Q (2023) Optimal consumption and life insurance under shortfall aversion and a drawdown constraint. {\it Insur. Math. Econ.} {108}: 25-45.

\bibitem[{Liu and Li(2023)}]{Liu23} Liu H, Li L (2023) On the concavity of consumption function under habit formation. {\it J. Math. Econ.} {106}: 102829.


\bibitem[{Lokka and Zervos(2008)}]{Lokka08} Lokka A, Zervos M (2008) Optimal dividend and issuance of equity policies in the presence of proportional costs. {\it Insur. Math. Econ.} {42}(3): 954-961.


\bibitem[Merton(1969)]{Merton69} Merton RC (1969) Lifetime portfolio selection under uncertainty: the continuous time case. {\it Rev. Econ. Stats.} {51}(3): 247-257.

\bibitem[Merton(1971)]{Merton1971} Merton RC (1971) Optimum consumption and portfolio rules in a continuous-time model. {\it J. Econ. Theor.} {3}(4): 373-413.

\bibitem[{Ni et al.(2022)}]{NLFC} Ni C, Li Y, Forsyth P, Carroll R (2022) Optimal asset allocation for outperforming a stochastic benchmark target. {\it Quant. Finance} {22}(9): 1595-1626.

\bibitem[{Pham(2002)}]{Pham02} Pham H (2002) Minimizing shortfall risk and applications to finance and insurance problems. {\it Ann. Appl. Probab.} {12}(1): 143-172.

\bibitem[{Sekine(2012)}]{Sekine12} Sekine J (2012) Long-term optimal portfolios with floor. {\it Finan. Stoch.} {16}: 369-401.

\bibitem[{Sethi et al.(1992)}]{Sethi1992} Sethi SP, Taksar MI, Presman EL (1992) Explicit solution of a general consumption/portfolio problem with subsistence consumption and bankruptcy. {\it J. Econ. Dyn. Contr.} 16(3-4): 747-768.

\bibitem[{Shin and Lim (2011)}]{Shin2011}Shin YH, Lim BH (2011) Comparison of optimal portfolios with and without subsistence consumption constraints. {\it Nonlinear Anal. TMA} {74}(1): 50-58.

\bibitem[{Shirakawa(1994)}]{Shirakawa1994} Shirakawa H (1994) Optimal consumption and portfolio selection with incomplete markets and upper and lower bound constraints. {\it Math. Finance} {4}(1): 1-24.

\bibitem[{Strub and Baumann(2018)}]{Strub18} Strub O, Baumann P (2018) Optimal construction and rebalancing of index-tracking portfolios. {\it Euro. J. Oper. Res.} {264}: 370-387.

\bibitem[{Wang and Wang(2021)}]{Wang2021} Wang L,  Wang S (2021) Unusual investor behavior under tacit and endogenous market signals. {\it Inter. Rev.  Econ.  Finan.} {73}: 76-97.

\bibitem[{Weerasinghe and Zhu(2016)}]{Weerasinghe16} Weerasinghe A, Zhu C (2016) Optimal inventory control with path-dependent cost criteria. {\it Stoch. Process. Appl.} {126}: 1585-1621.

\bibitem[{Xia(2011)}]{Xia11} Xia J (2011) Risk aversion and portfolio selection in a continuous-time model. {\it SIAM J. Contr. Optim.} {49}(5): 1916-1937.

\bibitem[{Yao et al.(2006)}]{YaoZZ06} Yao D, Zhang S, Zhou XY (2006) Tracking a financial benchmark using a few assets. {\it Oper. Res.} {54}(2): 232-246.

\bibitem[{Zimmerman and Carter(2003)}]{Zimmerman03} Zimmerman FJ, Carter MR (2003) Asset smoothing, consumption smoothing and the reproduction of inequality under risk and subsistence constraints. {\it J. Dev. Econ.} {71}(2): 233-260.
}
\end{thebibliography}
\end{document}